\documentclass[reqno,tbtags,sumlimits,intlimits,a4paper,oneside,11pt]{amsart}
\def\x{3.73}%{3.58}%{3.10}
\usepackage[lmargin=\x cm,rmargin=\x cm,tmargin=\x cm,bmargin=\x cm,bindingoffset=0cm]{geometry}

\newif\ifExceptional
\Exceptionaltrue

\usepackage{enumitem}
\setlist[1]{itemsep=5pt} 
\usepackage[table]{xcolor}
\usepackage{amsthm}
\usepackage{amsmath}
\usepackage{rotating}
\usepackage{adjustbox}
\usepackage{amssymb}
\usepackage{amsfonts}
\usepackage{mathrsfs}
\usepackage{tikz}
\usetikzlibrary{decorations.markings}
\usetikzlibrary{cd}
\usepackage{url}
\usepackage[backend=biber, 
    sorting=nyt, 
    date=year, 
    hyperref=false,
    arxiv=abs,  
    url=false, 
    doi=false,
    eprint=false, 
    giveninits,
    ]{biblatex}
\usepackage{mathscinet}
\usepackage{pdflscape}
\usepackage{afterpage}
\usepackage{capt-of}
\appto{\bibsetup}{\sloppy}

\newcommand{\mc}[1]{\mathcal{#1}}
\newcommand{\mb}[1]{\mathbb{#1}}
\newcommand{\mf}[1]{\mathfrak{#1}}
\newcommand{\ba}{\bar{a}}
\newcommand{\backspace}{\!\!\!}

\newcommand{\rd}{\textrm{d}}

\newcommand{\modaut}[2][(X,Y)]{\Phi_{#2}#1}

\newcommand{\poles}{S}

\def\beq#1#2\eeq{%
        \begin{equation}%
        \label{#1}%
            #2%
        \end{equation}%
    }
\newcommand{\SLNZ}[1][2]{\mathrm{SL}({#1},\mb{Z})}
\newcommand{\SLNC}[1][N+1]{\mathrm{SL}({#1},\mb{C})}
\newcommand{\slnc}[1][N+1]{\mf{sl}({#1},\mb{C})}

\newcommand{\rg}{\Gamma}%Reduction group
\newcommand{\brg}{B\Gamma}
\newcommand{\abg}{\brg/[\brg,\brg]}
\newcommand{\ag}{\rg/[\rg,\rg]}
\newcommand{\hombgc}{\Hom(\brg,\mb{C}^\ast)}

\newcommand{\roots}{\Phi}

\newcommand{\triv}{\epsilon}

\newcommand{\Id}{\mathrm{Id}}
\newcommand{\id}{\mathrm{id}}

\newcommand{\GL}{\mathrm{GL}}
\newcommand{\SL}{\mathrm{SL}}
\newcommand{\SO}{\mathrm{SO}}
\newcommand{\SU}{\mathrm{SU}}

\newcommand{\PSL}{\mathrm{PSL}}

\newcommand{\Hom}{\mathrm{Hom}}

\newcommand{\Aut}[1]{\mathrm{Aut}\!\left(#1 \right)}

\newcommand{\Der}[1]{\mathrm{Der}\!\left(#1 \right)}
\newcommand{\Ad}{\mathrm{Ad}}
\newcommand{\ad}{\mathrm{ad}}

\newcommand{\lcm}{\mathrm{\,lcm}}

\newcommand{\diag}{\mathrm{diag}}
\newcommand{\rank}{\mathrm{rank}\,}
\newcommand{\res}{\mathrm{res}\,}

\newcommand{\finitegroupfont}{\mathsf}
\newcommand{\zn}[1]{\mb{Z}/#1\mb{Z}}
\newcommand{\cg}[1]{\finitegroupfont{C}_{#1}}%for the cyclic group there is also a \zn command created
\newcommand{\dg}[1]{\finitegroupfont{D}_{#1}}
\newcommand{\tg}{\finitegroupfont{T}}
\newcommand{\og}{\finitegroupfont{O}}
\newcommand{\yg}{\finitegroupfont{Y}}

\newcommand{\bd}[1]{\finitegroupfont{Dic}_{#1}}
\newcommand{\bt}{B\finitegroupfont{T}}
\newcommand{\bo}{B\finitegroupfont{O}}
\newcommand{\by}{B\finitegroupfont{Y}}

\newcommand{\ch}{{\mathrm{ch}}}
\renewcommand{\deg}{{\mathrm{deg}}}
\newcommand{\ex}{\mathsf{n}}

\newtheorem{Theorem}{Theorem}[section]
\newtheorem{Proposition}[Theorem]{Proposition}
\newtheorem{Lemma}[Theorem]{Lemma}
\newtheorem{Corollary}[Theorem]{Corollary}

\newtheorem{Definition}[Theorem]{Definition}
\newtheorem{Example}[Theorem]{Example}
\newtheorem{Notation}[Theorem]{Notation}
\newtheorem{Remark}[Theorem]{Remark}

\newcommand\colorchain{gray}
\newcommand\colorp{red}
\newcommand\colorm{blue}
\newcommand\colora{red}
\newcommand\colorb{olive}
\newcommand\colorc{blue}
\tikzstyle{root}=[scale=0.6, shape=circle]
\tikzstyle{bigroot}=[scale=0.8, shape=circle]
\tikzstyle{2ndroot}=[scale=0.4, shape=circle]
\tikzstyle{cross}=[scale=0.15, shape=circle, fill]
\tikzstyle{chain}=[color=\colorchain,->, shorten >=8pt, shorten <=8pt, >=stealth]
\tikzstyle{schain}=[color=\colorchain,-, shorten >=4pt, shorten <=4pt]
\tikzstyle{cochainp}=[thick, color=\colorp,->, shorten >=8pt,shorten <=8pt, >=stealth]
\tikzstyle{cochainm}=[thick, color=\colorm,->, shorten >=8pt,shorten <=8pt, >=stealth]
\tikzstyle{scochain}=[-, shorten >=3pt, shorten <=3pt]
\tikzstyle{scochainp}=[color=\colorp,-, shorten >=3pt, shorten <=3pt]
\tikzstyle{scochainm}=[color=\colorm, densely dashed,-, shorten >=3pt, shorten <=3pt]
\tikzstyle{scochaina}=[color=\colora, densely dashdotted, -, shorten >=3pt, shorten <=3pt]
\tikzstyle{scochainb}=[color=\colorb, -, shorten >=3pt, shorten <=3pt]
\tikzstyle{scochainc}=[color=\colorc, dashed, -, shorten >=3pt, shorten <=3pt]
\newcommand{\scaleAA}{0.32}
\newcommand{\scaleBB}{1.15}
\newcommand{\scaleGG}{1.10}%{0.89}
\newcommand{\dangle}{20}
\newcommand{\bend}{15}

\newcommand{\dynkintablescale}{0.9}
\newcommand{\dynkinfont}{\footnotesize}
\tikzstyle{dynkinnode}=[draw, color=black, shape=circle, minimum size=3.5 pt, inner sep=0]
\tikzstyle{ldynkinnode}=[draw, color=black, shape=circle, minimum size=3.5 pt, inner sep=0]
\tikzstyle{sdynkinnode}=[draw, color=black, shape=circle, minimum size=3.5 pt, inner sep=0]
\tikzstyle{brace}=[decorate, decoration={brace, amplitude=5pt}]
\tikzstyle{mbrace}=[decorate, decoration={brace, amplitude=5pt, mirror}]

\title[Chevalley normal forms for automorphic~Lie~algebras]{A uniform construction of Chevalley normal forms for automorphic~Lie~algebras on the Riemann~sphere}

\author{Vincent Knibbeler}
\address{Maxwell Institute for Mathematical Sciences, The Bayes Centre, Edinburgh EH8 9BT, UK}
\address{Department of Mathematics, Heriot-Watt University, Edinburgh EH14 4AS, UK}
\email{V.Knibbeler@gmail.com}

\begin{document}

\begin{abstract}
For a finite subgroup $\brg$ of $SU(2)$ and one of its ground forms $P\in\mb{C}[X,Y]$, we show that the space of invariants $\mb{C}[X,Y,P^{-1}]^{\brg}_k$ of degree $k\in2\mb{Z}$ is a cyclic module over the algebra of invariants of degree zero. We find a generator for this module, uniformly for all finite subgroups of  $SU(2)$. 
Then we construct a uniform intertwiner sending the scalar invariants to vector-valued invariants. 
With these tools we construct all automorphic Lie algebras $\mf{g}[X,Y,P^{-1}]^{\brg}_0$ defined by a homomorphism from the symmetry group $\brg$ into the automorphism group of a finite dimensional Lie algebra $\mf g$, which factors through $SU(2)$. When the Lie algebra $\mf g$ is simple, we present a set of generators for the automorphic Lie algebra which is analogous to the Chevalley basis for $\mf g$. 
Previous observations of isomorphisms between automorphic Lie algebras with distinct symmetry groups $\brg$ are explained in terms of the Coxeter number of $\mf g$ and the orders appearing in $\brg$. Finally, we compute the structure constants for automorphic Lie algebras of all exceptional Lie types.
\end{abstract}

\maketitle

\thispagestyle{empty}

\vspace{3mm}

\vspace{3mm}

\section{Introduction}

Automorphic Lie algebras are originally defined as the Lie algebras of meromorphic maps from the Riemann sphere $\overline{\mb C}$ to a finite dimensional Lie algebra $\mf{g}$ over the complex numbers. The maps are required to be equivariant with respect to a finite subgroup $\rg$ of $\Aut{\overline{\mb C}}$ acting on $\mf g$ by Lie algebra automorphisms, and poles are confined to an orbit of this group. The introduction of automorphic Lie algebras occurred in the PhD thesis of Lombardo in 2004 \cite{lombardo2004reductions}, and the related work of Lombardo and Mikhailov \cite{lombardo2004reductions, lombardo2005reduction}. They are closely related to the reduction group appearing in earlier work of Mikhailov \cite{mikhailov1979integrability, mikhailov1980reduction, mikhailov1981the} and emerged in the area of integrable partial differential equations.

Special cases of automorphic Lie algebras appeared much earlier still. They include the twisted loop algebras, which extend to the affine Kac-Moody algebras \cite{kac1990infinite}, but also the Onsager algebra from 1944 \cite{onsager1944crystal}. The fact that the Onsager algebra fits in the framework of automorphic Lie algebras is only clear thanks to the concrete basis constructed by Roan \cite{roan1991onsager}.

After the first publications of Lombardo and Mikhailov on automorphic Lie algebras, a project to compute bases for these algebras in normal form, and the corresponding structure constants, was initiated by Lombardo and Sanders. The first milestone of this project, a classification for the case $\mf g=\slnc[2]$ ($2\times2$ matrices with trace zero), was published in 2010 \cite{lombardo2010on}. Here it was noticed that the automorphic Lie algebras had an element with eigenvalues $1$ and $-1$ which allowed the construction of a normal form. This normal form was generalised and called the Chevalley normal form in \cite{knibbeler2017higher}.

In 2010, a second important document in this field emerged: Bury finished his PhD thesis \cite{bury2010automorphic}, supervised by Mikhailov. There is a large overlap in computed automorphic Lie algebras in this thesis, and in \cite{lombardo2010on}. However, Bury did not determine a normal form for the automorphic Lie algebras. Instead, the normal form of the target Lie algebra $\mf g $ was averaged, and the resulting basis for the automorphic Lie algebra was used directly to study integrable partial differential equations.

An independent development by Neher, Savage and Senesi resulted in the introduction of equivariant map algebras in 2012 \cite{neher2012irreducible}. Automorphic Lie algebras are examples of equivariant map algebras, the latter being defined in much more general algebraic terms. Roughly, one replaces the Riemann sphere in the original definition of an automorphic Lie algebra by a scheme, to obtain the definition of an equivariant map algebra. 

We notice two interesting distinctions between the development of automorphic Lie algebras and equivariant map algebras, which have such a similar definition. Firstly, the motivation to introduce equivariant map algebras was to unify a large variety of Lie algebra families which proved to be of importance in the past, and there was no specific mention of integrable systems (the primary motivation of automorphic Lie algebras). Secondly, where initial research efforts on automorphic Lie algebras went into the computation of concrete bases, the first paper on equivariant map algebras works with the algebras on an abstract level and studies their finite dimensional representations from the start, obtaining the classification of irreducible representations, in much the same way as Roan did for the Onsager algebra in \cite{roan1991onsager}.

Lombardo and Sanders meanwhile continued the project to compute bases and structure constants of automorphic Lie algebras and published the classification of automorphic Lie algebras based on $\mf g=\slnc[N]$ together with the author in 2017 \cite{knibbeler2017higher}. This classification is restricted to representations $\rg\to\Aut{\slnc[N]}$ without trivial components. That is, the only constant invariant is the zero matrix: $\slnc[N]^\rg=\{0\}$ (which implies $N\le 6$). To this day, \cite{knibbeler2017higher} is the state of the art of the classification project for automorphic Lie algebras on the Riemann sphere. Published, that is. Sanders has computed cases based on $\mf{so}(N,\mb C)$ and $\mf{sp}(2N,\mb C)$, but we were unable to complete the classification as was done for $\slnc[N]$ in \cite{knibbeler2017higher}. The first automorphic Lie algebras of type $G_2$ were computed by Casper Oelen, but also this work has remained unpublished.
 
Automorphic Lie algebras on complex tori have a history going back to 1989 at least, when Reiman and Semenov-Tyan-Shanskii \cite{reiman1989lie} generalised the elliptic Lie algebra of Holod \cite{holod1987hidden}. In 1993, Uglov gave an algebraic description of the automorphic Lie algebras on complex tori with base Lie algebra $\mf g=\slnc[2]$ and symmetry group $\zn{2}\times\zn{2}$ \cite{uglov1994lie}, and Skrypnyk generalised this to base Lie algebra $\mf g=\slnc[N]$ and symmetry group $\zn{N}\times\zn{N}$ with an elegant construction \cite{skrypnyk2012quasi}. Skrypnyk also remarked that the commutation relations in the bases provided in \cite{reiman1989lie, uglov1994lie} are hardly manageable, and proposed better bases. The Lie algebras are still used in contemporary research on integrable systems \cite{skrypnyk2025reduction}. Up until the very recent paper of Lombardo and Oelen \cite{lombardo2024normal} nobody appeared to have seen that all automorphic Lie algebras on tori mentioned in this paragraph so far are isomorphic to a Lie algebra of the form $\mf{sl}(N,R)$ where $R$ is a ring of elliptic functions invariant under the symmetry group. This shows that the commutation relations are as manageable as those of the simple Lie algebra $\mf{sl}(N,\mb{C})$. The situation moves further away from the simple case when the action of the symmetry group on the Riemann surface involves nontrivial stabilizers, and the first classification of such automorphic Lie algebras on tori was achieved by Oelen with his PhD thesis \cite{oelen2022automorphic} in 2022, followed by the corresponding paper \cite{knibbeler2024classification}. 
Automorphic Lie algebras will never be of the form $\mf{g}(N,R)$ (also known as current algebras) when the group action on the Riemann surface has a nontrivial stabiliser. This follows from their local expansions determined by Duffield, Lombardo and the author \cite{duffield2023wild}.

All automorphic Lie algebras discussed to this point are examples of the huge class of equivariant map algebras. Lombardo, Veselov and the author also studied matrix-valued maps on the hyperbolic plane, equivariant with respect to the modular group $\SLNZ[2]$ and its subgroups of finite index \cite{10.1093/imrn/rnab376}. The terminology of automorphic Lie algebras was used once more, but these automorphic Lie algebras are no longer examples of equivariant map algebras.

For the reader with an interest in central extensions of automorphic Lie algebras we refer to Bremner \cite{bremner1994generalized, bremner1995four}, and for derivations we refer to Cox, Guo, Lu and Zhao \cite{cox2014n, cox2016module}, who studied these topics for the $M$-point loop algebras $\mf g \otimes_{\mb C} \mb C[t, (t-a_1)^{-1}, \ldots, (t-a_{M-1})^{-1}]$, which are automorphic Lie algebras with trivial symmetry group. Cox et al.~also describe the actions of Platonic solids on the space of derivations in detail. This work is taken further by Dos Santos \cite{dossantos2022on}. Skrypnyk \cite{skrypnyk2012quasi} and Bury and Mikhailov \cite{bury2021automorphic} provided central extensions and derivations with the desired symmetry properties for the automorphic Lie algebras described in the respective papers, allowing for interesting constructions to study integrable systems.

In this paper we proceed with the classification project for automorphic Lie algebras on the Riemann sphere. Global bases of automorphic Lie algebras with $\mf g=\slnc[N]$ have been published in \cite{knibbeler2017higher}. For all other Lie algebras $\mf g$, we have so far only understood the local expansions of the associated automorphic Lie algebras \cite{duffield2023wild}. We will once again impose a restriction on the representations $\rg\to\Aut{\mf g}$ used for the automorphic Lie algebra: we require it to factor through $\Aut{\overline{\mb{C}}}$, and call such representations factorisable. This is independent of regularity $\mf g^\rg=\{0\}$. That is, some factorisable representations are regular, but most of them are not. Moreover, many inner regular actions are factorisable, but not all of them (cf.~Proposition \ref{prop:one inner regular action not factorisable}).

For automorphic Lie algebras based on a simple Lie algebra $\mf g$ and a factorisable action, we obtain the global structure in Theorem \ref{thm:aliam}. This result solves some of the main open problems in the field, that have kept researchers busy since the appearance of \cite{lombardo2010on} in 2010: existence of Cartan subalgebras (CSA), the $2$-cocycle problem and the isomorphism question.
The CSA-existence problem is positively solved for these automorphic Lie algebras. With the terminology of \cite{knibbeler2019hereditary} we would say the automorphic Lie algebras algebras are hereditary.
The $2$-cocycles defining the Lie algebra, which were only partly described for hereditary automorphic Lie algebras, are now fully described. And with that we answer one direction of the isomorphism question: since the $2$-cocycle (\ref{eq:2cocycles}) only depends on the orders of the elliptic points, we have isomorphisms between automorphic Lie algebras with distinct reduction group when the orders of elliptic points in the holomorphic domain of the automorphic Lie algebra coincide.
It is still an open problem whether the reverse holds, claiming the non-existence of isomorphisms between automorphic Lie algebras.

Our approach is constructive, and another result is the construction of automorphic Lie algebras based on any finite dimensional Lie algebra $\mf g$, Theorem \ref{thm:alia}, allowing applications that require the concretisation of the automorphic Lie algebra by meromorphic equivariant matrices. For applications of the Lie structure in question, one can omit the computation of the equivariant matrices and directly use the structure constants given in Theorem \ref{thm:aliam}.

The seed for this paper was planted during the fruitful research period where Lombardo, Veselov and the author worked on automorphic Lie algebras in hyperbolic geometry, leading to \cite{10.1093/imrn/rnab376}. Veselov asked the author if an analogous productive construction with a general intertwining operator is possible for automorphic Lie algebras in spherical geometry. Such a construction had never been found. Nonetheless, work of Olver and Sanders \cite{olver2000transvectants} highlights a duality between modular forms on the upper half plane and invariant forms on the projective plane, or the Riemann sphere, suggesting that a general intertwiner for automorphic Lie algebras on the sphere dual to the intertwiner in the hyperbolic case exists. In this paper we find such an intertwiner. 
The construction is simple, yet this construction is hard to guess just from looking at the resulting matrix. Much harder than in the hyperbolic case, which explains the chronological order of discovery.

In Section \ref{sec:cit} we study the classical invariant theory of finite subgroups of $\Aut{\overline{\mb C}}$. The usual setting of rational expressions of polynomials in two variables of homogeneous degree $0$ is generalised to degree $k\in\mb Z$, gearing up to study vector-valued invariants. For each degree $k$ we construct a single generator of the invariants as a module over the better known degree $0$ invariants.
In Section \ref{sec:alias} we construct an intertwiner which sends the invariants of Section \ref{sec:cit} to vector-valued invariants, and we apply it to the case where the vector-space in this construction has a Lie algebra structure, in order to arrive at automorphic Lie algebras. This is a new method to construct automorphic Lie algebra on the Riemann sphere, and is  analogous to the method introduced in \cite{10.1093/imrn/rnab376} to construct automorphic Lie algebra on the upper half plane. It allows us to compute automorphic Lie algebras which have been out of reach so far, and even construct Chevalley normal forms for them.
This is demonstrated in Section \ref{sec:examples} which is devoted to examples and applications of the results in Section \ref{sec:alias} and ends with tables of structure constants for automorphic Lie algebras of all exceptional Lie types. Further research directions are proposed in Section \ref{sec:further research directions}.

\section{Homogeneous invariants of polyhedral groups}
\label{sec:cit}
\subsection{Data of polyhedral groups}
\label{sec:polyhedral groups}

A polyhedral group is a finite subgroup of $\SO(3)$, a finite group of rotations of $3$-dimensional real space, or rotations of the $2$-sphere. For each polyhedral group $\rg$, we let $\Omega$ be an index set for the $\rg$-orbits $S_i$, $i\in\Omega$, in the sphere, whose elements have nontrivial stabiliser groups (nontrivial isotropy). That is, each element in $S_i$ is fixed by an element of $\rg$ other than the identity. Let $d_i=|S_i|$ be the size of such an exceptional orbit and $\nu_i$ the order of an associated stabiliser subgroup. In particular
\[d_i\nu_i=|\rg|,\qquad i\in\Omega.\]
We will make use of the greatest common denominator $d=\gcd(\{d_i: i\in\Omega\})$ and the least common multiple $\nu=\lcm(\{\nu_i:i\in\Omega\})$, which satisfy $d\nu=|\rg|$ (and we remark that the integer $d$ equals the order of the Schur multiplier of $\rg$).

With a classical clever counting argument one can derive the formula
\begin{equation}
\label{eq:finite subgroups of SO(3) original}
2\left(1-\frac{1}{|\rg|}\right)=\sum_{i\in\Omega}\left(1-\frac{1}{\nu_i}\right)
\end{equation}
which in turn classifies the polyhedral groups, listed in Table \ref{tab:various properties of polyhedral groups} (see \cite{toth2002finite} for more details). 
\begin{center}
\begin{table}[ht!] 
\caption{Data of polyhedral groups}
\label{tab:various properties of polyhedral groups}
\begin{center}
\begin{tabular}{lllllllll}\hline
$\rg$&$|\Omega|$&$(\nu_i)$&$\nu$&$(d_i)$&$d$&$|\rg|$&$\ag$&{$\abg$}\\
\hline
$\cg{n}$&$2$&$(n,n)$&$n$&$(1,1)$&$1$&$n$&$\cg{n}$&$\cg{2n}$\\
$\dg{n=2m-1}$&$3$&$(n,2,2)$&$2n$&$(2,n,n)$&$1$&$2n$&$\cg{2}$&$\cg{4}$\\
$\dg{n=2m}$&$3$&$(n,2,2)$&$n$&$(2,n,n)$&$2$&$2n$&$\cg{2}\times \cg{2}$&$\cg{2}\times \cg{2}$\\
$\tg$&$3$&$(3,3,2)$&$6$&$(4,4,6)$&$2$&$12$&$\cg{3}$&$\cg{3}$\\
$\og$&$3$&$(4,3,2)$&$12$&$(6,8,12)$&$2$&$24$&$\cg{2}$&$\cg{2}$\\
$\yg$&$3$&$(5,3,2)$&$30$&$(12,20,30)$&$2$&$60$&$1$&$1$\\
\hline 
\end{tabular}
\end{center}
\end{table}
\end{center}

A well known elegant construction using quaternions establishes a short exact sequence
\begin{equation*}
1\rightarrow\{\pm 1\}\to\SU(2)
\to\SO(3)\to\{1\}
\end{equation*}
of groups.
We define $\brg$ as the preimage of $\rg$ under the $2:1$ map $\SU(2)
\to\SO(3)$. 

The group $\SU(2)$ acts naturally on $\mb{C}[X,Y]$, and preserves the finite dimensional subspaces $\mb{C}[X,Y]_k$ of polynomials of homogeneous degree $k$ (forms). For all isotypical components in $\mb{C}[X,Y]$, the generating function was found by Kostant \cite{kostant1984on,kostant2006the} and Springer \cite{springer1987poincare}. Kostant describes these generating functions in terms of the Lie algebras associated to the binary polyhedral groups through the McKay correspondence (cf.~Appendix \ref{sec:kostant}).

A ground form for $\rg$ or $\brg$ is a form $P_j$ of degree $d_j$ that vanishes on $S_j$ (more precisely the image of $S_j$ under the map sending the $2$-sphere conformally to the complex projective line). The first goal of this paper is to understand the invariants in \[R=\mb{C}[X,Y,P_j^{-1}].\] 
The subspace $R_k$ of $R$ consists of those rational expressions with homogeneous degree $k\in\mb{Z}$. 
It is well known that ring of invariants $R^\rg_0$ is a polynomial ring in one variable, reflecting the fact that the quotient surface has genus zero. We aim to understand $R^\rg_k$ as $R^\rg_0$-module in this section and find the following main result.
\begin{Theorem}For $k\in2\mb{Z}$, the $R^\rg_0$-module $R^\rg_k$ is cyclic and generated by 
\[P_k=P_j^{\nu_j\ell}\prod_{i\in\Omega} P_i^{\res_i(k/2)},\quad 
\ell=\frac{k-\sum_{i\in\Omega} \res_i(k/2) d_i}{|\rg|}
\]
where the residue is taken modulo $\nu_i$. For $k\in2\mb{Z}+1$ we have $R^{\brg}_k=\{0\}$.
\end{Theorem}
This paper can also be understood as a story how (\ref{eq:finite subgroups of SO(3) original}) rewritten as
\begin{equation}
\label{eq:finite subgroups of SO(3)}
\sum_{i\in\Omega}d_i=(|\Omega|-2)|\rg|+2
\end{equation}
explains the isomorphisms between automorphic Lie algebras on the Riemann sphere.

\subsection{Characters and degrees}
\label{sec:characters and degrees}
Ground forms are necessarily relative invariant: there exists $\chi_i\in\Hom(\brg,\mb{C}^\ast)$ such that $\gamma P_i=\chi_i(\gamma)P_i$ for all $\gamma$ in $\brg$. We say that $\chi_i$ is the character of the ground form $P_i$. Henceforth, sums and products without subscript run over $i\in\Omega$. 

\begin{Lemma}
\label{lem:characters}
Let $\chi_i$ be the character of the ground form $P_i$, $i\in\Omega$.
\begin{enumerate}[label=(\roman*)]
\item \label{item:generation} The characters $\chi_i$ generate the group $\Hom(\brg,\mb{C}^\ast)$.
\item \label{item:existence character} There exists a character $\chi$ such that $\chi_i^{\nu_i}=\chi$ for all $i\in\Omega$ and $\prod\chi_i=\chi^{(|\Omega|-2)}$. 
\item \label{item:order character} The order of $\chi$ is $2/d$.
\item \label{item:hom} $\Hom(\brg,\mb{C}^\ast)/\langle\chi\rangle\cong\Hom(\rg,\mb{C}^\ast)$.  
\item \label{item:order hom} $\nu|\Hom(\rg,\mb{C}^\ast)|=\prod \nu_i$.
\end{enumerate}
\end{Lemma} 
\begin{proof}
All statements can be checked case by case, thereby finishing the proof (see Table \ref{tab:various properties of polyhedral groups} and Appendix \ref{sec:case by case analysis}). For some items we can provide a more satisfying explanation, but not for all of them.

The first statement can be understood using the fact that each irreducible representation of $\brg$ occurs in $\mb{C}[X,Y]$ (symmetric products of the natural representations, cf.~\cite[Problem 4.12.10]{etingof2011introduction}). In particular, each element of $\Hom(\brg,\mb{C}^\ast)$ occurs as a character of a form. Such a relative invariant form is a polynomial in ground forms (by induction on the zeros, cf.~\cite[Section 1.3]{toth2002finite}). Each term of this polynomial has the same character, which is a product of $\chi_i$, $i\in\Omega$.

The second statement can be understood for the noncyclic groups as follows. Any two out of three forms $P_i^{\nu_i}$ are linearly independent, but all three are dependent: \[\sum c_i P_i^{\nu_i}=0,\] as follows from the zeros and the fundamental theorem of algebra (cf.~\cite[Section 1.3]{toth2002finite}).
Suppose there exists $i,j,\gamma$ such that $\chi_i^{\nu_i}(\gamma)\ne\chi_j^{\nu_j}(\gamma)$. Then 
\begin{align*}0&=\chi_i^{\nu_i}(\gamma)\sum c_k P_k^{\nu_k}- \gamma \sum c_k P_k^{\nu_k}\\&=\chi_i^{\nu_i}(\gamma)\sum c_k P_k^{\nu_k}- \sum c_k \chi_k^{\nu_k}(\gamma)P_k^{\nu_k}
\\&=c_j( \chi_i^{\nu_i}(\gamma)-\chi_j^{\nu_j}(\gamma))P_j^{\nu_j}+c_k( \chi_i^{\nu_i}(\gamma)-\chi_k^{\nu_k}(\gamma))P_k^{\nu_k}
\end{align*} and $c_j( \chi_i^{\nu_i}(\gamma)-\chi_j^{\nu_j}(\gamma))\ne 0$ which contradicts the observation that any pair of forms $P_i^{\nu_i}$ are independent.
\end{proof}

We define two homomorphisms. Firstly the character
\begin{align*}
\ch:\prod \zn{\nu_i}\to \Hom(\brg,\mb{C}^\ast)/\langle\chi\rangle,
\end{align*}
sending $(\bar{n}_i)$ to the character $\prod \chi_i^{n_i}$, where $n_i\mod{\nu_i}=\bar{n}_i$. This map is well-defined due to Lemma \ref{lem:characters}, item \ref{item:existence character}. It is surjective by Lemma \ref{lem:characters}, item \ref{item:generation}.

Secondly, we define the degree
\newcommand{\odeg}{{\overline{\deg}}}
\begin{align*}
\odeg:\prod \zn{\nu_i}\to d\mb{Z}/|\rg|\mb{Z},
\end{align*}
sending $(\bar{n}_i)$ to the degree 
$\sum d_i n_i\mod{|\rg|}$, where 
$n_i\mod{\nu_i}=\bar{n}_i$. This is well-defined because $d_i\nu_i=|\rg|$, and it is a homomorphism. Moreover, it is surjective by Bezout's theorem in arithmetic.

Let $N_\ch$ and $N_\deg$ be the respective kernels of these maps. Their orders are
\[
|N_\ch|=\nu,\quad |N_\deg|=\nu^{-1}\prod \nu_i.
\]
Indeed, since the maps are epimorphisms, the order of the kernel is the order of the domain divided by the order of the image. Lemma \ref{lem:characters}, items \ref{item:hom} and \ref{item:order hom} then give this result. 

We will use the notation $\bar{m}^\Omega\in\prod \zn{\nu_i} $ for $(\bar{m},\bar{m})$ when $|\Omega|=2$ and $(\bar{m},\bar{m},\bar{m})$ when $|\Omega|=3$.
\begin{Lemma}
\label{lem:generated by 1}
$N_\ch$ is generated by $\bar{1}^\Omega$.
\end{Lemma}
\begin{proof}
The element $\bar{1}^\Omega$ is contained in $N_\ch$ by Lemma \ref{lem:characters} item \ref{item:existence character} and thus generates a subgroup of order $\nu$. We have seen that $N_\ch$ has the same order, which finishes the proof.
\end{proof}
By (\ref{eq:finite subgroups of SO(3)}) we have \[\odeg(\bar{1}^\Omega)=\bar{2}.\]
Thus, the restriction of the degree map 
\[N_\ch\xrightarrow[]{\odeg}\langle \bar{2}\rangle\subset d\mb{Z}/|\rg|\mb{Z}\]
is an epimorphism (and an isomorphism, except for the case $\rg=\cg{2m}$).
We construct a right inverse:
\beq{eq:definition ex}
\bar{\ex}:\langle 2 \mod \rg \rangle\to N_\ch,\quad \bar{\ex}(\bar{2})=\bar{1}^\Omega,\quad \odeg(\bar{\ex}(\bar{k}))=\bar{k},
\eeq
and extend this map to $\ex:2\mb{Z}\to\prod\{0,\ldots,\nu_i-1\}$ in the canonical way.

The map $\ex$ will allow us to construct automorphic Lie algebras without going through the considerable trouble of computing them. For this reason we also provide $\ex$ explicitly in Table \ref{tab:exponents}, Appendix \ref{sec:case by case analysis}.

\subsection{Invariant forms}
Henceforth, $\rg$ continues to be a polyhedral group, $S$ one of its orbits on the sphere, and $P$ a relative invariant form of degree $|\rg|$ vanishing on $S$. This form $P$ is a linear combination of $P_i^{\nu_i}$ and $P_j^{\nu_j}$ which is unique op to nonzero scalar multiple.

The first goal of this paper is to understand the invariants in \[R=\mb{C}[X,Y,P^{-1}]\] of fixed degree $k$ (this is slightly more general than $\mb{C}[X,Y,P_j^{-1}]$ mentioned earlier). 

The element $-\Id$ in $\brg$ ensures that there are no $\brg$-invariants of odd degree, other than zero. At even degree, the action of this element is trivial, and we have an action of $\rg\cong\brg/\pm\Id$.

At degree $0$ we have \[R^\rg_0=\mb{C}[\mb{I}]\] where $\mb{I}$ can be chosen to be any nonconstant quotient $P_i^{\nu_i}P^{-1}$.
We aim to understand $R^\rg_k$ as $R^\rg_0$-module. 
For even weight $k$, define
\begin{equation}\label{eq:P_k}
    P_k=P^{\ell}\prod_i P_i^{\ex(k)_i},\quad \ell=\frac{k-\sum_i \ex(k)_i d_i}{|\rg|},
\end{equation}
using the map $\ex$ defined by (\ref{eq:definition ex}). When $k$ is odd, put $P_k=0$.
Our main result on invariant forms can be formulated as
\begin{Theorem}
\label{thm:invariant forms}
\[R^{\brg}_k=R^\rg_0 P_k.\]
\end{Theorem}
\begin{proof}
To show that $P_k$ lives in $R_k$ we need to show that $\ell$ is an integer.  We have
\begin{align*} 
\bar{\ex}(\ell|\rg|\mod |\rg|)
&=\bar{\ex}( \overline{k-\sum_i {\ex}({k})_i d_i})
\\&=\bar{\ex}( \bar{k}-\odeg(\bar{\ex}(\bar{k})))
\\&=\bar{\ex}( \bar{k})-\bar{\ex}(\bar{k})=0
\end{align*}
and $\bar{\ex}$ is a monomorphism. Therefore, $P_k\in R_k$. 

Next we show that $P_k$ is invariant. The character of $P_k$ is in $\langle\chi\rangle$ as $\ch(\bar{\ex}(\bar{k}))=0$. For the groups $\dg{2m}, \tg, \og, \yg$ this is sufficient to see that $P_k\in R_k^\rg$ since $\chi=1$. The remaining groups demand a little more attention. 

Recall that $\ex(k)_i=k/2+p_i\nu_i$ for some integer $p_i$ such that $0\le \ex(k)_i< \nu_i$. Hence, 
\begin{align*}
k&=\ell|\rg|+\sum_i \ex(k)_i d_i
\\&=\ell|\rg|+\sum_i (k/2+p_i\nu_i) d_i
\\&=\ell|\rg|+k/2\sum_i d_i+|\rg|\sum_i p_i
\\&=\ell|\rg|+k/2((|\Omega|-2)|\rg|+2)+|\rg|\sum_i p_i,
\end{align*}
so that 
\[\ell+(|\Omega|-2)k/2+\sum_i p_i=0.\]
The character of $P_k$ is $\chi^{\ell+(|\Omega|-2)k/2+\sum_i p_i}$ by Lemma \ref{lem:characters}, item \ref{item:existence character}, and thus we see that $F_k$ is invariant.

Finally, we show that $P_k$ generates $R_k^\rg$. Let $G_k$ be in $R_k^\rg$. Then $G_k/P_k\in\mb{C}(\mb{I})$, meaning $G_k=f(\mb{I})P_k$ for some rational function $f$. It remains to be shown that the poles of $f(\mb{I})$ are in $S$. Since $G_k$ has no poles outside $S$, such poles of $f(\mb{I})$ must be cancelled by zeros of $P_k$. Moreover, the order of any pole of $f(\mb{I})$ is a multiple of the ramification number. But all zeros of $P_k$ outside $S$ have order less than the ramification number. We conclude that $f(\mb{I})\in\mb{C}[\mb{I}]=R_0^\rg$.
\end{proof}

\section{Automorphic Lie algebras from $\SLNC[2]$ embeddings}
\label{sec:alias}
Let $\mf{g}$ be a complex finite dimensional Lie algebra from here onwards.  The space $\mf{g}[X,Y,P^{-1}]$, defined as $\mf g \otimes_\mb{C} \mb C[X,Y,P^{-1}]$, has a canonical Lie bracket that is linear over $\mb C[X,Y,P^{-1}]$. A homomorphism $\rho:\brg\to\Aut{\mf{g}}$ defines an action of $\brg$ on $\mf{g}[X,Y,P^{-1}]$ by Lie algebra homomorphisms, preserving $\mf{g}[X,Y,P^{-1}]_k$, the graded summands. 
The subalgebra of invariants of degree zero is the automorphic Lie algebra 
\[\mf{A}=\mf{g}[X,Y,P^{-1}]_0^{\brg}.\]
If the homomorphism $\rho$ is the restriction of a homomorphism $\bar{\rho}:\SLNC[2]\to\Aut{\mf{g}}$ we call the associated action of $\brg$ on $\mf{g}$ factorisable. In this section we show how automorphic Lie algebras defined by factorisable actions can be constructed and understood using Theorem \ref{thm:invariant forms} and a newly obtained uniform intertwiner (\ref{eq:modaut1}), with an approach analogous to \cite{10.1093/imrn/rnab376}.

The reader might have noticed that elements of automorphic Lie algebras were defined as meromorphic maps on the Riemann sphere in the introduction of this paper, and now these elements are functions of two variables $X$ and $Y$. This is simply a different perspective of the same thing. Since the degree in $X$ and $Y$ is zero in the automorphic Lie algebra, elements are in fact functions of $X/Y$. The linear transformations on the vector $(X,Y)$ translate to fractional linear transformations on the variable $X/Y$ on the Riemann sphere (cf.~Appendix \ref{sec:case by case analysis}).

\subsection{The classification of $\bar{\rho}:\SLNC[2]\rightarrow\Aut{\mf{g}}$}
We provide a brief description of the classification of embeddings of $\SLNC[2]$ into automorphism groups of simple Lie algebras. This classification will be used in the same sense it was used in \cite{10.1093/imrn/rnab376} and we give a similar summary as can be found there.

The classification of $\SLNC[2]$-embeddings is equivalent to the classification of nilpotent orbits: orbits of nilpotent elements in $\mf{g}$ under the action of the connected component of the identity $G=\Aut{\mf{g}}^0$ of the automorphism group. The latter classification is well described and listed in \cite{collingwood1993nilpotent}. We first describe the equivalence between the two classifications and then briefly recap the theorems of Jacobson-Morozov and of Kostant that enable the classification.

To classify Lie group morphisms $\bar{\rho}:\SLNC[2]\rightarrow\Aut{\mf{g}}$ when $\mf{g}$ is semisimple, it is sufficient to classify Lie algebra morphisms $\phi:\slnc[2]\rightarrow{\mf{g}}$. Indeed, since $\SLNC[2]$ is connected and $\bar{\rho}$ continuous, one only needs to consider maps $\bar{\rho}:\SLNC[2]\rightarrow\Aut{\mf{g}}^0=G$. The tangent map is $\rd\bar{\rho}:\slnc[2]\rightarrow T_1G$, and by semisimplicity of $\mf{g}$ we have $T_1G \cong \mf{g}$ \cite[Propositions 1.120 and 1.121]{knapp2002lie}. The exponential map intertwines these maps in the sense that $\bar{\rho}(\exp(X))=\exp(\rd\bar{\rho}(X))$. Since $\exp(\slnc[2])$ generates $\SLNC[2]$, we see that $\bar{\rho}$ is determined by its derivative $\rd\bar{\rho}$. Vice versa, any homomorphism $\phi:\slnc[2]\rightarrow{\mf{g}}$ defines a homomorphism $\bar{\rho}:\SLNC[2]\rightarrow\Aut{\mf{g}}$ by $\bar{\rho}(\exp(X))=\exp(\ad(\phi(X)))$.

The $G$-orbits of $\Hom(\slnc[2],\mf{g})$ are in turn in one-to-one correspondence with nilpotent orbits, by the following theorems, as outlined in \cite[Section 3.2]{collingwood1993nilpotent}. A standard triple will be a triple of elements $(H,E,F)$ of a Lie algebra with Lie brackets $[H,E]=2E$, $[H,F]=-2F$ and $[E,F]=H$, hence spanning a subalgebra isomorphic to $\slnc[2]$. The element $E$ is the nilpositive element of the triple.
\begin{Theorem}[Jacobson-Morozov] Let $\mf{g}$ be a complex semisimple Lie algebra. If $X$ is
a nonzero nilpotent element of $\mf{g}$, then it is the nilpositive element of a standard
triple. Equivalently, for any nilpotent element $X$, there exists a homomorphism $\phi:\slnc\rightarrow\mf{g}$ such that $\phi\left(\begin{pmatrix}0&1\\0&0\end{pmatrix}\right)= X$.
\end{Theorem}
\begin{Theorem}[Kostant] Let $\mf{g}$ be a complex semisimple Lie algebra. Any two standard
triples $\{H,E, F\}$ and $\{H',E, F'\}$ with the same nilpositive element are conjugate
by an element of $\Aut{\mf{g}}^0$.
\end{Theorem}
If $\mc{N}$ is the set of $G$-orbits of nilpotent elements of $\mf{g}$, then there is a map
\[\Hom(\slnc[2],\mf{g})/G\rightarrow\mc{N},\quad G\cdot\phi\mapsto G\cdot\phi\left(\begin{pmatrix}0&1\\0&0\end{pmatrix}\right).\]
It is clear that this map is well-defined. It is surjective by the Theorem of Jacobson-Morozov, and it is injective by the Theorem of Kostant.

The classification of nilpotent orbits $\mc{N}$ is described in \cite{collingwood1993nilpotent} (see the lists in Section 8.4 for all exceptional Lie algebras). At the same time, this classifies the embeddings $\bar{\rho}:\SLNC[2]\rightarrow\Aut{\mf{g}}$ needed for the current research.

\subsection{The intertwiner}
From the relative invariant form $P$ with zeros $S$ we construct the following matrix:
\begin{equation}
    \label{eq:modaut1}
    \modaut{P,1}=
    \begin{pmatrix}
        \frac{\partial_Y P}{\deg(P)P}&X\\[3mm]-\frac{\partial_X P}{\deg(P)P}&Y\\
    \end{pmatrix}
\end{equation}
where $\partial_X=\frac{\partial}{\partial X}$, $\partial_Y=\frac{\partial}{\partial Y}$ and $\deg(P)$ is the degree of $P$. This matrix allows the factorisation
\[
    \modaut{P,1}=
    \exp\left[\frac{X}{Y}\begin{pmatrix}0&1\\0&0\end{pmatrix}\right]
    \exp\left[-\frac{Y\partial_X P}{\deg(P)P}\begin{pmatrix}0&0\\1&0\end{pmatrix}\right]
    \begin{pmatrix}Y^{-1}&0\\0&Y\end{pmatrix},
\]
and has the following properties:
\begin{enumerate}[label=(\roman*)]
    \item\label{it:det1}$\det\modaut{P,1}=1$ for all $(X,Y)$ outside the zeros of $P$.
    \item\label{it:algebraic} The component functions of $\modaut{P,1}$ are in $\mb{C}[X,Y,P^{-1}]$.
    \item\label{it:equivariant} The columns of $\modaut{P,1}$ are $\brg$-equivariant: 
    \[\modaut[(\gamma(X,Y))]{P,1}=\gamma \modaut{P,1},\;\forall \gamma\in\brg.\]
    \item\label{it:homogeneous} The left column of $\modaut{P,1}$ has homogeneous degree $-1$, and the right has degree $1$.
\end{enumerate}
To explain the equivariance of \ref{it:equivariant}, and how we constructed the matrix $\modaut{P,1}$ in the first place, we remark that its left column is a scaling of the vector $\left(-\partial_Y P,\partial_X P\right)^T$. This is the Jacobian determinant, or the first transvectant, of $P$ with the $\SLNC[2]$-equivariant vector $\left( X,Y\right)$ (more on transvectants can be found in \cite{olver1999classical, olver2000transvectants}). Therefore, it is a relative invariant vector for $\rg$ with the same character as $P$.

Let $\bar{\rho}:\SLNC[2]\rightarrow\Aut{\mf{g}}$ be a representation with derivative $\rd\bar{\rho}:\slnc[2]\rightarrow\Der{\mf{g}}$ and set for notational convenience
\begin{equation}
    \label{eq:HEF}
    H=\rd\bar{\rho}\left(\begin{pmatrix}1&0\\0&-1\end{pmatrix}\right),\quad E=\rd\bar{\rho}\left(\begin{pmatrix}0&1\\0&0\end{pmatrix}\right),\quad F=\rd\bar{\rho}\left(\begin{pmatrix}0&0\\1&0\end{pmatrix}\right).
\end{equation}  
By $\mf{g}_k$ we will denote the eigenspace of $H$ with eigenvector $k$.

Using property \ref{it:det1} we can define $\modaut{P,\bar{\rho}}=\bar{\rho}(\modaut{1})$ for $(X, Y)\in\mb{C}^2\setminus S$ which reads
\beq{eq:modautlie}
\modaut{P,\bar{\rho}}=
\exp\left[\frac{X}{Y}E\right]
    \exp\left[-\frac{Y\partial_X P}{\deg(P)P}F\right]
    \exp\left[-\ln(Y)H\right].
\eeq
Here we use $\exp\left[-\ln(Y)H\right]$ as a notation for the operator that multiplies $H$-eigenvectors with eigenvalue $k$ by $Y^{-k}$.
\begin{Theorem}
\label{thm:aliahz}
$\modaut[]{P,\bar{\rho}}$ realises an isomorphism of graded Lie algebras and graded modules of $\mb{C}[X,Y,P^{-1}]^{\brg}$,
\[\mf{g}\otimes\mb{C}[X,Y,P^{-1}]^{\brg}\xrightarrow{\modaut[]{P,\bar{\rho}}}\mf{g}[X,Y,P^{-1}]^{\brg},\]
where $\mf{g}$ on the left-hand side is endowed with the grading defined by minus the weights as $\slnc[2]$-module, and $\mf{g}$ on the right-hand side has degree zero.
\end{Theorem}
\begin{proof}
First we notice that $\modaut[]{P,\bar{\rho}}$ sends $\mf{g}[X,Y,P^{-1}]$ to itself due to property \ref{it:algebraic} and the fact that any finite dimensional representation of $\SLNC[2]$ is algebraic. Because $\bar\rho$ maps into $\Aut{\mf{g}}$ we can be sure that $\modaut[]{P,\bar{\rho}}$ is a Lie algebra isomorphism of $\mf{g}[X,Y,P^{-1}]$.
Due to the equivariance of $\modaut[]{P,\bar{\rho}}$, \ref{it:equivariant}, we can conclude that the space $\mf{g}\otimes\mb{C}[X,Y,P^{-1}]^{\brg}$ is mapped onto $\mf{g}[X,Y,P^{-1}]^{\brg}$.

To understand the claim about the grading, we recall that the right most factor of $\modaut[]{P,\bar{\rho}}$ multiplies the eigenspace of $H$ with eigenvector $k$ by $Y^{-k}$. The degree $-k$ of this result is not changed after multiplication with the remaining factors of $\modaut[]{P,\bar{\rho}}$, since these have homogeneous degree zero.
\end{proof}
One could interpret the isomorphism $\mf{g}\otimes\mb{C}[X,Y,P^{-1}]^{\brg}\cong\mf{g}[X,Y,P^{-1}]^{\brg}$ as a form of triviality of the Lie algebra $\mf{g}[X,Y,P^{-1}]^{\brg}$, as the latter seemingly complicated Lie algebra is nothing but the finite dimensional Lie algebra $\mf{g}$ with its field extended to the ring $\mb{C}[X,Y,P^{-1}]^{\brg}$. 

Theorem \ref{thm:aliahz} is the analogue in spherical geometry to \cite[Theorem 4.3]{10.1093/imrn/rnab376} in hyperbolic geometry. The latter result does not have the factor $-1$ appearing in the identification of the grading. This is in line with the duality between binary forms and modular forms described in \cite{olver2000transvectants}, where the factor $-1$ appears.

Recall that the automorphic Lie algebra is the subalgebra of $\mf{g}[X,Y,P^{-1}]^{\brg}$ consisting of the elements of degree zero. By Theorem \ref{thm:aliahz} this equals 
\[
    \modaut[]{P,\bar{\rho}}\left(\bigoplus_{k}\mf{g}_{k}\otimes \mb{C}[X,Y,P^{-1}]^{\brg}_k\right).
\] 
From Theorem \ref{thm:invariant forms} we learned that $\mb{C}[X,Y,P^{-1}]^{\brg}_k=\mb{C}[X,Y,P^{-1}]^{\brg}_0 P_k$. We summarise in the following theorem, dropping the tensor notation for clarity.
\begin{Theorem}
    \label{thm:alia}
    The automorphic Lie algebra $\mf{g}[X,Y,P^{-1}]^{\brg}_0$ defined by the representation $\bar\rho:\SLNC[2]\to\Aut{\mf g}$ and the relative invariant form $P$ of the group $\brg\subset \SU(2)$ is generated by \[\bigoplus_k P_k\modaut[]{P,\bar{\rho}}(\mf g_{k})\]
    as a module over the invariant functions $\mb{C}[X,Y,P^{-1}]^{\brg}_0$. The subscript of $\mf g_{k}$ refers to the grading of $\mf g$ as $\SLNC[2]$-module, i.e. the $H$-eigenvalue, and $P_k$ is the form defined in (\ref{eq:P_k}).
\end{Theorem}
The key to the uniform construction of automorphic Lie algebras in this section is the intertwiner $\modaut[]{P,1}$. 
Of course, any matrix satisfying the properties (\ref{it:det1}, \ref{it:algebraic}, \ref{it:equivariant}, \ref{it:homogeneous}) can be used in place of $\modaut[]{P,1}$. Property \ref{it:algebraic} can also be relaxed: one only needs that $\modaut[]{P,\bar{\rho}}$ sends $\mf{g}[X,Y,P^{-1}]$ to itself. The first time this relaxation proved essential is in the PhD thesis of Oelen \cite{oelen2022automorphic} 
and the paper \cite{knibbeler2024classification} based on this work, 
where an intertwiner is constructed using square roots. After applying $\bar{\rho}$ to this intertwiner, only algebraic functions remain. Property \ref{it:homogeneous} can be relaxed as well: the columns of $\modaut[]{P,1}$ can be homogeneous of degree $-m$ and $m$, but this factor $m$ then appears in the identification of the gradings in Theorem \ref{thm:aliahz}, and consequently at several places in the next section (the definition of $\bar{a}_\alpha$ and $\omega^2$ to be explicit).

\subsection{Chevalley normal forms}
It is well known that any semisimple element in a simple Lie algebra is contained in a Cartan subalgebra. Thus, for any embedding of $\SLNC[2]$ in the automorphism group of $\mf{g}$ we may choose a Cartan subalgebra $\mf{h}$ of $\mf{g}$ such that $H$ is in $\ad(\mf{h})$. A Chevalley basis relative to such a Cartan subalgebra consists of $H$ eigenvectors with integral eigenvalues. More precisely, it is a basis $\{H_i, A_\alpha, i=1,\ldots,N, \alpha\in\roots\}$ with brackets
\[
\begin{array}{rll}
{[}H_i, H_j]&\backspace=0,&
\\{[}H_i, A_\alpha]&\backspace=\alpha(H_i) A_\alpha,&
\\{[}A_\alpha, A_{-\alpha}]&\backspace=H_\alpha,
\\{[}A_\alpha,A_\beta]&\backspace=\epsilon(\alpha,\beta)A_{\alpha+\beta},& \alpha+\beta\in\roots,
\\{[}A_\alpha,A_\beta]&\backspace=0,& \alpha+\beta\notin\roots\cup\{0\},
\end{array} 
\]
The element $H_\alpha$ of the Cartan subalgebra is dual to $\frac{2\alpha}{(\alpha,\alpha)}$, and often denoted $\alpha^\vee$.\footnote{For the simply laced Lie types $ADE$, $H_\alpha$ depends linearly on $\alpha$ and a simpler expression of $H_\alpha$ is possible, which has been erroneously used for all Lie types by the author in \cite{knibbeler2020cohomology} and \cite{10.1093/imrn/rnab376}.} The map $\epsilon$ satisfies nice algebraic properties (see for instance \cite[Section 25]{humphreys1972introduction} or \cite[Chapter VIII, Section 2, Number 4, Definition 3]{bourbaki2005lie}).

If $k(\alpha)$ is the $H$-eigenvalue of $A_\alpha$ then $k$ is an additive map on the root system $\roots$. 
In fact, the classification of $\slnc[2]$-triples in $\mf{g}$ is listed in \cite{collingwood1993nilpotent} using Dynkin diagrams with labels $0$, $1$ and $2$. These are known as Dynkin gradings, and one can find more wonderful facts about these gradings in the recent work of Elashvili, Jibladze and Kac \cite{elashvili2019on}. The map $k$ is the additive extension of the Dynkin grading to the whole root system.

By Theorem \ref{thm:aliahz}, the elements
\begin{align*}
    &h_i(X,Y)=\modaut{P,\bar{\rho}} H_i\otimes 1,\quad i=1,\ldots,N,\\
    &a_\alpha(X,Y)=\modaut{P,\bar{\rho}} A_\alpha\otimes 1, \quad \alpha\in\roots,
\end{align*}
form a basis of $\mf{g}[X,Y,P^{-1}]^{\brg}$, as a free module of $\mb{C}[X,Y,P^{-1}]^{\brg}$, with identical structure constants as $\{H_i, A_\alpha\}$. 
By Theorem \ref{thm:alia}, the nonzero elements
\[h_i,\quad \ba_\alpha=P_{k(\alpha)} a_\alpha\]
for $i=1,\ldots,N$ and $\alpha\in\roots$
form a basis of $\mf{g}[X,Y,P^{-1}]^{\brg}_0$ as free $\mb{C}[\mb{I}]$-module (here we recall that $P_k=0$ when $k$ is odd, but we are mostly interested in the case where $k(\alpha)$ is even for all roots $\alpha$). 

To describe the structure constants of this basis we will use the language of root cohomology introduced in \cite{knibbeler2020cohomology}. 
Recall that the functions $P_k$ in Theorem \ref{thm:invariant forms} are
formulated in terms of the map $\ex:2\mb{Z}\to\prod\{0,1,\ldots,\nu_i-1\}$. We need to divide component $i$ by $\nu_i$. Let $\frac{1}{\nu}$ denote the map doing exactly this.
Now we define 
\beq{eq:2cocycles}
\omega^2=\mathsf{d}\left(\frac{1}{\nu}\circ\ex\circ k\right)
\eeq
where $\mathsf{d}$ is the coboundary operator defined on the functions $f$ on the root system by \[(\mathsf{d}f)(\alpha,\beta)=f(\beta)-f(\alpha+\beta)+f(\alpha).\] 
The $2$-cocycle (\ref{eq:2cocycles}) takes values in $\{0,1\}^{|\Omega|}$, except when $\rg=\cg{2m}$, in which case each value is $(0,0)$ or $(1/2,1/2)$.
We have proved the following
\begin{Theorem}[Chevalley normal forms]
\label{thm:aliam}
The automorphic Lie algebra 
\[\mf{g}[X,Y,P^{-1}]^{\brg}_0\] 
is generated as a free $\mb{C}[\mb{I}]$-module by elements $h_i(X,Y)$ and $\ba_\alpha(X,Y)$ where $i=1,\ldots,N$, and $\alpha\in \roots$. The $\mb{C}[\mb{I}]$-linear Lie structure is given by the brackets
\[
\begin{array}{rll}
{[}h_i, h_j]&\backspace=0,&
\\{[}h_i, \ba_\alpha]&\backspace=\alpha(H_i) \ba_\alpha,&
\\{[}\ba_\alpha, \ba_{-\alpha}]&\backspace=\mb{I}^{\omega^2(\alpha,-\alpha)}h_\alpha,&
\\{[}\ba_\alpha,\ba_\beta]&\backspace=\epsilon(\alpha,\beta)\mb{I}^{\omega^2(\alpha,\beta)} \ba_{\alpha+\beta},& \alpha+\beta\in\roots,
\\{[}\ba_\alpha,\ba_\beta]&\backspace=0,& \alpha+\beta\notin\roots\cup\{0\},
\end{array} 
\]
where $\omega^2$ is the symmetric $2$-cocycles (\ref{eq:2cocycles}) on the root system $\roots$, and $\mb{I}^{\omega^2(\alpha,-\alpha)}$ denotes the product $\prod_{i\in\Omega}\mb{I}_i^{\omega^2(\alpha,-\alpha)_i}$ with $\mb{I}_i=P_i^{\nu_i} P^{-1}$.
\end{Theorem}
In analogy with the Chevalley basis for the simple Lie algebra described above, we use the term Chevalley normal forms for elements $h_i$, $i=1,\ldots,N$, $\ba_\alpha$, $\alpha\in\roots$, generating an automorphic Lie algebra as a module over the automorphic functions and satisfying the bracket relations of Theorem \ref{thm:aliam} (following the terminology of \cite{knibbeler2017higher}).\footnote{Bourbaki defines the notion of a Chevalley system in \cite[Chapter VIII, Section 2, Number 4, Definition 3]{bourbaki2005lie}. One of its defining features is that the linear extension of $H_\alpha\mapsto -H_\alpha$, $A_\alpha\mapsto A_{-\alpha}$ is an automorphism. This property is lost in the generalised setting of Theorem \ref{thm:aliam}, because $\omega^2(\alpha,\beta)$ is not generally equal to $\omega^2(-\alpha,-\beta)$.}
The normal forms introduced in \cite{lombardo2010on} and also obtained in \cite{knibbeler2017higher, knibbeler2024classification} correspond to Chevalley normal forms.

It is worth noting that the reductive Lie algebra $\mf g ^{\rg_i}$, where $\rg_i$ is the subgroup of $\rg$ stabilising an element of the exceptional orbit $S_i$, has a root system corresponding to the kernel of $\ex_i \circ k$. This allows us to find the derived series and lower central series of the automorphic Lie algebra.
\begin{Corollary}
    \label{cor:abelianisation}
The abelianisation of an automorphic Lie algebra 
\[\mf{A}=\mf{g}[X,Y,P^{-1}]^{\brg}_0\] 
with factorisable action on a simple Lie algebra $\mf g$ has dimension 
\[
    \dim_{\mb C} \mf{A}/[\mf{A}, \mf{A}]=
    \sum_{i\in\Omega^*}\dim_{\mb C}\mc Z \left(\mf g ^{\rg_i}\right)= 
    |\Omega^*|\rank(\mf g)-\sum_{i\in\Omega^*}\rank [\mf g ^{\rg_i},\mf g ^{\rg_i}]
\]
where $\mc Z \left(\mf g ^{\rg_i}\right)$ denotes the centre of $\mf g ^{\rg_i}$, and $\Omega^*=\{i\in\Omega\,|\,\mb I_i\ne 1\}$. 
\end{Corollary}
\begin{proof}
    Theorem \ref{thm:aliam} shows that $[\mf{A}, \mf{A}]$ is generated by the elements $\mb{I}^{\omega^2(\alpha,-\alpha)}h_\alpha$ and $\ba_\alpha$ with $\alpha\in\roots$. We pick one $i\in\Omega^*$ and investigate $\omega^2(\alpha,-\alpha)_i$. The roots $\alpha$ for which $\omega^2(\alpha,-\alpha)_i=0$ correspond to the roots of the reductive Lie algebra $\mf g ^{\rg_i}$. The dimension of the linear span of $\{h_\alpha\,|\, \omega^2(\alpha,-\alpha)_i=0\}$ corresponds to the rank of the semisimple summand $[\mf g ^{\rg_i},\mf g ^{\rg_i}]$ of $\mf g ^{\rg_i}$. This leaves $\rank(\mf g)- \rank [\mf g ^{\rg_i},\mf g ^{\rg_i}]=\dim_{\mb C}\mc Z \left(\mf g ^{\rg_i}\right)$ generators $\mb{I}^{\omega^2(\alpha,-\alpha)}h_\alpha$ of $[\mf{A}, \mf{A}]$ which are divisible by $\mb I_i$ (but not by $\mb I_i^2$, since $\omega^2_i$ only has values $0$ and $1$). 
    If we repeat this argument for all $i\in\Omega^*$ we find the displayed formula.
\end{proof}

\begin{Corollary}
    \label{cor:isomorphic alias}
    An automorphic Lie algebra with factorisable action depends only on the base Lie algebra and the set of $\min\{\nu,k(\tilde{\alpha})/2+1\}$ when $\nu$ runs over all orders of orbifold points of the quotient map $\overline{\mb C}\to\overline{\mb C}/\rg$ in the holomorphic domain of the automorphic Lie algebra, and $\tilde{\alpha}$ is the highest root of $\mf g$.
\end{Corollary}
\begin{proof}
    The $2$-cocycle (\ref{eq:2cocycles}) depends only on $\min\{\nu,k(\tilde{\alpha})/2+1\}$ when $\nu$ runs over all orders of orbifold points of the quotient map. The automorphic Lie algebra depends only on the components $\omega^2_i$ of the $2$-cocycle for which $\mb I_i\ne 1$ (which are precisely the Hauptmoduln vanishing on the orbifold points in the holomorphic domain) due to Theorem \ref{thm:aliam}.
\end{proof}
For example, any automorphic Lie algebra based on $\slnc[2]$ has a factorisable action, and $k(\tilde{\alpha})/2+1=2$. Therefore, such automorphic Lie algebras depend only on the number of orbifold points in the holomorphic domain. This recovers one of the main results of \cite{lombardo2010on} and \cite{bury2021automorphic} in a completely different way.

For another example, the lowest order orbifold points for the symmetry groups of the Platonic solids $\tg$, $\og$ and $\yg$ are $2$ and $3$. Choosing the poles at the remaining orbifold point, we find the same automorphic Lie algebras for these three groups, for all simple Lie algebras $\mf g$.\footnote{These automorphic Lie algebras also exist on the hyperbolic plane \cite{10.1093/imrn/rnab376}.}

The above corollaries are also useful in the search for nonisomorphic automorphic Lie algebras. If one would like to have nonisomorphic algebras for symmetry groups $\tg$, $\og$ and $\yg$ for instance, it is sufficient to pick $\mf g$ such that $k(\tilde{\alpha})/2+1\ge 5$. Corollary \ref{cor:abelianisation} establishes the nonexistence of an isomorphism. For instance, one can pick the regular embedding of $\PSL(2,\mb{C})$ in $\Aut{\mf g}$ so that $k(\tilde{\alpha})/2+1$ corresponds to the Coxeter number of $\mf g$, and ensure this Coxeter number is at least $5$, which amounts to taking $\slnc[N]$, with $N\ge 5$, $\mf {so}(N,\mb C)$ with $N\ge 7$, $\mf {sp}(2N,\mb C)$ with $2N\ge 6$, or taking any of the exceptional Lie algebras. All these cases are beyond the scope of former publications on the classification of automorphic Lie algebras, and these publications therefore showed many isomorphisms.

\begin{Example}[Icosahedral group and smallest orbit of poles, Lie type $A_2$ and principal nilpotent orbit.]
    \label{ex:explicit invariant matrices}
We let the binary icosahedral group be generated by 
\begin{align*}a&=\left(\begin{array}{rr}
\zeta_{20}^{2} & 0 \\
0 & -\zeta_{20}^{6} + \zeta_{20}^{4} - \zeta_{20}^{2} + 1
\end{array}\right)
\\
b&=\left(\begin{array}{rr}
\frac{2}{5} \zeta_{20}^{6} + \frac{1}{5} \zeta_{20}^{4} + \frac{1}{5} \zeta_{20}^{2} + \frac{2}{5} & -\frac{1}{5} \zeta_{20}^{6} + \frac{2}{5} \zeta_{20}^{4} + \frac{2}{5} \zeta_{20}^{2} - \frac{1}{5} \\
\frac{4}{5} \zeta_{20}^{6} - \frac{3}{5} \zeta_{20}^{4} + \frac{2}{5} \zeta_{20}^{2} - \frac{1}{5} & -\frac{2}{5} \zeta_{20}^{6} - \frac{1}{5} \zeta_{20}^{4} - \frac{1}{5} \zeta_{20}^{2} + \frac{3}{5}
\end{array}\right)
\\
c&=\left(\begin{array}{rr}
\frac{3}{5} \zeta_{20}^{6} - \frac{1}{5} \zeta_{20}^{4} + \frac{4}{5} \zeta_{20}^{2} - \frac{2}{5} & \frac{1}{5} \zeta_{20}^{6} + \frac{3}{5} \zeta_{20}^{4} - \frac{2}{5} \zeta_{20}^{2} + \frac{1}{5} \\
\frac{1}{5} \zeta_{20}^{6} + \frac{3}{5} \zeta_{20}^{4} - \frac{2}{5} \zeta_{20}^{2} + \frac{1}{5} & -\frac{3}{5} \zeta_{20}^{6} + \frac{1}{5} \zeta_{20}^{4} - \frac{4}{5} \zeta_{20}^{2} + \frac{2}{5}
\end{array}\right)
\end{align*}
satisfying $a^5=b^3=c^2=abc=-\Id$. The entries of the matrices live in the cyclotomic field $\mb{Q}(\zeta_{20})$ where $\zeta_{20}$ is a primitive $20$th root of unity. Computer calculations appear to be faster if the matrices are defined to live in this space.

The smallest orbit on the Riemann sphere has size $12$ and corresponding invariant form \[
P=X^{11} Y - 11 \, X^{6} Y^{6} - X Y^{11}.\]
From $P$ we can construct $\modaut{P,1}$ defined by (\ref{eq:modaut1}) resulting in
\[\modaut{P,1}=\left(\begin{array}{rr}
\frac{X^{10} - 66 \, X^{5} Y^{5} - 11 \, Y^{10}}{12 \, {\left(X^{10} Y - 11 \, X^{5} Y^{6} - Y^{11}\right)}} & X \\[3mm]
-\frac{11 \, X^{10} - 66 \, X^{5} Y^{5} - Y^{10}}{12 \, {\left(X^{11} - 11 \, X^{6} Y^{5} - X Y^{10}\right)}} & Y
\end{array}\right)
\]

For this example we choose a homomorphism $\bar{\rho}:\SLNC[2]\to\Aut{\slnc[3]}$ from the conjugacy class of such homomorphisms corresponding to the principal nilpotent orbit in $\slnc[3]$. We can make such a choice concrete using the homomorphism  
\newcommand{\Sym}{\mathrm{Sym}}
\[\Sym^2:\SLNC[2]\to\SLNC[3],\quad \Sym^2 \begin{pmatrix}a&b\\c&d\end{pmatrix}=
\begin{pmatrix}a^2&ab&b^2\\2ac&ad+bc&2
    
    bd\\c^2&cd&d^2\end{pmatrix},\]
which has kernel $\pm\Id$. Then we compose with the adjoint map
\[\bar{\rho}=\Ad\circ\Sym^2.\]
This composition has kernel $\pm\Id$ as well.
Notice that $H=\ad\,\diag(2,0,-2)$ (where $H$ is as in (\ref{eq:HEF})). The kernel of $H$ is a CSA of $\slnc[3]$. This proves that $E$ and $F$ belong to the principal nilpotent orbit (the nilpotent orbit of the largest dimension) of $\slnc[3]$ \cite{collingwood1993nilpotent}.

Now we pick our favourite basis for $\slnc[3]$ that diagonalises $H$:
$$
E_0=\left(\begin{array}{rrr}
0 & 0 & 0 \\
0 & 0 & 0 \\
1 & 0 & 0
\end{array}\right), 
E_1=\left(\begin{array}{rrr}
0 & 1 & 0 \\
0 & 0 & 0 \\
0 & 0 & 0
\end{array}\right), 
E_2=\left(\begin{array}{rrr}
0 & 0 & 0 \\
0 & 0 & 1 \\
0 & 0 & 0
\end{array}\right), 
$$
\[F_0=[E_1,E_2],\quad F_1=[E_2,E_0],\quad F_2=[E_0,E_1],\quad H_i=[E_i,F_i].\]
The eigenvalues of $E_0,E_1,E_2$ are respectively $-4,2,2$.
We apply $\modaut{P, \Sym^2}$ defined in (\ref{eq:modautlie}) to these generators to get
\[
e_0=\left(\begin{array}{rrr}
    X^{2} Y^{2} & -X^{3} Y & X^{4} \\
    2 \, X Y^{3} & -2 \, X^{2} Y^{2} & 2 \, X^{3} Y \\
    Y^{4} & -X Y^{3} & X^{2} Y^{2}
    \end{array}\right)
=
\left(\begin{array}{r}
    X^{2} \\
    2 \, X Y \\
    Y^{2}
    \end{array}
\right)
\left(\begin{array}{rrr}
    Y^{2} & -X Y & X^{2}
    \end{array}
\right)
\]
{\footnotesize
\begin{multline*}
    e_1=
    \left(
        \begin{array}{r}
            \frac{{\left(X^{10} - 66 \, X^{5} Y^{5} - 11 \, Y^{10}\right)}^{2}}{144 \, {\left(X^{4} + 3 \, X^{3} Y + 4 \, X^{2} Y^{2} + 2 \, X Y^{3} + Y^{4}\right)}^{2} {\left(X^{4} - 2 \, X^{3} Y + 4 \, X^{2} Y^{2} - 3 \, X Y^{3} + Y^{4}\right)}^{2} {\left(X^{2} - X Y - Y^{2}\right)}^{2} Y^{2}} \\[4mm]
            \frac{{-\left(11 \, X^{10} - 66 \, X^{5} Y^{5} - Y^{10}\right)} {\left(X^{10} - 66 \, X^{5} Y^{5} - 11 \, Y^{10}\right)}}{72 \, {\left(X^{4} + 3 \, X^{3} Y + 4 \, X^{2} Y^{2} + 2 \, X Y^{3} + Y^{4}\right)}^{2} {\left(X^{4} - 2 \, X^{3} Y + 4 \, X^{2} Y^{2} - 3 \, X Y^{3} + Y^{4}\right)}^{2} {\left(X^{2} - X Y - Y^{2}\right)}^{2} X Y} \\[4mm]
            \frac{{\left(11 \, X^{10} - 66 \, X^{5} Y^{5} - Y^{10}\right)}^{2}}{144 \, {\left(X^{4} + 3 \, X^{3} Y + 4 \, X^{2} Y^{2} + 2 \, X Y^{3} + Y^{4}\right)}^{2} {\left(X^{4} - 2 \, X^{3} Y + 4 \, X^{2} Y^{2} - 3 \, X Y^{3} + Y^{4}\right)}^{2} {\left(x^{2} - X Y - Y^{2}\right)}^{2} X^{2}}
        \end{array}
    \right)
    \\[4mm]
    \cdot\left(
        \begin{array}{r}
            \frac{{\left(11 \, X^{10} - 66 \, X^{5} Y^{5} - Y^{10}\right)} Y}{6 \, {\left(X^{4} + 3 \, X^{3} Y + 4 \, X^{2} Y^{2} + 2 \, X Y^{3} + Y^{4}\right)} {\left(X^{4} - 2 \, X^{3} Y + 4 \, X^{2} Y^{2} - 3 \, X Y^{3} + Y^{4}\right)} {\left(X^{2} - X Y - Y^{2}\right)} X} \\[4mm]
            \frac{-5 \, {\left(X^{8} - X^{6} Y^{2} + X^{4} Y^{4} - X^{2} Y^{6} + Y^{8}\right)} {\left(X^{2} + Y^{2}\right)}}{6 \, {\left(X^{4} + 3 \, X^{3} Y + 4 \, X^{2} Y^{2} + 2 \, X Y^{3} + Y^{4}\right)} {\left(X^{4} - 2 \, X^{3} Y + 4 \, X^{2} Y^{2} - 3 \, X Y^{3} + Y^{4}\right)} {\left(X^{2} - X Y - Y^{2}\right)}} \\[4mm]
            \frac{{-\left(X^{10} - 66 \, X^{5} Y^{5} - 11 \, Y^{10}\right)} X}{6 \, {\left(X^{4} + 3 \, X^{3} Y + 4 \, X^{2} Y^{2} + 2 \, X Y^{3} + Y^{4}\right)} {\left(X^{4} - 2 \, X^{3} Y + 4 \, X^{2} Y^{2} - 3 \, X Y^{3} + Y^{4}\right)} {\left(X^{2} - X Y - Y^{2}\right)} Y}
        \end{array}
    \right)^T
\end{multline*}
\begin{multline*}
    e_2=
    \left(
        \begin{array}{r}
            \frac{{\left(X^{10} - 66 \, X^{5} Y^{5} - 11 \, Y^{10}\right)} X}{12 \, {\left(X^{4} + 3 \, X^{3} Y + 4 \, X^{2} Y^{2} + 2 \, X Y^{3} + Y^{4}\right)} {\left(X^{4} - 2 \, X^{3} Y + 4 \, X^{2} Y^{2} - 3 \, X Y^{3} + Y^{4}\right)} {\left(X^{2} - X Y - Y^{2}\right)} Y} \\[4mm]
            \frac{-5 \, {\left(X^{8} - X^{6} Y^{2} + X^{4} Y^{4} - X^{2} Y^{6} + Y^{8}\right)} {\left(X^{2} + Y^{2}\right)}}{6 \, {\left(X^{4} + 3 \, X^{3} Y + 4 \, X^{2} Y^{2} + 2 \, X Y^{3} + Y^{4}\right)} {\left(X^{4} - 2 \, X^{3} Y + 4 \, X^{2} Y^{2} - 3 \, X Y^{3} + Y^{4}\right)} {\left(X^{2} - X Y - Y^{2}\right)}} \\[4mm]
            \frac{{-\left(11 \, X^{10} - 66 \, X^{5} Y^{5} - Y^{10}\right)} Y}{12 \, {\left(X^{4} + 3 \, X^{3} Y + 4 \, X^{2} Y^{2} + 2 \, X Y^{3} + Y^{4}\right)} {\left(X^{4} - 2 \, X^{3} Y + 4 \, X^{2} Y^{2} - 3 \, X Y^{3} + Y^{4}\right)} {\left(X^{2} - X Y - Y^{2}\right)} X}
        \end{array}
    \right)
    \\[4mm]
    \cdot\left(
        \begin{array}{r}
            \frac{{\left(11 \, X^{10} - 66 \, X^{5} Y^{5} - Y^{10}\right)}^{2}}{144 \, {\left(X^{4} + 3 \, X^{3} Y + 4 \, X^{2} Y^{2} + 2 \, X Y^{3} + Y^{4}\right)}^{2} {\left(X^{4} - 2 \, X^{3} Y + 4 \, X^{2} Y^{2} - 3 \, X Y^{3} + Y^{4}\right)}^{2} {\left(X^{2} - X Y - Y^{2}\right)}^{2} X^{2}} \\[4mm]
            \frac{{\left(11 \, X^{10} - 66 \, X^{5} Y^{5} - Y^{10}\right)} {\left(X^{10} - 66 \, X^{5} Y^{5} - 11 \, Y^{10}\right)}}{144 \, {\left(X^{4} + 3 \, X^{3} Y + 4 \, X^{2} Y^{2} + 2 \, X Y^{3} + Y^{4}\right)}^{2} {\left(X^{4} - 2 \, X^{3} Y + 4 \, X^{2} Y^{2} - 3 \, X Y^{3} + Y^{4}\right)}^{2} {\left(X^{2} - X Y - Y^{2}\right)}^{2} X Y} \\[4mm]
            \frac{{\left(X^{10} - 66 \, X^{5} Y^{5} - 11 \, Y^{10}\right)}^{2}}{144 \, {\left(X^{4} + 3 \, X^{3} Y + 4 \, X^{2} Y^{2} + 2 \, X Y^{3} + Y^{4}\right)}^{2} {\left(X^{4} - 2 \, X^{3} Y + 4 \, X^{2} Y^{2} - 3 \, X Y^{3} + Y^{4}\right)}^{2} {\left(X^{2} - X Y - Y^{2}\right)}^{2} Y^{2}}
        \end{array}
    \right)^{T}.
\end{multline*}
}

These matrices generate a subalgebra of $\slnc[3][X,Y,P^{-1}]^{\brg}$ isomorphic to $\slnc[3]$, illustrating Theorem \ref{thm:aliahz}.  
From Theorem \ref{thm:alia} we learn that the automorphic Lie algebra, the subalgebra $\slnc[3][X,Y,P^{-1}]^{\brg}_0$ of degree zero, has basis
\begin{align*}
\bar{e}_0&=F_{-4}e_0,\quad \bar{e}_1=F_{2}e_1,\quad \bar{e}_2=F_{2}e_2,
\\\bar{f}_0&=F_{4}f_0,\quad \bar{f}_1=F_{-2}f_1,\quad \bar{f}_2=F_{-2}f_2,
\\\bar{h}_1&=h_1,\quad\bar{h}_2=h_2
\end{align*}
as free $\mb{C}[\mb{I}]$-module. The structure constants are seen to be as in Theorem \ref{thm:aliam}.
\end{Example}

\section{Examples and applications}
\label{sec:examples}
In this section we provide more ways to use the results obtained in Section \ref{sec:alias}. First we recall that automorphic Lie algebras can be studied through graphs, whose set of vertices is a root system. In rank $2$ this allows a graphical classification of automorphic Lie algebras. In higher rank, graph theory can be used as well, but pictures of these graphs will likely be too complicated to be of any help.

In the next subsection we study automorphic Lie algebras by integrating the $2$-cocycle to a $1$-form $\omega^1$ as was done in \cite{knibbeler2017higher}, and compare the results. It turns out the results can be extended to every simple Lie type, and there exists a unique normal form for $\omega^1$ if we allow values $1$, $0$ and $-1$.

We end this section with tables of structure constants for automorphic Lie algebras of all exceptional Lie types.

\subsection{Graphs on root systems and the classification of automorphic Lie algebras of rank 2}
\label{sec:graphs}
Each component of the symmetric $2$-cocycle (\ref{eq:2cocycles}) is a map sending a pair of roots to one of two values, and can therefore conveniently be represented by a graph, whose set of vertices is the set of roots, and whose set of edges is, for instance, the collection of pairs of roots that are sent to the positive value by the cocycle.

In this section we present these graphs for all cases with Lie rank $2$ 
and all embeddings $\PSL(2,\mb C)\rightarrow \Aut{\mf g}$.
The choice to stick to rank $2$ is practical: we can draw the root system without using projections.

The automorphic Lie algebras of type $A_2$ (referring to the Lie type of $\mf g$) appearing in this section have been published in \cite{knibbeler2017higher}. The automorphic Lie algebras of type $B_2\cong C_2$ and $G_2$ are new results. 
The classification of automorphic Lie algebras of type $G_2$ is especially interesting, not only because the exceptional Lie algebras are unexplored in this context, but also because the portion of monomorphisms of polyhedral groups into $\Aut{\mf g _2}$ that factor through $\PSL(2,\mb C)$ is very large \cite[Theorem 3.3]{knibbeler2022polyhedral}.

\begin{Example}
    \label{ex:graph}
    Let us construct the graph, and therefore the structure constants, of the automorphic Lie algebra of Example \ref{ex:explicit invariant matrices} without computing the equivariant matrices. To this end, we must obtain the 2-cocycle (\ref{eq:2cocycles}).

    We start with a homomorphism $\slnc[2]\to\slnc[3]$ (to obtain the $k$ in (\ref{eq:2cocycles})). This is given by the Dynkin diagram with labels from $\{0,1,2\}$. In Example \ref{ex:explicit invariant matrices}, both simple roots have label $2$, and this is the only even case for $A_2$ (where even means that there is no label $1$, and implies the homomorphism $\slnc[2]\to\slnc[3]$ lifts to a homomorphism $\PSL(2,\mb C)\to\Aut{\SLNC[3]}$). The map $k$ on the roots is the additive extension of these labels. If we write the values of $k$ in the root system, we get 
    \[\begin{tikzpicture}[scale=0.4]
        \path 
        node at ( 0,0) [root,draw,label=270: $ $] (zero) {$ $}	
        node at ( 0,0) [2ndroot,draw,label=270: $ $] (zero) {$ $}	
        node at ( 4,0) [root,draw,label=0: $\alpha_1$, label=90: $2$] (one) {$ $}
        node at ( 2,3.464) [root,draw,label=60: $ $, label=90: $4$] (two) {$ $}
        node at ( -2,3.464) [root,draw,label=180: $\alpha_2$, label=90: $2$] (three) {$ $}
        node at ( -4,0) [root,draw,label=180: $ $, label=90: $-2$] (four) {$ $}
        node at ( -2,-3.464) [root,draw,label=240: $ $, label=90: $-4$] (five) {$ $}
        node at ( 2,-3.464) [root,draw,label=300: $ $, label=90: $-2$] (six) {$ $};
    \end{tikzpicture}\]
    where we do not draw lines from the origin to the roots, as is usually done, in order to avoid confusion with the edges of the graph to follow. Instead, we draw the origin with a double circle.
    
    The map $\ex$ appearing in (\ref{eq:2cocycles}) can be read of in Table \ref{tab:exponents}. Choosing the icosahedral group, the map $\ex\circ k$ on the root system becomes
    \[\begin{tikzpicture}[scale=0.4]
        \path 
        node at ( 0,0) [root,draw,label=270: $ $] (zero) {$ $}	
        node at ( 0,0) [2ndroot,draw,label=270: $ $] (zero) {$ $}	
        node at ( 4,0) [root,draw,label=0: $\alpha_1$, label=90: {$({\color{\colora}1},{\color{\colorb}1},{\color{\colorc}1})$}] (one) {$ $}
        node at ( 2,3.464) [root,draw,label=60: $ $, label=90: {$({\color{\colora}2},{\color{\colorb}2},{\color{\colorc}0})$}] (two) {$ $}
        node at ( -2,3.464) [root,draw,label=180: $\alpha_2$, label=90: {$({\color{\colora}1},{\color{\colorb}1},{\color{\colorc}1})$}] (three) {$ $}
        node at ( -4,0) [root,draw,label=180: $ $, label=90: {$({\color{\colora}4},{\color{\colorb}2},{\color{\colorc}1})$}] (four) {$ $}
        node at ( -2,-3.464) [root,draw,label=240: $ $, label=90: {$({\color{\colora}3},{\color{\colorb}1},{\color{\colorc}0})$}] (five) {$ $}
        node at ( 2,-3.464) [root,draw,label=300: $ $, label=90: {$({\color{\colora}4},{\color{\colorb}2},{\color{\colorc}1})$}] (six) {$ $};
    \end{tikzpicture}\]
where we have used colours to distinguish the orders ${\color{\colora}5}$, ${\color{\colorb}3}$ and ${\color{\colorc}2}$ appearing in the icosahedral group. 

At this stage we can draw the graph of the 2-cocycle (\ref{eq:2cocycles}), resulting in
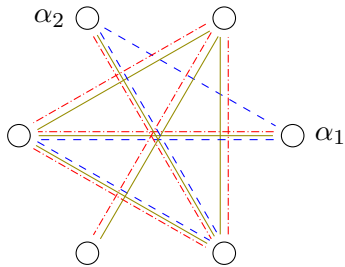
\[\begin{tikzpicture}[scale=0.45]
    \path 
    node at ( 4,0) [bigroot,draw,label=0: $\alpha_1$] (one) {$ $}
    node at ( 2,3.464) [bigroot,draw,label=60: $ $] (two) {$ $}
    node at ( -2,3.464) [bigroot,draw,label=180: $\alpha_2$] (three) {$ $}
    node at ( -4,0) [bigroot,draw,label=180: $ $] (four) {$ $}
    node at ( -2,-3.464) [bigroot,draw,label=240: $ $] (five) {$ $}
    node at ( 2,-3.464) [bigroot,draw,label=300: $ $] (six) {$ $};
      \draw[scochainc] (three.330) to node [sloped,above]{$ $} (one.150);
      \draw[scochaina] (two.210-\dangle) to node []{$ $} (four.30+\dangle);
      \draw[scochainb] (two.210+\dangle) to node []{$ $} (four.30-\dangle);
      \draw[scochaina] (four.330-\dangle) to node [sloped,above]{$ $} (six.150+\dangle);
      \draw[scochainb] (four.330) to node [sloped,above]{$ $} (six.150);
      \draw[scochainc] (four.330+\dangle) to node [sloped,above]{$ $} (six.150-\dangle);
      \draw[scochaina] (six.90-\dangle) to node [sloped,above]{$ $} (two.270+\dangle);
      \draw[scochainb] (six.90+\dangle) to node [sloped,above]{$ $} (two.270-\dangle);
      \draw[scochaina] (one.180-\dangle) to node []{$ $} (four.0+\dangle);
      \draw[scochainb] (one.180) to node []{$ $} (four.0);
      \draw[scochainc] (one.180+\dangle) to node []{$ $} (four.0-\dangle);
      \draw[scochaina] (two.240-\dangle) to node []{$ $} (five.60+\dangle);
      \draw[scochainb] (two.240+\dangle) to node []{$ $} (five.60-\dangle);
      \draw[scochaina] (three.300-\dangle) to node []{$ $} (six.120+\dangle);
      \draw[scochainb] (three.300) to node []{$ $} (six.120);
      \draw[scochainc] (three.300+\dangle) to node []{$ $} (six.120-\dangle);
  \end{tikzpicture}\]
  where we have used dashdotted lines, solid lines and dashed lines for the numbers $5$, $3$ and $2$ respectively, just in case the colours are hard to distinguish.
\end{Example}

If we repeat the above example for all polyhedral groups, we obtain the classification presented in Figure \ref{fig:A2}. Notice in particular that the automorphic Lie algebras with tetrahedral, octahedral and icosahedral symmetry groups are isomorphic at type $A_2$. This Lie algebra has been used to study integrable partial differential equations by Berkeley, Mikhailov and Xenitidis \cite{berkeley2016darboux}.

\begin{figure}[ht!] 
    \caption[]{Automorphic Lie algebras of type $A_2$, from Dynkin grading 
        \begin{tikzpicture}[scale=\dynkintablescale, baseline=(current bounding box.center), font=\dynkinfont, decoration={markings, mark=between positions 0.6 and 5 step 8mm with {\arrow{<}}}]
            \path 
            node at ( -0.5,0) [dynkinnode,label=90: $2 $,label=270: $\alpha_1$] (one) {$ $}
            node at ( 0.5,0) [dynkinnode,label=90: $2 $,label=270: $\alpha_2$] (two) {$ $}
            ;
            \draw[] (one) to node []{$ $} (two);
        \end{tikzpicture}
    }
    \label{fig:A2}
    \begin{center}
        $\begin{array}{c}
            \begin{array}{ccc}
                \rg=\cg{2}&\rg=\cg{n},\,n\ge 3&\rg=\dg{2}\\\\ 
                \begin{tikzpicture}[scale=\scaleAA]
                    \path 	
                    node at ( 4,0) [root,draw,label=0: $\alpha_1$] (one) {$ $}
                    node at ( 2,3.464) [root,draw,label=60: $ $] (two) {$ $}
                    node at ( -2,3.464) [root,draw,label=180: $\alpha_2$] (three) {$ $}
                    node at ( -4,0) [root,draw,label=180: $ $] (four) {$ $}
                    node at ( -2,-3.464) [root,draw,label=240: $ $] (five) {$ $}
                    node at ( 2,-3.464) [root,draw,label=300: $ $] (six) {$ $};
                        \draw[scochaina] (three.330-\dangle) to node [sloped,above]{$ $} (one.150+\dangle);
                        \draw[scochainb] (three.330+\dangle) to node [sloped,above]{$ $} (one.150-\dangle);
                        \draw[scochaina] (four.330-\dangle) to node [sloped,above]{$ $} (six.150+\dangle);
                        \draw[scochainb] (four.330+\dangle) to node [sloped,above]{$ $} (six.150-\dangle);
                        \draw[scochaina] (one.180+\dangle) to node []{$ $} (four.0-\dangle);
                        \draw[scochainb] (one.180-\dangle) to node []{$ $} (four.0+\dangle);
                        \draw[scochaina] (three.300-\dangle) to node []{$ $} (six.120+\dangle);
                        \draw[scochainb] (three.300+\dangle) to node []{$ $} (six.120-\dangle);
                \end{tikzpicture}&
                \begin{tikzpicture}[scale=\scaleAA]
                    \path 	
                    node at ( 4,0) [root,draw,label=0: $\alpha_1$] (one) {$ $}
                    node at ( 2,3.464) [root,draw,label=60: $ $] (two) {$ $}
                    node at ( -2,3.464) [root,draw,label=180: $\alpha_2$] (three) {$ $}
                    node at ( -4,0) [root,draw,label=180: $ $] (four) {$ $}
                    node at ( -2,-3.464) [root,draw,label=240: $ $] (five) {$ $}
                    node at ( 2,-3.464) [root,draw,label=300: $ $] (six) {$ $};
                    \draw[scochaina] (two.210-\dangle) to node []{$ $} (four.30+\dangle);
                    \draw[scochainb] (two.210+\dangle) to node []{$ $} (four.30-\dangle);
                    \draw[scochaina] (four.330-\dangle) to node [sloped,above]{$ $} (six.150+\dangle);
                    \draw[scochainb] (four.330) to node [sloped,above]{$ $} (six.150);
                    \draw[scochaina] (six.90-\dangle) to node [sloped,above]{$ $} (two.270+\dangle);
                    \draw[scochainb] (six.90+\dangle) to node [sloped,above]{$ $} (two.270-\dangle);
                    \draw[scochaina] (one.180-\dangle) to node []{$ $} (four.0+\dangle);
                    \draw[scochainb] (one.180+\dangle) to node []{$ $} (four.0-\dangle);
                    \draw[scochaina] (two.240-\dangle) to node []{$ $} (five.60+\dangle);
                    \draw[scochainb] (two.240+\dangle) to node []{$ $} (five.60-\dangle);
                    \draw[scochaina] (three.300-\dangle) to node []{$ $} (six.120+\dangle);
                    \draw[scochainb] (three.300+\dangle) to node []{$ $} (six.120-\dangle);
                \end{tikzpicture}&
                \begin{tikzpicture}[scale=\scaleAA]
                    \path 	
                    node at ( 4,0) [root,draw,label=0: $\alpha_1$] (one) {$ $}
                    node at ( 2,3.464) [root,draw,label=60: $ $] (two) {$ $}
                    node at ( -2,3.464) [root,draw,label=180: $\alpha_2$] (three) {$ $}
                    node at ( -4,0) [root,draw,label=180: $ $] (four) {$ $}
                    node at ( -2,-3.464) [root,draw,label=240: $ $] (five) {$ $}
                    node at ( 2,-3.464) [root,draw,label=300: $ $] (six) {$ $};
                        \draw[scochaina] (three.330-\dangle) to node [sloped,above]{$ $} (one.150+\dangle);
                        \draw[scochainb] (three.330) to node [sloped,above]{$ $} (one.150);
                        \draw[scochainc] (three.330+\dangle) to node [sloped,above]{$ $} (one.150-\dangle);
                        \draw[scochaina] (four.330-\dangle) to node [sloped,above]{$ $} (six.150+\dangle);
                        \draw[scochainb] (four.330) to node [sloped,above]{$ $} (six.150);
                        \draw[scochainc] (four.330+\dangle) to node [sloped,above]{$ $} (six.150-\dangle);
                        \draw[scochaina] (one.180+\dangle) to node []{$ $} (four.0-\dangle);
                        \draw[scochainb] (one.180) to node []{$ $} (four.0);
                        \draw[scochainc] (one.180-\dangle) to node []{$ $} (four.0+\dangle);
                        \draw[scochaina] (three.300-\dangle) to node []{$ $} (six.120+\dangle);
                        \draw[scochainb] (three.300) to node []{$ $} (six.120);
                        \draw[scochainc] (three.300+\dangle) to node []{$ $} (six.120-\dangle);
                    \end{tikzpicture}
            \end{array}\\\\
            \begin{array}{cc}
                \rg=\dg{n}, n\ge 3&\rg=\tg, \og, \yg\\\\
                \begin{tikzpicture}[scale=\scaleAA]
                    \path 	
                    node at ( 4,0) [root,draw,label=0: $\alpha_1$] (one) {$ $}
                    node at ( 2,3.464) [root,draw,label=60: $ $] (two) {$ $}
                    node at ( -2,3.464) [root,draw,label=180: $\alpha_2$] (three) {$ $}
                    node at ( -4,0) [root,draw,label=180: $ $] (four) {$ $}
                    node at ( -2,-3.464) [root,draw,label=240: $ $] (five) {$ $}
                    node at ( 2,-3.464) [root,draw,label=300: $ $] (six) {$ $};
                        \draw[scochainb] (three.330) to node [sloped,above]{$ $} (one.150);
                        \draw[scochainc] (three.330+\dangle) to node [sloped,above]{$ $} (one.150-\dangle);
                        \draw[scochaina] (two.210) to node []{$ $} (four.30);
                        \draw[scochaina] (four.330-\dangle) to node [sloped,above]{$ $} (six.150+\dangle);
                        \draw[scochainb] (four.330) to node [sloped,above]{$ $} (six.150);
                        \draw[scochainc] (four.330+\dangle) to node [sloped,above]{$ $} (six.150-\dangle);
                        \draw[scochaina] (six.90) to node [sloped,above]{$ $} (two.270);
                        \draw[scochaina] (one.180-\dangle) to node []{$ $} (four.0+\dangle);
                        \draw[scochainb] (one.180) to node []{$ $} (four.0);
                        \draw[scochainc] (one.180+\dangle) to node []{$ $} (four.0-\dangle);
                        \draw[scochaina] (two.240) to node []{$ $} (five.60);
                        \draw[scochaina] (three.300-\dangle) to node []{$ $} (six.120+\dangle);
                        \draw[scochainb] (three.300) to node []{$ $} (six.120);
                        \draw[scochainc] (three.300+\dangle) to node []{$ $} (six.120-\dangle);
                \end{tikzpicture}&
                \begin{tikzpicture}[scale=\scaleAA]
                    \path 	
                    node at ( 4,0) [root,draw,label=0: $\alpha_1$] (one) {$ $}
                    node at ( 2,3.464) [root,draw,label=60: $ $] (two) {$ $}
                    node at ( -2,3.464) [root,draw,label=180: $\alpha_2$] (three) {$ $}
                    node at ( -4,0) [root,draw,label=180: $ $] (four) {$ $}
                    node at ( -2,-3.464) [root,draw,label=240: $ $] (five) {$ $}
                    node at ( 2,-3.464) [root,draw,label=300: $ $] (six) {$ $};
                        \draw[scochainc] (three.330) to node [sloped,above]{$ $} (one.150);
                        \draw[scochaina] (two.210-\dangle) to node []{$ $} (four.30+\dangle);
                        \draw[scochainb] (two.210+\dangle) to node []{$ $} (four.30-\dangle);
                        \draw[scochaina] (four.330-\dangle) to node [sloped,above]{$ $} (six.150+\dangle);
                        \draw[scochainb] (four.330) to node [sloped,above]{$ $} (six.150);
                        \draw[scochainc] (four.330+\dangle) to node [sloped,above]{$ $} (six.150-\dangle);
                        \draw[scochaina] (six.90-\dangle) to node [sloped,above]{$ $} (two.270+\dangle);
                        \draw[scochainb] (six.90+\dangle) to node [sloped,above]{$ $} (two.270-\dangle);
                        \draw[scochaina] (one.180-\dangle) to node []{$ $} (four.0+\dangle);
                        \draw[scochainb] (one.180) to node []{$ $} (four.0);
                        \draw[scochainc] (one.180+\dangle) to node []{$ $} (four.0-\dangle);
                        \draw[scochaina] (two.240-\dangle) to node []{$ $} (five.60+\dangle);
                        \draw[scochainb] (two.240+\dangle) to node []{$ $} (five.60-\dangle);
                        \draw[scochaina] (three.300-\dangle) to node []{$ $} (six.120+\dangle);
                        \draw[scochainb] (three.300) to node []{$ $} (six.120);
                        \draw[scochainc] (three.300+\dangle) to node []{$ $} (six.120-\dangle);
                \end{tikzpicture}
            \end{array}
        \end{array}$
    \end{center}
\end{figure}
Notice that we might as well look at one component of $\ex$ at the time, and add the resulting graphs afterwards. Therefore we present the graphs on the root systems of type $C_2$ and $G_2$ in Figures \ref{fig:B2} and \ref{fig:G2} for each value of $\nu$ instead of each group, which keeps things simpler and shorter. We have used the numbering of simple roots consistent with Bourbaki \cite{bourbaki2002lie} and SageMath \cite{sagemath} (which we will use in Section \ref{sec:exceptional aLias}). For Lie type $G_2$ this numbering is not consistent with that of Kac \cite{kac1990infinite} and \cite{10.1093/imrn/rnab376}.

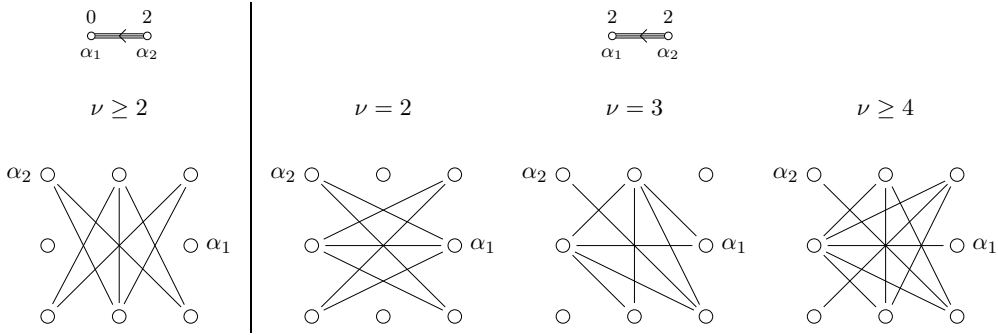
\begin{figure}[ht!] 
    \caption[]{$2$-cocycles on $C_2$ from Dynkin gradings        
    }
    \label{fig:B2}
    \begin{center}
        \begin{adjustbox}{width=\linewidth}
        $\begin{array}{c|ccc}
            \begin{tikzpicture}[scale=\dynkintablescale,baseline=(current bounding box.center), font=\dynkinfont,decoration={
                markings,
                    mark=at position 0.5 with {
            \draw (-3pt,-3pt) -- (0pt,0pt);
            \draw (0pt,-0pt) -- (-3pt,3pt);
            } }]
            \path 
            node at ( 0,0) [sdynkinnode,label=90: $0 $,label=270: $\alpha_1$] (one) {$ $}
            node at ( 1,0) [ldynkinnode,label=90: $2 $,label=270: $\alpha_2$] (two) {$ $}
            ;
            \draw[] (one.30) to node []{$ $} (two.150);
            \draw[] (one.-30) to node []{$ $} (two.-150);
            \fill[postaction={decorate}] (two) -- (one);
            \end{tikzpicture} 
            &
            &
            \begin{tikzpicture}[scale=\dynkintablescale,baseline=(current bounding box.center), font=\dynkinfont,decoration={
                markings,
                    mark=at position 0.5 with {
            \draw (-3pt,-3pt) -- (0pt,0pt);
            \draw (0pt,-0pt) -- (-3pt,3pt);
            } }]
            \path 
            node at ( 0,0) [sdynkinnode,label=90: $2 $,label=270: $\alpha_1$] (one) {$ $}
            node at ( 1,0) [ldynkinnode,label=90: $2 $,label=270: $\alpha_2$] (two) {$ $}
            ;
            \draw[] (one.30) to node []{$ $} (two.150);
            \draw[] (one.-30) to node []{$ $} (two.-150);
            \fill[postaction={decorate}] (two) -- (one);
            \end{tikzpicture}
            &
            \\\\
            \nu\ge 2&\nu= 2&\!\!\!\nu=3&\!\!\!\nu\ge 4
            \\\\
            \begin{tikzpicture}[scale=\scaleBB]
                \path 		
                node at (1,0) [root,draw,label=0: {$\alpha_1$}] (m10) {$  $}
                node at (-1,1) [root,draw,label=180: {$\alpha_2$}] (m01) {$  $}
                node at (0,1) [root,draw,label=90: {$ $}] (m11) {$  $}
                node at (1,1) [root,draw,label=0: {$ $}] (m21) {$  $}
                
                node at (-1,0) [root,draw,label=0: {$ $}] (n10) {$  $}
                node at (1,-1) [root,draw,label=180: {$ $}] (n01) {$  $}
                node at (0,-1) [root,draw,label=90: {$ $}] (n11) {$  $}
                node at (-1,-1) [root,draw,label=0: {$ $}] (n21) {$  $}
                ;
                \draw[scochain](m01.315) to node{}(n01.135);
                \draw[scochain](m11.270) to node{}(n11.90);
                \draw[scochain](m21.225) to node{}(n21.45);
                \draw[scochain](m01.300) to node{}(n11.120);
                \draw[scochain](m11.300) to node{}(n01.120);
                \draw[scochain](m11.255) to node{}(n21.75);
                \draw[scochain](m21.255) to node{}(n11.75);
            \end{tikzpicture}
            &
            \begin{tikzpicture}[scale=\scaleBB]
                \path 	
                node at (1,0) [root,draw,label=0: {$\alpha_1$}] (m10) {$  $}
                node at (-1,1) [root,draw,label=180: {$\alpha_2$}] (m01) {$  $}
                node at (0,1) [root,draw,label=90: {$ $}] (m11) {$  $}
                node at (1,1) [root,draw,label=0: {$ $}] (m21) {$  $}
                
                node at (-1,0) [root,draw,label=0: {$ $}] (n10) {$  $}
                node at (1,-1) [root,draw,label=180: {$ $}] (n01) {$  $}
                node at (0,-1) [root,draw,label=90: {$ $}] (n11) {$  $}
                node at (-1,-1) [root,draw,label=0: {$ $}] (n21) {$  $}
                ;
                \draw[scochain](m10.180) to node{}(n10.0);
                \draw[scochain](m01.315) to node{}(n01.135);
                \draw[scochain](m21.225) to node{}(n21.45);
                \draw[scochain](n10.-30) to node{}(n01.150);
                \draw[scochain](m01.-30) to node{}(m10.150);
                \draw[scochain](m21.210) to node{}(n10.30);
                \draw[scochain](m10.210) to node{}(n21.30);
            \end{tikzpicture}
            &\!\!\!
            \begin{tikzpicture}[scale=\scaleBB]
                \path 	
                node at (1,0) [root,draw,label=0: {$\alpha_1$}] (m10) {$  $}
                node at (-1,1) [root,draw,label=180: {$\alpha_2$}] (m01) {$  $}
                node at (0,1) [root,draw,label=90: {$ $}] (m11) {$  $}
                node at (1,1) [root,draw,label=0: {$ $}] (m21) {$  $}
                
                node at (-1,0) [root,draw,label=0: {$ $}] (n10) {$  $}
                node at (1,-1) [root,draw,label=180: {$ $}] (n01) {$  $}
                node at (0,-1) [root,draw,label=90: {$ $}] (n11) {$  $}
                node at (-1,-1) [root,draw,label=0: {$ $}] (n21) {$  $}
                ;
                \draw[scochain](m10.180) to node{}(n10.0);
                \draw[scochain](m01.315) to node{}(n01.135);
                \draw[scochain](m11.270) to node{}(n11.90);
                \draw[scochain](m10.135) to node{}(m11.-45);
                \draw[scochain](n10.45) to node{}(m11.225);
                \draw[scochain](n10.-45) to node{}(n11.135);
                \draw[scochain](m11.-60) to node{}(n01.120);
                \draw[scochain](n10.-30) to node{}(n01.150);
            \end{tikzpicture}
            &\!\!\!
            \begin{tikzpicture}[scale=\scaleBB]
                \path 	
                node at (1,0) [root,draw,label=0: {$\alpha_1$}] (m10) {$  $}
                node at (-1,1) [root,draw,label=180: {$\alpha_2$}] (m01) {$  $}
                node at (0,1) [root,draw,label=90: {$ $}] (m11) {$  $}
                node at (1,1) [root,draw,label=0: {$ $}] (m21) {$  $}
                
                node at (-1,0) [root,draw,label=0: {$ $}] (n10) {$  $}
                node at (1,-1) [root,draw,label=180: {$ $}] (n01) {$  $}
                node at (0,-1) [root,draw,label=90: {$ $}] (n11) {$  $}
                node at (-1,-1) [root,draw,label=0: {$ $}] (n21) {$  $}
                ;
                \draw[scochain](m10.180) to node{}(n10.0);
                \draw[scochain](m01.315) to node{}(n01.135);
                \draw[scochain](m11.270) to node{}(n11.90);
                \draw[scochain](m21.225) to node{}(n21.45);

                \draw[scochain](n10.45) to node{}(m11.225);
                \draw[scochain](n10.-45) to node{}(n11.135);
                \draw[scochain](m11.-60) to node{}(n01.120);
                \draw[scochain](n10.-30) to node{}(n01.150);
                \draw[scochain](n10.30) to node{}(m21.210);
                \draw[scochain](n11.60) to node{}(m21.240);
            \end{tikzpicture}
        \end{array}$
        \end{adjustbox}
    \end{center}
\end{figure}

\begin{figure}[ht!] 
    \caption[]{$2$-cocycles on $G_2$ from Dynkin gradings       
    }
    \label{fig:G2}
    \begin{center}
        $\begin{array}{c}
            \begin{tikzpicture}[scale=\dynkintablescale,baseline=(current bounding box.center), font=\dynkinfont,decoration={
                markings,
                    mark=at position 0.5 with {
            \draw (3pt,3pt) -- (0pt,0pt);
            \draw (0pt,-0pt) -- (3pt,-3pt);
            } }]
            \path 
            node at ( 1,0) [ldynkinnode,label=90: $0 $,label=270: $\alpha_1$] (one) {$ $}
            node at ( 2,0) [sdynkinnode,label=90: $2 $,label=270: $\alpha_2$] (two) {$ $}
            ;
            \draw[] (one) to node []{$ $} (two);
            \draw[] (one.45) to node []{$ $} (two.135);
            \draw[] (one.-45) to node []{$ $} (two.-135);
            \fill[postaction={decorate}] (one) -- (two);
            \end{tikzpicture}
            \\
            \begin{array}{cc}
                \nu=2&\nu\ge 3\\
                \begin{tikzpicture}[scale=\scaleGG]
                    \path 
                    node at (-1.5,.866) [root,draw,label=90: $\alpha_2$] (m10) {$ $}
                    node at (1,0) [root,draw,label=0: $\alpha_1$] (m01) {$ $}
                    node at (-.5,.866) [root,draw,label=90: $ $] (m11) {$ $}
                    node at (.5,.866) [root,draw,label=90: $ $] (m12) {$ $}
                    node at (1.5,.866) [root,draw,label=90: $ $] (m13) {$ $}
                    node at (0,2*.866) [root,draw,label=90: $ $] (m23) {$ $}
                    node at (1.5,-.866) [root,draw,label=270: $ $] (n10) {$ $}
                    node at (-1,0) [root,draw,label=180: $ $] (n01) {$ $}
                    node at (.5,-.866) [root,draw,label=270: $ $] (n11) {$ $}
                    node at (-.5,-.866) [root,draw,label=270: $ $] (n12) {$ $}
                    node at (-1.5,-.866) [root,draw,label=270: $ $] (n13) {$ $}
                    node at (0,-2*.866) [root,draw,label=270: $ $] (n23) {$ $}		
                    ;              
                    \draw[scochain] (m11) to node []{$ $} (n12);
                    \draw[scochain] (m12) to (n11);
                    
                    \draw[scochain] (m12) to node []{$ $} (m11);
                    \draw[scochain] (n11) to node []{$ $} (n12);
                    
                    \draw[scochain] (m10) to [bend left=\bend](m13);
                    \draw[scochain] (n10) to [bend left=\bend](n13);
                    
                    \draw[scochain]  (m13) to (n12);
                    \draw[scochain] (m12) to (n13);
                    \draw[scochain] (m11) to (n10);
                    \draw[scochain] (m10) to (n11);
                    
                    \draw[scochain] (m10) to (n10);
                    \draw[scochain] (m11) to (n11);
                    \draw[scochain] (m12) to (n12);
                    \draw[scochain] (m13) to (n13);
                \end{tikzpicture}
                &
                \begin{tikzpicture}[scale=\scaleGG]
                    \path 
                    node at (-1.5,.866) [root,draw,label=90: $\alpha_2$] (m10) {$ $}
                    node at (1,0) [root,draw,label=0: $\,\,\alpha_1$] (m01) {$ $}
                    node at (-.5,.866) [root,draw,label=90: $ $] (m11) {$ $}
                    node at (.5,.866) [root,draw,label=90: $ $] (m12) {$ $}
                    node at (1.5,.866) [root,draw,label=90: $ $] (m13) {$ $}
                    node at (0,2*.866) [root,draw,label=90: $ $] (m23) {$ $}
                    node at (1.5,-.866) [root,draw,label=270: $ $] (n10) {$ $}
                    node at (-1,0) [root,draw,label=180: $ $] (n01) {$ $}
                    node at (.5,-.866) [root,draw,label=270: $ $] (n11) {$ $}
                    node at (-.5,-.866) [root,draw,label=270: $ $] (n12) {$ $}
                    node at (-1.5,-.866) [root,draw,label=270: $ $] (n13) {$ $}
                    node at (0,-2*.866) [root,draw,label=270: $ $] (n23) {$ $}		
                    ;
                    \draw[scochain] (m11) to node []{$ $} (n12);
                    \draw[scochain] (m12) to (n11);
                    \draw[scochain] (n11) to node []{$ $} (n12);
    
                    \draw[scochain] (m23) to [bend right=\bend](n13);
                    \draw[scochain] (m23) to [bend left=\bend](n10);
                    \draw[scochain] (n10) to [bend left=\bend](n13);
    
                    \draw[scochain] (m10) to (n11);
                    \draw[scochain] (m23) to (n11);
                    \draw[scochain] (m23) to (n12);
                    \draw[scochain] (m13) to (n12);
                    \draw[scochain] (n10) to (m11);
                    \draw[scochain] (n13) to (m12);
    
                    \draw[scochain] (m10) to (n10);
                    \draw[scochain] (m11) to (n11);
                    \draw[scochain] (m12) to (n12);
                    \draw[scochain] (m13) to (n13);
                    \draw[scochain] (m23) to (n23);
    
                \end{tikzpicture}
            \end{array}
            \\
            \hline
            \\
            \begin{tikzpicture}[scale=\dynkintablescale,baseline=(current bounding box.center), font=\dynkinfont,decoration={
                markings,
                    mark=at position 0.5 with {
                \draw (3pt,3pt) -- (0pt,0pt);
                \draw (0pt,-0pt) -- (3pt,-3pt);
                } }]
                \path 
                node at ( 1,0) [ldynkinnode,label=90: $2 $,label=270: $\alpha_1$] (one) {$ $}
                node at ( 2,0) [sdynkinnode,label=90: $2 $,label=270: $\alpha_2$] (two) {$ $}
                ;
                \draw[] (one) to node []{$ $} (two);
                \draw[] (one.45) to node []{$ $} (two.135);
                \draw[] (one.-45) to node []{$ $} (two.-135);
                \fill[postaction={decorate}] (one) -- (two);
            \end{tikzpicture}
            \\\\
            \begin{array}{ccc}
                \nu=2&\nu=3&\nu=4
                \\
                \begin{tikzpicture}[scale=\scaleGG]
                    \path 
                    node at (-1.5,.866) [root,draw,label=90: $\alpha_2$] (six) {$ $}
                    node at (1,0) [root,draw,label=0: $\,\,\alpha_1$] (one) {$ $}
                    node at (-.5,.866) [root,draw,label=90: $ $] (five) {$ $}
                    node at (.5,.866) [root,draw,label=90: $ $] (three) {$ $}
                    node at (1.5,.866) [root,draw,label=90: $ $] (two) {$ $}
                    node at (0,2*.866) [root,draw,label=90: $ $] (four) {$ $}
                    node at (1.5,-.866) [root,draw,label=270: $ $] (twelve) {$ $}
                    node at (-1,0) [root,draw,label=180: $ $] (seven) {$ $}
                    node at (.5,-.866) [root,draw,label=270: $ $] (eleven) {$ $}
                    node at (-.5,-.866) [root,draw,label=270: $ $] (nine) {$ $}
                    node at (-1.5,-.866) [root,draw,label=270: $ $] (eight) {$ $}
                    node at (0,-2*.866) [root,draw,label=270: $ $] (ten) {$ $}		
                    ;          
                    \draw[scochain] (three) to node []{$ $} (seven);
                    \draw[scochain] (one) to (nine);
                    
                    \draw[scochain] (one) to node []{$ $} (three);
                    \draw[scochain] (nine) to node []{$ $} (seven);
                    
                    \draw[scochain] (four) to [bend left=\bend](twelve);
                    \draw[scochain] (ten) to [bend left=\bend](six);
                    
                    \draw[scochain]  (twelve) to (seven);
                    \draw[scochain] (one) to (six);
                    \draw[scochain] (three) to (ten);
                    \draw[scochain] (four) to (nine);
                    
                    \draw[scochain] (four) to (ten);
                    \draw[scochain] (three) to (nine);
                    \draw[scochain] (one) to (seven);
                    \draw[scochain] (twelve) to (six);
                \end{tikzpicture}
                &
                \begin{tikzpicture}[scale=\scaleGG]
                    \path 
                    node at (-1.5,.866) [root,draw,label=90: $\alpha_2$] (six) {$ $}
                    node at (1,0) [root,draw,label=0: $\,\,\alpha_1$] (one) {$ $}
                    node at (-.5,.866) [root,draw,label=90: $ $] (five) {$ $}
                    node at (.5,.866) [root,draw,label=90: $ $] (three) {$ $}
                    node at (1.5,.866) [root,draw,label=90: $ $] (two) {$ $}
                    node at (0,2*.866) [root,draw,label=90: $ $] (four) {$ $}
                    node at (1.5,-.866) [root,draw,label=270: $ $] (twelve) {$ $}
                    node at (-1,0) [root,draw,label=180: $ $] (seven) {$ $}
                    node at (.5,-.866) [root,draw,label=270: $ $] (eleven) {$ $}
                    node at (-.5,-.866) [root,draw,label=270: $ $] (nine) {$ $}
                    node at (-1.5,-.866) [root,draw,label=270: $ $] (eight) {$ $}
                    node at (0,-2*.866) [root,draw,label=270: $ $] (ten) {$ $}		
                    ;
                    \draw[scochain] (one) to node []{$ $} (five);
                    \draw[scochain] (eleven) to (seven);
                    \draw[scochain] (seven) to node []{$ $} (five);
                    
                    \draw[scochain] (twelve) to [bend right=\bend](four);
                    \draw[scochain] (twelve) to [bend left=\bend](eight);
                    \draw[scochain] (eight) to [bend left=\bend](four);
                    
                    \draw[scochain] (two) to (seven);
                    \draw[scochain] (twelve) to (seven);
                    \draw[scochain] (twelve) to (five);
                    \draw[scochain] (ten) to (five);
                    \draw[scochain] (eight) to (one);
                    \draw[scochain] (four) to (eleven);
                    
                    \draw[scochain] (two) to (eight);
                    \draw[scochain] (one) to (seven);
                    \draw[scochain] (eleven) to (five);
                    \draw[scochain] (ten) to (four);
                    \draw[scochain] (twelve) to (six);
                \end{tikzpicture}
                &
                \begin{tikzpicture}[scale=\scaleGG]
                    \path 
                    node at (-1.5,.866) [root,draw,label=90: $\alpha_2$] (six) {$ $}
                    node at (1,0) [root,draw,label=0: $\,\,\alpha_1$] (one) {$ $}
                    node at (-.5,.866) [root,draw,label=90: $ $] (five) {$ $}
                    node at (.5,.866) [root,draw,label=90: $ $] (three) {$ $}
                    node at (1.5,.866) [root,draw,label=90: $ $] (two) {$ $}
                    node at (0,2*.866) [root,draw,label=90: $ $] (four) {$ $}
                    node at (1.5,-.866) [root,draw,label=270: $ $] (twelve) {$ $}
                    node at (-1,0) [root,draw,label=180: $ $] (seven) {$ $}
                    node at (.5,-.866) [root,draw,label=270: $ $] (eleven) {$ $}
                    node at (-.5,-.866) [root,draw,label=270: $ $] (nine) {$ $}
                    node at (-1.5,-.866) [root,draw,label=270: $ $] (eight) {$ $}
                    node at (0,-2*.866) [root,draw,label=270: $ $] (ten) {$ $}		
                    ;
                    \draw[scochain] (one) to (three);
                    \draw[scochain] (three) to (five);
                    \draw[scochain] (three) to (seven);
                    \draw[scochain] (three) to (ten);
                    \draw[scochain] (three) to (eleven);
                    \draw[scochain] (four) to [bend left=\bend](twelve);
                    \draw[scochain] (five) to (seven);
                    \draw[scochain] (five) to (ten);
                    \draw[scochain] (five) to (twelve);
                    \draw[scochain] (six) to [bend right=\bend](ten);
                    \draw[scochain] (seven) to (nine);
                    \draw[scochain] (seven) to (eleven);
                    \draw[scochain] (seven) to (twelve);
                    
                    \draw[scochain] (three) to (nine);
                    \draw[scochain] (one) to (seven);
                    \draw[scochain] (eleven) to (five);
                    \draw[scochain] (ten) to (four);
                    \draw[scochain] (twelve) to (six);
                \end{tikzpicture}
            \end{array}
            \\\\
            \begin{array}{cc}
                \nu=5&\nu\ge 6
                \\
                \begin{tikzpicture}[scale=\scaleGG]
                    \path 
                    node at (-1.5,.866) [root,draw,label=90: $\alpha_2$] (six) {$ $}
                    node at (1,0) [root,draw,label=0: $\,\,\alpha_1$] (one) {$ $}
                    node at (-.5,.866) [root,draw,label=90: $ $] (five) {$ $}
                    node at (.5,.866) [root,draw,label=90: $ $] (three) {$ $}
                    node at (1.5,.866) [root,draw,label=90: $ $] (two) {$ $}
                    node at (0,2*.866) [root,draw,label=90: $ $] (four) {$ $}
                    node at (1.5,-.866) [root,draw,label=270: $ $] (twelve) {$ $}
                    node at (-1,0) [root,draw,label=180: $ $] (seven) {$ $}
                    node at (.5,-.866) [root,draw,label=270: $ $] (eleven) {$ $}
                    node at (-.5,-.866) [root,draw,label=270: $ $] (nine) {$ $}
                    node at (-1.5,-.866) [root,draw,label=270: $ $] (eight) {$ $}
                    node at (0,-2*.866) [root,draw,label=270: $ $] (ten) {$ $}		
                    ;
                    \draw[scochain] (two) to [bend right=\bend](six);
                    \draw[scochain] (two) to (seven);
                    \draw[scochain] (two) to (nine);
                    \draw[scochain] (three) to (five);
                    \draw[scochain] (three) to (seven);
                    \draw[scochain] (three) to (eleven);
                    \draw[scochain] (five) to (seven);
                    \draw[scochain] (five) to (twelve);
                    \draw[scochain] (seven) to (nine);
                    \draw[scochain] (seven) to (eleven);
                    \draw[scochain] (seven) to (twelve);
                    \draw[scochain] (eight) to [bend right=\bend](twelve);
                    \draw[scochain] (nine) to (eleven);
        
                    \draw[scochain] (one) to (seven);
                    \draw[scochain] (two) to (eight);
                    \draw[scochain] (three) to (nine);
                    \draw[scochain] (five) to (eleven);
                    \draw[scochain] (six) to (twelve);
                \end{tikzpicture}
                &
                \begin{tikzpicture}[scale=\scaleGG]
                    \path 
                    node at (-1.5,.866) [root,draw,label=90: $\alpha_2$] (six) {$ $}
                    node at (1,0) [root,draw,label=0: $\,\,\alpha_1$] (one) {$ $}
                    node at (-.5,.866) [root,draw,label=90: $ $] (five) {$ $}
                    node at (.5,.866) [root,draw,label=90: $ $] (three) {$ $}
                    node at (1.5,.866) [root,draw,label=90: $ $] (two) {$ $}
                    node at (0,2*.866) [root,draw,label=90: $ $] (four) {$ $}
                    node at (1.5,-.866) [root,draw,label=270: $ $] (twelve) {$ $}
                    node at (-1,0) [root,draw,label=180: $ $] (seven) {$ $}
                    node at (.5,-.866) [root,draw,label=270: $ $] (eleven) {$ $}
                    node at (-.5,-.866) [root,draw,label=270: $ $] (nine) {$ $}
                    node at (-1.5,-.866) [root,draw,label=270: $ $] (eight) {$ $}
                    node at (0,-2*.866) [root,draw,label=270: $ $] (ten) {$ $}		
                    ;
                    \draw[scochain] (two) to (seven);
                    \draw[scochain] (two) to (nine);
        
                    \draw[scochain] (three) to (seven);
                    \draw[scochain] (three) to (eleven);
        
                    \draw[scochain] (four) to [bend right=\bend](eight);
                    \draw[scochain] (four) to (nine);
                    \draw[scochain] (four) to (eleven);
                    \draw[scochain] (four) to [bend left=\bend](twelve);
        
                    \draw[scochain] (five) to (seven);
                    \draw[scochain] (five) to (twelve);
        
                    \draw[scochain] (seven) to (nine);
                    \draw[scochain] (seven) to (eleven);
                    \draw[scochain] (seven) to (twelve);
        
                    \draw[scochain] (eight) to [bend right=\bend](twelve);
        
                    \draw[scochain] (nine) to (eleven);
        
                    \draw[scochain] (one) to (seven);
                    \draw[scochain] (two) to (eight);
                    \draw[scochain] (three) to (nine);
                    \draw[scochain] (four) to (ten);
                    \draw[scochain] (five) to (eleven);
                    \draw[scochain] (six) to (twelve);
                \end{tikzpicture}
            \end{array}
        \end{array}$
    \end{center}

\end{figure}

One can construct a particular automorphic Lie algebra of rank $2$ from Figures \ref{fig:B2} and \ref{fig:G2} as follows. Pick a Lie type $X_2$ and $\PSL(2,\mb C)$ embedding by its labelled Dynkin diagram. Then pick a polyhedral group, and for all its orders $\nu_i$, consider the graph $(\roots,E_i)$ on the root system $\roots$ of type $X_2$ from the figures. Define $\omega^2_i(\alpha,\beta)=1$ if $\{\alpha,\beta\}$ is an edge in $E_i$, and $\omega^2_i(\alpha,\beta)=0$ otherwise. Now all structure constants can be found in Theorem \ref{thm:aliam}

\subsection{The inner regular classification of automorphic Lie algebras of type $A_N$}
\label{sec:A_N}

In this section, we understand an action to be a homomorphism $\rg\to\Aut{\mf{g}}$ from a polyhedral group $\rg$ into the automorphism group of a simple complex Lie algebra $\mf g$.

\subsubsection{Factorisable, regular and inner actions}
An action is called factorisable if it factors through a homomorphism $\PSL(2,\mb C)\to\Aut{\mf{g}}$, it is called regular if zero is the only invariant ($\mf{g}^\rg=\{0\}$), and it is called inner if $\rg$ is mapped into $\Aut{\mf{g}}^0$, the connected component of the identity.

Automorphic Lie algebras have been intensively studied for inner regular actions \cite{knibbeler2017higher}. There, it is explained that these actions on Lie algebras $\slnc[N+1]$ of type $A_N$ (also denoted $\mf g (A_N)$) are provided by irreducible representations of the binary polyhedral groups of dimension $N+1$, and the combinations of $\rg$ and $N$ which allow an inner regular action are those given in Table \ref{tab:BRG irrep dimensions}.
\begin{table}[ht]
    \caption{Inner regular actions on $\mf g(A_N)$}
    \label{tab:BRG irrep dimensions}
    \begin{center}
        \begin{tabular}{ll}
            $\rg$&$N$\\
            \hline
            $\cg{n}$ & -\\
            $\dg{n\ge 2}$&$1$\\
            $\tg$&$1,2$\\
            $\og$&$1,2,3$\\
            $\yg$&$1,2,3,4,5$
        \end{tabular}
    \end{center}
\end{table}

We will compare this collection of actions with the factorisable actions used in this paper, and find the following curious result. It tells us that all but one of the automorphic Lie algebras in \cite{knibbeler2017higher} are reproduced in this paper.
\begin{Proposition}
    \label{prop:one inner regular action not factorisable}
There is precisely one regular inner action on $\mf g(A_N)$ which is not factorisable. This action is provided by the four dimensional representation of the icosahedral group.
\end{Proposition}
We split the proof up into bite-size pieces with the three lemmas to follow. For their formulation we need the notion of the regular action of $\SLNC[2]$ on $\Aut{\mf g}$. Every simple Lie algebra $\mf g$ has an embedding $\PSL(2,\mb C)\to\Aut{\mf g}$ corresponding to the Dynkin diagram with label $2$ at each node. We call this the regular action of $\SLNC[2]$ on $\Aut{\mf g}$, due to the associated regular nilpotent orbit \cite{collingwood1993nilpotent}. If $\mf g=\mf g (A_N)$, it is defined by the irreducible representation of $\mf g (A_1)$ on $\mb C^n$. With the usual bases, the element $\diag(1,-1)$ acts on each simple root vector of $\mf g (A_N)$ with eigenvalue $2$. 
\begin{Lemma}
    \label{lem:regular action factors through regular representation of SL2}
A regular factorisable action on $\mf g(A_N)$ factors through the regular action of $\SLNC[2]$ on $\Aut{\mf g(A_N)}$. 
\end{Lemma}
\begin{proof}
    Being concerned with the dimension of $\mf g^\rg$ it is natural to compute the character of the action of $\rg$ on $\mf g$. But the construction we are using gives a direct way to find $\mf g^{\gamma_i}$ for each generator $\gamma_i$ of $\rg$, and therefore it is convenient to use the formula
    \begin{equation}
        \label{eq:dimension fixed points}
        \dim\mf g+2\dim\mf g^\rg=\sum \dim\mf g^{\langle \gamma_i\rangle}
    \end{equation}
    that can be found in \cite[Lemma 2.1]{knibbeler2019hereditary}.

    We will do a case by case analysis for all combinations of $\rg$ and $N$ that allow a regular action (cf.~Table \ref{tab:BRG irrep dimensions}). Consider a factorisable action $\rg\to\PSL(2,\mb C)\to\Aut{\mf g(A_N)}$ for one of these cases. We will show that if $\PSL(2,\mb C)\to\Aut{\mf g(A_N)}$ is not regular, then $\sum \dim\mf g(A_N)^{\langle \gamma_i\rangle}>\dim \mf g(A_N)=N(N+2)$. The proof is then finished using (\ref{eq:dimension fixed points}).

    The nonregular actions $\PSL(2,\mb C)\to\Aut{\mf g(A_N)}$ for $N\le 5$ are classified by the following eigenvalues of $\diag(1,-1)$ on the simple root vectors: 
    \begin{align*}
        &(2,0,2), \\
        &(0,2,0), \\
        &(2,0,0,2), \\
        &(2,2,0,2,2), \\
        &(2,0,2,0,2), \\
        &(0,2,0,2,0), \\
        &(2,0,0,0,2), \\
        &(0,0,2,0,0)
    \end{align*}
    as follows from \cite[Section 3.6]{collingwood1993nilpotent}.

    The group generator $\gamma_i$ acts trivially on a CSA of $\mf g$ and with eigenvalues $\ex_i\circ k$ on the roots relative to this CSA. Hence $\dim \mf g ^{\gamma_i}=\rank \mf g + |\ker \ex_i\circ k|$. In some cases, $\ker k$ is big enough to show immediately that $\sum \dim\mf g(A_N)^{\langle \gamma_i\rangle}\ge N(N+2)$ using only that $\ker \ex_i\circ k \supset \ker k$. For instance, $|\ker k_{(0,2,0)}|=4$, so that $\sum \dim\mf g(A_3)^{\langle \gamma_i\rangle}=\sum (\rank \mf g(A_3)+|\ker \ex_i\circ k_{(0,2,0)}|)\ge 3(\rank \mf g(A_3)+|\ker k_{(0,2,0)}|)=3(3+4)=21>15=\dim \mf g(A_5)$. This argument also suffices for $(2,0,0,2)$,  $(2,0,0,0,2)$ and $(0,0,2,0,0)$. The remaining cases rely on $\ex$, and we will write them out.

    The function $k$ on the roots corresponding to the eigenvalues on all root vectors can be presented in a matrix. For instance, 
    $$k_{(2,0,2)}=\begin{bmatrix}\ast&2&2&4\\-2&\ast&0&2\\-2&0&\ast&2\\-4&-2&-2&\ast\end{bmatrix}.$$
    With this notation we can write $\ex\circ k_{(2,0,2)}$ for $\rg=\tg, \og, \yg$ with $\nu_1=3,4,5$ respectively as
    \[
        \ex\circ k_{(2,0,2)}=
        \left(
            \begin{bmatrix}\ast&1&1&2\\\nu_1-1&\ast&0&1\\\nu_1-1&0&\ast&1\\\nu_1-2&\nu_1-1&\nu_1-1&\ast\end{bmatrix}, 
            \begin{bmatrix}\ast&1&1&2\\2&\ast&0&1\\2&0&\ast&1\\1&2&2&\ast\end{bmatrix}, 
            \begin{bmatrix}\ast&1&1&0\\1&\ast&0&1\\1&0&\ast&1\\0&1&1&\ast\end{bmatrix}
        \right),
    \]
    and we can read off that $\sum \dim\mf g(A_3)^{\langle \gamma_i\rangle}=5+5+7$, which is greater than $\dim\mf g(A_3)=15$. 

    The remaining cases, all of which have $\nu_i=(5,3,2)$, are checked similarly as follows. The eigenvalues $\ex\circ k_{(2,2,0,2,2)}$ are
    \[
        \left(
            \begin{bmatrix}
                \ast&1&2&2&3&4\\
                &\ast&1&1&2&3\\
                &&\ast&0&1&2\\
                &&&\ast&1&2\\
                &&&&\ast&1\\
                &&&&&\ast
            \end{bmatrix}, 
            \begin{bmatrix}
                \ast&1&2&2&0&1\\
                &\ast&1&1&2&0\\
                &&\ast&0&1&2\\
                &&&\ast&1&2\\
                &&&&\ast&1\\
                &&&&&\ast
            \end{bmatrix}, 
            \begin{bmatrix}
                \ast&1&0&0&1&0\\
                &\ast&1&1&0&1\\
                &&\ast&0&1&0\\
                &&&\ast&1&0\\
                &&&&\ast&1\\
                &&&&&\ast
            \end{bmatrix}
        \right).
    \]
    This shows that $\sum \dim\mf g(A_5)^{\langle \gamma_i\rangle}=7+11+19=37$, which is greater than $\dim\mf g(A_5)=35$. The eigenvalue $\ex\circ k_{(2,0,2,0,2)}$ are
    \[
        \left(
            \begin{bmatrix}
                \ast&1&1&2&2&3\\
                &\ast&0&1&1&2\\
                &&\ast&1&1&2\\
                &&&\ast&0&1\\
                &&&&\ast&1\\
                &&&&&\ast
            \end{bmatrix}, 
            \begin{bmatrix}
                \ast&1&1&2&2&0\\
                &\ast&0&1&1&2\\
                &&\ast&1&1&2\\
                &&&\ast&0&1\\
                &&&&\ast&1\\
                &&&&&\ast
            \end{bmatrix}, 
            \begin{bmatrix}
                \ast&1&1&0&0&1\\
                &\ast&0&1&1&0\\
                &&\ast&1&1&0\\
                &&&\ast&0&1\\
                &&&&\ast&1\\
                &&&&&\ast
            \end{bmatrix}
        \right).
    \]
    This shows that $\sum \dim\mf g(A_5)^{\langle \gamma_i\rangle}=9+11+17 =37$, also $>35$. Finally, the eigenvalues $\ex\circ k_{(0,2,0,2,0)}$ are
    \[
        \left(
            \begin{bmatrix}
                \ast&0&1&1&2&2\\
                &\ast&1&1&2&2\\
                &&\ast&0&1&1\\
                &&&\ast&1&1\\
                &&&&\ast&0\\
                &&&&&\ast
            \end{bmatrix}, 
            \begin{bmatrix}
                \ast&0&1&1&2&2\\
                &\ast&1&1&2&2\\
                &&\ast&0&1&1\\
                &&&\ast&1&1\\
                &&&&\ast&0\\
                &&&&&\ast
            \end{bmatrix}, 
            \begin{bmatrix}
                \ast&0&1&1&0&0\\
                &\ast&1&1&0&0\\
                &&\ast&0&1&1\\
                &&&\ast&1&1\\
                &&&&\ast&0\\
                &&&&&\ast
            \end{bmatrix}
        \right).
    \]
    This shows that $\sum \dim\mf g(A_5)^{\langle \gamma_i\rangle}=11+11+19=41>35$.
\end{proof}
Due to Lemma \ref{lem:regular action factors through regular representation of SL2} we know that all automorphic Lie algebras from regular factorisable actions on $\mf g (A_N)$ are given by a $2$-cocycle $\omega^2=\sf d \left(\frac{1}{\nu}\circ\ex\circ k\right)$ where $k(\alpha_i)=2$ for each simple root. That is, $k(\alpha)$ is twice the height of $\alpha$.
\begin{Lemma}
    \label{lem:regular action is spinorial iff n is even}
    The regular representation of $\diag(-1,-1)\in\SLNC[2]$ on $\mb C^{N+1}$ is $(-1)^{N}\Id$.
\end{Lemma}
\begin{proof}
    On the algebra level, the regular action sends $\diag(1,-1)$ to 
    \[
        \diag(N,N-2,\ldots,-N).
    \] The implication at the group level is that $\diag(-1,-1)$, which we can write as $\exp(\pi i \diag(1,-1))$, is sent to $\exp (\pi i \diag(N,N-2,\ldots,-N))=(-1)^{(N)}\Id$ (in the language of \cite{knibbeler2019hereditary}, the regular action of $\SLNC[2]$ on $\mb C^{N+1}$ produces only spinorial actions when $N$ is odd, and only nonspinorial actions when $N$ is even).
\end{proof}

\begin{Lemma}
    \label{lem:equivalent inner actions}
    Let $\rho_1:\rg\to\GL(V)$ and $\rho_2:\rg\to\GL(V)$ be representations with characters $\chi_1$ and $\chi_2$ respectively. Let $G=\Aut{\mf{sl}(V)}^0$ and define $\Ad:\GL(V)\to G$ by $\Ad(M)X=MXM^{-1}$.
    Then the representations 
    \[\Ad\circ \rho_i:\rg\to G,\quad i=1,2\] 
    are conjugate in $G$ if and only if $\chi_2=\chi\chi_1$ for some homomorphism $\chi:\rg\to\mb C^\ast$.
\end{Lemma}
\begin{proof}
    Suppose first that $\chi\in\Hom(\rg,\mb C^\ast)$ is such that $\chi_2=\chi\chi_1$. The representation $\rho_2$ is then equivalent to $\chi\rho_1$. That is, there exists $T\in\GL(V)$ such that 
    \[
        \rho_2(\gamma)=T \chi(\gamma)\rho_1(\gamma)T^{-1},\quad \forall \gamma\in\rg.
    \] 
    Applying $\Ad$ to this identity tells us that 
    \[
        \Ad(\rho_2(\gamma))=\Ad(T) \Ad(\rho_1(\gamma))\Ad(T)^{-1},\quad\forall\gamma\in\rg,
    \] 
    since $\Ad$ is a homomorphism with kernel consisting of all multiples of the identity. Given that $\Ad(T)\in G$, we have shown that $\Ad\circ \rho_1$ and $\Ad\circ \rho_2$ are indeed conjugate in $G$.

    Suppose now that $\Ad\circ \rho_i:\rg\to G$ for $i=1,2$ are conjugate in $G$, so that $\Ad(\rho_2(\gamma))=\Ad(T) \Ad(\rho_1(\gamma))\Ad(T)^{-1}$ for some $T\in\SL(V)$ and all $\gamma \in\rg$. Then 
    \[
        \Ad(\rho_2(\gamma)T\rho_1(\gamma)^{-1}T^{-1})=1,\quad\forall\gamma\in\rg,
    \] and therefore
    \[ 
        \rho_2(\gamma)T\rho_1(\gamma)^{-1}T^{-1}=\chi(\gamma)\id
    \] 
    for some map $\chi:\rg\to\mb C^\ast$, which is readily checked to be a homomorphism. Rewriting to $\rho_2(\gamma)=T\chi(\gamma)\rho_1(\gamma)T^{-1}$ shows that $\rho_2$ and $\chi\rho_1$ are equivalent representations, so that $\chi_2=\chi\chi_1$.
\end{proof}
\begin{Remark}
    The two $4$-dimensional irreducible representations of the binary icosahedral group provide an example of representations which are conjugate in $\GL(\mf{sl} (V))$ but not in $\Aut{\mf{sl}(V)}^0$. 
\end{Remark}

\begin{proof}[Proof of Proposition \ref{prop:one inner regular action not factorisable}]
The classification of inner regular actions on $\mf g(A_N)$ can be found in \cite[Tables 1-6]{knibbeler2017higher}. They are given by conjugation with irreducible representations of the binary polyhedral groups. When we apply Lemma \ref{lem:equivalent inner actions} we see that there is a unique regular inner action of the tetrahedral group at $N=1$ and $N=2$, and a unique regular inner action of the octahedral group at $N=1, 2$ and $3$. The icosahedral group has two regular inner actions at $N=1, 2, 3$, and unique regular inner actions at $N=4, 5$.

Using (\ref{eq:dimension fixed points}), we can check that the regular action of $\SLNC[2]$ on $\Aut{\mf g(A_N)}$ provides regular actions of the polyhedral groups, precisely at the dimensions we know them to exist. Moreover, no other actions of $\SLNC[2]$ on $\Aut{\mf g(A_N)}$ result in regular actions of polyhedral groups, due to Lemma \ref{lem:regular action factors through regular representation of SL2}.

Using the two nonconjugate embeddings of the binary icosahedral group in $\SLNC[2]$, we obtain the two inner regular actions of $\yg$ on $\mf g(A_1)$ and $\mf g(A_2)$, but at $N=3$ we only obtain one of the two. Hence, the other one is not factorisable.

There is precisely one even dimensional irreducible representation of a polyhedral group in the special linear group. This is the four dimensional irreducible representation of the icosahedral group. By Lemma \ref{lem:regular action is spinorial iff n is even}, this representation does not factor through the regular representation of $\SLNC[2]$, and we have identified its adjoint action as the unique nonfactorisable regular inner action.
\end{proof}

\subsubsection{Integrals of $2$-forms in normal form}
\label{sec:integrals of 2-forms in normal form}
When a Lie algebra is given by the brackets in Theorem \ref{thm:aliam} for some $2$-cocycle $\omega^2=\sf d \omega^1$, then we can think of $\omega^2$ as the descriptor of the Lie algebra and its integral $\omega^1$ as the descriptor of a concrete version for this Lie algebra, in the following sense. If $\{H_i, A_\alpha, i=1,\ldots,N, \alpha\in\roots\}$ is a concrete Chevalley basis of $\mf{g}$, then $\{H_i, \mb{I}^{\omega^1(\alpha)}A_\alpha, i=1,\ldots,N, \alpha\in\roots\}$ is a concrete basis of the automorphic Lie algebra in Theorem \ref{thm:aliam}. Besides concreteness, working with $\omega^1$ rather than $\omega^2$ has the additional advantage that the domain of the function is much smaller.

It is important to note in this context that $\sf d$ has a nontrivial kernel consisting of additive functions on the root system. Hence, if $\omega^1$ is an integral for $\omega^2$, then so is $\omega^1+L$ for any map $L$ on $\roots$ satisfying $L(\alpha+\beta)=L(\alpha)+L(\beta)$ for all $\alpha, \beta, \alpha+\beta\in\roots$. In order to work with this freedom, it is useful to define a normal form for the $1$-forms $\omega^1$. The freedom to choose an integral will also be used to avoid fractional powers of $\mb I$.

\begin{Definition}[Normal form of an integral of $\omega^2$]
    \label{def:normal form om1}
    Let $\omega^2$ be a $2$-cocycle on the root system of the form (\ref{eq:2cocycles})).
    A $1$-form $\omega^1$ on this root system is an integral for $\omega^2$ in normal form if 
    \begin{enumerate}[label=(\roman*)]
        \item $\sf d\omega^1=\omega^2$, 
        \item \label{it:integral natural values} $\omega^1(\alpha)\in\{0,1\}$ for all $\alpha\in\roots$, when the Lie type is $A_N$, and $\omega^1(\alpha)\in\{-1,0,1\}$ for other Lie types,
        \item the values $\omega^1(\alpha_1), \omega^1(\alpha_2), \ldots, \omega^1(\alpha_N)$ on the simple roots have minimal lexicographical order among integrals satisfying the first two conditions.
    \end{enumerate}
\end{Definition}
For the group $\rg=\cg{2n}$, the values of the integral are half the values stated in item \ref{it:integral natural values}. At the end of this subsection, we prove
\begin{Proposition}
    \label{prop:existence and uniqueness normal form}
 The normal form in Definition \ref{def:normal form om1} exists and is unique.
\end{Proposition}

\begin{Example}
    \label{ex:normal form}
    We revisit the automorphic Lie algebra of Example \ref{ex:explicit invariant matrices} and Example \ref{ex:graph} defined by an icosahedral action on $\slnc[3]$. This Lie algebra is given by the $2$ form $\omega^2=\sf d \theta$ where 
    \begin{align*}
        \theta(\alpha_1)&=(1/5, 1/3, 1/2)\\
        \theta(\alpha_2)&=(1/5, 1/3, 1/2)\\
        \theta(\alpha_1+\alpha_2)&=(2/5, 2/3, 0)\\
        \theta(-\alpha_1)&=(4/5, 2/3, 1/2)\\
        \theta(-\alpha_2)&=(4/5, 2/3, 1/2)\\
        \theta(-\alpha_1-\alpha_2)&=(3/5, 1/3, 0).\\
    \end{align*}
    If we add the additive map $L$ to $\theta$ defined by 
    \[
        L(\alpha_1)=(-1/5, -1/3, -1/2),\quad L(\alpha_2)=(-1/5, -1/3, 1/2),
    \]
    we obtain a normal form $\omega^1$ for $\omega^2$. To be explicit, we get 
    \begin{align*}
        \omega^1(\alpha_1)&=(0, 0, 0)\\
        \omega^1(\alpha_2)&=(0, 0, 1)\\
        \omega^1(\alpha_1+\alpha_2)&=(0, 0, 0)\\
        \omega^1(-\alpha_1)&=(1,1,1)\\
        \omega^1(-\alpha_2)&=(1,1,0)\\
        \omega^1(-\alpha_1-\alpha_2)&=(1, 1, 0).\\
    \end{align*}
\end{Example}
\begin{Notation}
    A normal form $\omega^1$ of type $A_N$ can conveniently be presented in a matrix
    \[
        \begin{bmatrix}
            \ast&\mb I^{\omega^1(\alpha_1)}&\mb I^{\omega^1(\alpha_1+\alpha_2)}&\ldots&\mb I^{\omega^1(\tilde{\alpha})}&\\
            \mb I^{\omega^1(-\alpha_1)}&\ast&\mb I^{\omega^1(\alpha_2)}&\ddots&\vdots\\
            \mb I^{\omega^1(-\alpha_1-\alpha_2)}&\mb I^{\omega^1(-\alpha_2)}&\ast&\ddots&\vdots\\
            \vdots&\ddots&\ddots&\ddots&\mb I^{\omega^1(\alpha_N)}\\
            \mb I^{\omega^1(-\tilde{\alpha})}&\vdots&\vdots&\mb I^{\omega^1(-\alpha_N)}&\ast
        \end{bmatrix}.
    \]
    The normal form of Example \ref{ex:normal form} is written as 
    \[\begin{bmatrix}
        \ast&1&1\\
        \mb I\mb J\mb K&\ast&\mb K\\
        \mb I\mb J & \mb I \mb J&\ast
    \end{bmatrix}.\]
\end{Notation}

\begin{Theorem}
    All automorphic Lie algebras of type $A_N$ defined by a factorisable regular inner action of a polyhedral group are described by a normal form in (\ref{eq:inner regular classification type A_N}).
\end{Theorem}
\begin{proof}
    This follows from Theorem \ref{thm:aliam} and Lemma \ref{lem:regular action factors through regular representation of SL2} and the classification of regular inner action found in \cite{knibbeler2017higher}.
\end{proof}

\begin{align}
    \begin{split}
        \label{eq:inner regular classification type A_N}
        &\begin{bmatrix}
            \ast&1\\\mb I\mb J\mb K&\ast
        \end{bmatrix},
        \quad
        \begin{bmatrix}
            \ast&1&1\\
            \mb I\mb J\mb K&\ast&\mb K\\
            \mb I\mb J & \mb I \mb J&\ast
        \end{bmatrix},
        \quad
        \begin{bmatrix}
            \ast&1 &1 &1 \\
            \mb I\mb J\mb K&\ast &\mb K &\mb J \\
            \mb I\mb J&\mb I\mb J &\ast &\mb J \\
            \mb I\mb K&\mb I &\mb I\mb K &\ast
        \end{bmatrix},
        \\
        &
        \begin{bmatrix}
            \ast&1 &1 &1 &1 \\
            \mb I\mb J\mb K&\ast &\mb K &\mb J &\mb K \\
            \mb I\mb J&\mb I\mb J &\ast &\mb J &\mb J \\
            \mb I\mb K&\mb I &\mb I\mb K &\ast &\mb K \\
            \mb I\mb J&\mb I &\mb I &\mb I\mb J &\ast
        \end{bmatrix},
        \quad
        \begin{bmatrix}
            \ast&1 &1 &1 &1 &1\\
            \mb I\mb J\mb K&\ast &\mb K &\mb J &\mb K &\mb I\\
            \mb I\mb J&\mb I\mb J &\ast &\mb J &\mb J &\mb I\\
            \mb I\mb K&\mb I &\mb I\mb K &\ast &\mb K &\mb I\\
            \mb I\mb J&\mb I &\mb I &\mb I\mb J &\ast &\mb I\\
            \mb J\mb K&\mb J &\mb K &\mb J &\mb J\mb K & \ast
        \end{bmatrix}
    \end{split}
\end{align}

\begin{Remark}
    The $2$-cocycles of the automorphic Lie algebras in (\ref{eq:inner regular classification type A_N}) also allow an integral which is invariant under the symmetry of the Dynkin diagram of type $A_N$, whenever $N$ is odd. For $N=3$ and $N=5$ these are
    \[
    \begin{bmatrix}
        \ast&1 &\mb J &1 \\
        \mb I\mb J\mb K&\ast &\mb J\mb K &\mb J \\
        \mb I&\mb I &\ast &1 \\
        \mb I\mb K&\mb I &\mb I\mb J\mb K &\ast
    \end{bmatrix},\quad
    \begin{bmatrix}
        \ast&1 &1 &\mb I &\mb I &1\\
        \mb I\mb J\mb K&\ast &\mb K &\mb I\mb J &\mb I\mb K &\mb I\\
        \mb I\mb J&\mb I\mb J &\ast &\mb I\mb J &\mb I\mb J &\mb I\\
        \mb K&1 &\mb K &\ast &\mb K &1\\
        \mb J&1 &1 &\mb I\mb J &\ast &1 \\
        \mb J\mb K&\mb J &\mb K &\mb I\mb J &\mb I\mb J\mb K & \ast
    \end{bmatrix}.
\]
\end{Remark}
The automorphic Lie algebras with representations (\ref{eq:inner regular classification type A_N}) are isomorphic to the automorphic Lie algebras presented in \cite{knibbeler2017higher} as follows from the classification theorems therein and the results of the present paper. However, we have no direct construction of these isomorphisms.

We finish this subsection with a proof of existence and uniqueness of the normal form. The existence question corresponds to integrability of $2$-cocycles on a root system over the natural numbers. This is a tricky question in general and formulated as an open problem in \cite{knibbeler2020cohomology}. In the proof below we restrict to the $2$-cocycles which are relevant to this paper and prove integrability over the natural numbers for type $A_N$. The remaining Lie types need the value $-1$ for integration.
\begin{proof}[Proof of Proposition \ref{prop:existence and uniqueness normal form}]
Uniqueness can be shown as follows. 
Let $\omega^1$ and $\tilde{\omega}^1$ be two normal forms for $\omega^2$. Then their sequence of values on simple roots is identical due to the uniqueness of the minimum in lexicographical order. That is $\omega^1-\tilde{\omega}^1$ is zero on simple roots. But this function is also additive since $\sf d(\omega^1-\tilde{\omega}^1)=0$. Hence, $\omega^1=\tilde{\omega}^1$.

To show existence of the normal form, we will use the geometry of affine Weyl groups, and refer to \cite{bourbaki2002lie} for the necessary background information.
Let $\omega^1$ be a $1$-form on a root system $\roots$ having values in $[0,1)$, and suppose the derivative $\sf d \omega^1$ takes integer values (such as (\ref{eq:2cocycles})). We aim to show existence of a normal form $\tilde{\omega}^1$ for $\sf d \omega^1$. Put $$\tilde{\omega}^1=\omega^1+L.$$ For type $A_N$ it is sufficient to find a map $L$ such that
\begin{enumerate}[label=(\roman*)]
    \item\label{it:L additive} $L(\alpha+\beta)=L(\alpha)+L(\beta)$ for all $\alpha,\beta,\alpha+\beta\in\roots$,
    \item\label{it:L integer on base} $\omega^1(\alpha_i)+L(\alpha_i)\in\mb Z$, for all $\alpha_i$ in some base of $\roots$,
    \item\label{it:L bounded} $-1<L(\alpha)<1$, $\alpha\in\roots$.
\end{enumerate}
Indeed, property \ref{it:L additive} ensures that $\sf d \tilde{\omega}^1=\sf d \omega^1$. Properties \ref{it:L additive} and \ref{it:L integer on base} together ensure that $\tilde{\omega}^1$ has integer values. To see this, note that $\sf d \tilde{\omega}^1$ has integer values. We could rephrase that as $\tilde{\omega}^1$ being linear modulo $\mb Z$. Likewise, property \ref{it:L integer on base} says $\tilde{\omega}^1$ is zero modulo $\mb Z$ on the simple roots. Hence, the $1$-form is zero modulo $\mb Z$ on all roots.
Finally, property \ref{it:L bounded}, together with the condition $\omega^1(\alpha)\in[0,1)$ , ensures that $-1<\tilde{\omega}^1(\alpha)<2$, hence $\tilde{\omega}^1$ can only take values $0$ or $1$. For the other Lie types, we only need to show that $-1\le L(\alpha)\le 1$ for all $\alpha\in\roots$.

The additivity of $L$, \ref{it:L additive}, implies that $L$ is the restriction of a linear map on the root space. That is, an element of the dual of the root space, which is isomorphic to the Cartan subalgebra (CSA). 

Property \ref{it:L integer on base} is satisfied for one fixed base $\alpha_1,\ldots,\alpha_N$ of $\roots$ precisely when $L$ is an element of the coweight lattice $P$ in this CSA, shifted by $\omega^1$. The coweight lattice is the integer span of fundamental coweights $\bar{w}_1, \ldots,\bar{w}_N$ defined relative to a base of $\roots$ by the condition that $\langle \alpha_i, \bar{w}_j\rangle$ equals $1$ if $i=j$ and $0$ otherwise. In summary, the first two properties are satisfied if
\[L\in P-\sum_{i=1}^N\omega^1(\alpha_i) \bar{w}_i.\]

The freedom to choose any base of $\roots$ means we can work modulo the action of the Weyl group $W$, which acts simply transitively on bases of $\roots$. The affine lattice is also invariant under shifts in $P$. Thus, if $D$ is a fundamental domain of the action of the extended affine Weyl group $W\ltimes P$ on the dual of the root space, there exists $L\in D$ satisfying property \ref{it:L additive} and property \ref{it:L integer on base}.

Now we fix a base and associated fundamental weights $\bar{w}_1, \ldots,\bar{w}_N$ and highest root $\tilde{\alpha}=n_1\alpha_1+\ldots+n_N\alpha_N$. Define the alcove $A$ as the open simplex with vertices 
\[
    0,\bar{w}_1/n_1,\ldots,\bar{w}_N/n_N,
\] or equivalently, all elements $x$ of the dual root space satisfying $\langle \alpha_i,x\rangle>0$ and $\langle\tilde{\alpha},x\rangle<1$. The closure of $A$ is known to be a fundamental domain of the affine Weyl group $W\ltimes Q$, where $Q$ is the integer span of the dual root system (\cite[Ch. VI, \S 2, Proposition 5]{bourbaki2002lie}). 
The quotient $\pi_1=P/Q$ acts on $A$, and a fundamental domain of this action is a fundamental domain of the action of the extended affine Weyl group on the root space. 

At this stage we know that there exists $L'$ in the intersection of the Weyl group orbit of the affine lattice $P-\sum_{i=1}^N\omega^1(\alpha_i) \bar{w}_i$ and the closure of the alcove $A$ which satisfies $-1\le \langle \alpha, L\rangle\le 1$ since $|\langle \alpha, L\rangle|\le \langle \tilde{\alpha},L\rangle$ for all $\alpha \in \roots$. This finishes the proof for all Lie types except $A_N$.
In the $A_N$ case, it remains to find an image $L$ of $L'$ under $\pi_1$ for which $\langle\tilde{\alpha},L\rangle<1$. This is the case for $L'$ itself unless $L'$ is contained in the hyperplane through the nonzero vertices $v_i=\bar{w}_i/n_i$, $i=1,\ldots,N$ of the alcove (the hyperplane $\langle\tilde{\alpha},x\rangle=1$).

The fundamental group $\pi_1$ at Lie type $A_N$ is isomorphic to $\zn{(N+1)}$. Because this group also permutes the vertices $0$ and $v_i$, $i=1,\ldots,N$, of the alcove, we see that the orbit of zero corresponds to all vertices. Suppose $L'$ is contained in the hyperplane through the nonzero vertices $v_1,\ldots,v_N$. Let $v_{i_1},\ldots,v_{i_M}$ be the vertices of a lowest dimensional simplex containing $L'$. Pick $p\in\pi_1$ such that $p v_{i_M}=0$ and put $L=pL'$. If $L$ was contained in the hyperplane through $v_1,\ldots,v_N$, then it would be contained in the simplex with vertices $pv_{i_1},\ldots,pv_{i_{M-1}}$, and $v_{i_1},\ldots,v_{i_{M-1}}$ would be vertices of a simplex containing $L'$, contradicting the minimality of the dimension of such a simplex. Therefore, $L$ is not contained in the hyperplane $\langle\tilde{\alpha},x\rangle=1$ so that $\langle \tilde{\alpha} ,L\rangle<1$ and the normal form exists.

For all Lie types other than $A_N$, we have $|\pi_1|<N+1$, so that we can conclude that there are vertices of the alcove which cannot be mapped to zero by $\pi_1$, and the above argument cannot be applied.
\end{proof}

\begin{Remark}
    On the Lie algebra of type $B_2$ there is a $2$-cocycle which cannot be integrated over the natural numbers, despite this cocycle only taking values $0$ and $1$. See \cite[Example 6.5]{knibbeler2020cohomology}.
    The cocycles for $G_2$, $\nu=2$ and $\nu=5$, shown in Figure \ref{fig:G2} are not integrable over the natural numbers either. This shows the necessity of the value $-1$ in Definition \ref{def:normal form om1}.
\end{Remark}

\subsection{Automorphic Lie algebras of exceptional Lie type}
\label{sec:exceptional aLias}
The structure of automorphic Lie algebras found in Theorem \ref{thm:aliam} has been presented by graphs in Section \ref{sec:graphs} which are convenient in rank $2$, and in Section \ref{sec:A_N} by normal forms which are particularly nice for type $A_N$. But one can also produce a complete list of Lie brackets. To this end, one can conveniently use SageMath \cite{sagemath} because the function $\epsilon$ describing simple Lie algebras is already incorporated in this software. In this section we present the automorphic Lie algebras of all exceptional Lie types with icosahedral symmetry group factoring though the regular embedding $\PSL(2,\mb C)\to \Aut{\mf g}$ by listing the structure constants
\[\epsilon(\alpha,\beta)\mb I^{\omega^2(\alpha,\beta)},\quad \alpha,\beta\in\roots, \alpha+\beta\in\roots\cup\{0\}\]
of a Chevalley normal form, in Tables \ref{tab:structure constants G_2}-\ref{tab:structure constants E_8}. For readability, we will write $\mb I_i$ as $\mb I, \mb J$ and $\mb K$ for the Hauptmoduln vanishing at the exceptional orbits of increasing size, as we did in the previous subsection. The ordering of the roots is copied from the methods in SageMath which were used: 
\[\texttt{RootSystem(['X',N]).root\_space().positive\_roots()}\]
The numbering of the simple roots is as in SageMath and Bourbaki \cite{bourbaki2002lie}, and shown in the Dynkin diagrams below\footnote{The numbering of simple roots used by Kac \cite{kac1990infinite} is different, and \cite{10.1093/imrn/rnab376} follows the numbering of Kac for Lie type $G_2$. In this paper we opted for smooth integration with SageMath over consistency with \cite{kac1990infinite,10.1093/imrn/rnab376}}.

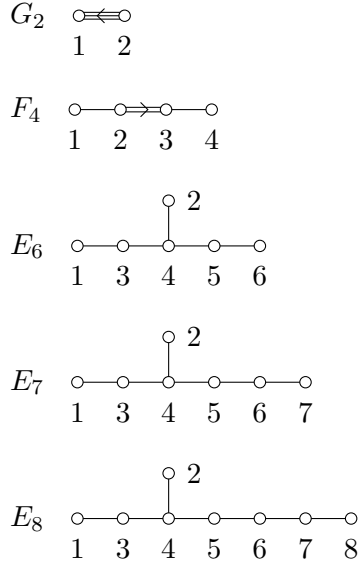
\newcommand{\sagedynkinscale}{0.3}
\begin{align*}
    &
    \begin{tikzpicture}[scale=\sagedynkinscale]
        \draw (-1,0) node[anchor=east] {$G_2$};
        \draw (0,0) -- (2 cm,0);
        \draw (0, 0.15 cm) -- +(2 cm,0);
        \draw (0, -0.15 cm) -- +(2 cm,0);
        \draw[shift={(0.8, 0)}, rotate=180] (135 : 0.45cm) -- (0,0) -- (-135 : 0.45cm);
        \draw[fill=white] (0 cm, 0 cm) circle (.25cm) node[below=4pt]{$1$};
        \draw[fill=white] (2 cm, 0 cm) circle (.25cm) node[below=4pt]{$2$};
    \end{tikzpicture}
    \\
    &
    \begin{tikzpicture}[scale=\sagedynkinscale]
        \draw (-1,0) node[anchor=east] {$F_4$};
        \draw (0 cm,0) -- (2 cm,0);
        \draw (2 cm, 0.1 cm) -- +(2 cm,0);
        \draw (2 cm, -0.1 cm) -- +(2 cm,0);
        \draw (4.0 cm,0) -- +(2 cm,0);
        \draw[shift={(3.2, 0)}, rotate=0] (135 : 0.45cm) -- (0,0) -- (-135 : 0.45cm);
        \draw[fill=white] (0 cm, 0 cm) circle (.25cm) node[below=4pt]{$1$};
        \draw[fill=white] (2 cm, 0 cm) circle (.25cm) node[below=4pt]{$2$};
        \draw[fill=white] (4 cm, 0 cm) circle (.25cm) node[below=4pt]{$3$};
        \draw[fill=white] (6 cm, 0 cm) circle (.25cm) node[below=4pt]{$4$};
    \end{tikzpicture}
    \\
    &
    \begin{tikzpicture}[scale=\sagedynkinscale]
        \draw (-1,0) node[anchor=east] {$E_6$};
        \draw (0 cm,0) -- (8 cm,0);
        \draw (4 cm, 0 cm) -- +(0,2 cm);
        \draw[fill=white] (0 cm, 0 cm) circle (.25cm) node[below=4pt]{$1$};
        \draw[fill=white] (2 cm, 0 cm) circle (.25cm) node[below=4pt]{$3$};
        \draw[fill=white] (4 cm, 0 cm) circle (.25cm) node[below=4pt]{$4$};
        \draw[fill=white] (6 cm, 0 cm) circle (.25cm) node[below=4pt]{$5$};
        \draw[fill=white] (8 cm, 0 cm) circle (.25cm) node[below=4pt]{$6$};
        \draw[fill=white] (4 cm, 2 cm) circle (.25cm) node[right=3pt]{$2$};
    \end{tikzpicture}
    \\
    &
    \begin{tikzpicture}[scale=\sagedynkinscale]
        \draw (-1,0) node[anchor=east] {$E_7$};
        \draw (0 cm,0) -- (10 cm,0);
        \draw (4 cm, 0 cm) -- +(0,2 cm);
        \draw[fill=white] (0 cm, 0 cm) circle (.25cm) node[below=4pt]{$1$};
        \draw[fill=white] (2 cm, 0 cm) circle (.25cm) node[below=4pt]{$3$};
        \draw[fill=white] (4 cm, 0 cm) circle (.25cm) node[below=4pt]{$4$};
        \draw[fill=white] (6 cm, 0 cm) circle (.25cm) node[below=4pt]{$5$};
        \draw[fill=white] (8 cm, 0 cm) circle (.25cm) node[below=4pt]{$6$};
        \draw[fill=white] (10 cm, 0 cm) circle (.25cm) node[below=4pt]{$7$};
        \draw[fill=white] (4 cm, 2 cm) circle (.25cm) node[right=3pt]{$2$};
    \end{tikzpicture}
    \\
    &
    \begin{tikzpicture}[scale=\sagedynkinscale]
        \draw (-1,0) node[anchor=east] {$E_8$};
        \draw (0 cm,0) -- (12 cm,0);
        \draw (4 cm, 0 cm) -- +(0,2 cm);
        \draw[fill=white] (0 cm, 0 cm) circle (.25cm) node[below=4pt]{$1$};
        \draw[fill=white] (2 cm, 0 cm) circle (.25cm) node[below=4pt]{$3$};
        \draw[fill=white] (4 cm, 0 cm) circle (.25cm) node[below=4pt]{$4$};
        \draw[fill=white] (6 cm, 0 cm) circle (.25cm) node[below=4pt]{$5$};
        \draw[fill=white] (8 cm, 0 cm) circle (.25cm) node[below=4pt]{$6$};
        \draw[fill=white] (10 cm, 0 cm) circle (.25cm) node[below=4pt]{$7$};
        \draw[fill=white] (12 cm, 0 cm) circle (.25cm) node[below=4pt]{$8$};
        \draw[fill=white] (4 cm, 2 cm) circle (.25cm) node[right=3pt]{$2$};
    \end{tikzpicture}
\end{align*}

Table \ref{tab:structure constants E_8} is too large to be practically usable, but we include it to show that it can be computed with the approach of this paper, and while doing so, a digital version of the table is created which is much more usable than a printed version.

Due to the large dimensions at Lie type $F_4$, $E_6$, $E_7$ and $E_8$, we present these structure constants for positive roots $\alpha$ and $\beta$ only (and for $\beta=-\alpha$), just as in the work of Vavilov \cite{vavilov2004yourself} where the structure constants $\epsilon$ for type $E_N$ are determined. The remaining structure constants can be deduced from these, using the following known result.
\begin{Lemma}
    \label{lem:transformation on epsilon and omega^2}
    Let $\epsilon$ describe a Chevalley basis of a simple Lie algebra with root system $\roots$ and let $\omega^2$ be a symmetric $2$-cocycle on $\roots$. If $\alpha,\beta,\alpha+\beta\in \roots$ then
    \begin{align*}
        \epsilon(\alpha+\beta,-\alpha)&=-\frac{|\beta|^2}{|\alpha+\beta|^2}\epsilon(\alpha,\beta),\\
        \omega^2(\alpha+\beta,-\alpha)&=\omega^2(\alpha,-\alpha)-\omega^2(\alpha,\beta).
    \end{align*}
\end{Lemma}
\begin{proof}
    The statement concerning $\epsilon$ can be deduced from \cite[Section 4]{vavilov2004yourself} or from \cite{humphreys1972introduction}. First one can show that \[\epsilon(-\alpha,-\beta)=-\epsilon(\alpha,\beta)\] using the isomorphism $H_i\mapsto -H_i$, $A_\alpha\mapsto -A_{-\alpha}$ of the simple Lie algebra (cf.~\cite[Proposition 25.2]{humphreys1972introduction}). Then one can use the Jacobi identity for roots $\alpha, \beta, \gamma$ with $\alpha+\beta+\gamma=0$ to find \[\frac{\epsilon(\alpha,\beta)}{|\gamma|^2}=\frac{\epsilon(\beta,\gamma)}{|\alpha|^2}=\frac{\epsilon(\gamma,\alpha)}{|\beta|^2}\]
    and combine these identities to get 
    \[
        \epsilon(\alpha+\beta,-\alpha)=\epsilon(-\gamma,-\alpha)=-\epsilon(\gamma,\alpha)=\frac{|\beta|^2}{|\alpha+\beta|^2}\epsilon(\alpha,\beta).
    \]

    The statement concerning $\omega^2$ can be found in \cite[Lemma 6.2]{knibbeler2020cohomology} and is obtained from the cocycle identity for the $3$-chain $(\alpha+\beta, -\alpha, \alpha)$ (which also stems from the Jacobi identity) and the symmetry $\omega^2(\gamma,\delta)=\omega^2(\delta,\gamma)$.
\end{proof}
The assignment $(\alpha,\beta)\mapsto(\alpha+\beta,-\alpha)$ defines a map on the set of pairs of roots which add to a root, and this map has order $6$. Each orbit of this map has precisely one element consisting of two positive roots. The values of $\epsilon$ and $\omega^2$ on this pair determine all other values on the orbit due to Lemma \ref{lem:transformation on epsilon and omega^2}. They are listed in Table \ref{tab:values of epsilon and omega^2 on orbit} for convenience.
\begin{table}[ht]
    \caption{Structure constants on the orbit of $(\alpha,\beta)\mapsto(\alpha+\beta,-\alpha)$}
    \label{tab:values of epsilon and omega^2 on orbit}
    \begin{center}
\begin{tabular}{lll}
            $(\gamma,\delta)$&$\epsilon(\gamma,\delta)$&$\omega^2(\gamma,\delta)$\\[2mm]
            \hline\\[-3mm]
            $(\alpha,\beta)$&$\epsilon(\alpha,\beta)$&$\omega^2(\alpha,\beta)$\\[2mm]
            $(\alpha+\beta,-\alpha)$&$-\frac{|\beta|^2}{|\alpha+\beta|^2}\epsilon(\alpha,\beta)$&$\omega^2(\alpha,-\alpha)-\omega^2(\alpha,\beta)$\\[2mm]
            $(\beta,-\alpha-\beta)$&$\frac{|\alpha|^2}{|\alpha+\beta|^2}\epsilon(\alpha,\beta)$&$\omega^2(\alpha+\beta,-\alpha-\beta)-\omega^2(\alpha+\beta,-\alpha)$\\[2mm]
            $(-\alpha,-\beta)$&$-\epsilon(\alpha,\beta)$&$\omega^2(\beta,-\beta)-\omega^2(\beta,-\alpha-\beta)$\\[2mm]
            $(-\alpha-\beta,\alpha)$&$\frac{|\beta|^2}{|\alpha+\beta|^2}\epsilon(\alpha,\beta)$&$\omega^2(\alpha,-\alpha)-\omega^2(-\alpha,-\beta)$\\[2mm]
            $(-\beta,\alpha+\beta)$&$-\frac{|\alpha|^2}{|\alpha+\beta|^2}\epsilon(\alpha,\beta)$&$\omega^2(\alpha+\beta,-\alpha-\beta)-\omega^2(-\alpha-\beta,\alpha)$
        \end{tabular}
    \end{center}
\end{table}
\begin{Example}
    We revisit the automorphic Lie algebra of type $A_2$ with icosahedral symmetry, known well now due to Examples \ref{ex:explicit invariant matrices}, \ref{ex:graph} and \ref{ex:normal form}.
    The structure constants $\epsilon(\alpha,\beta)\mb I^{\omega^2(\alpha,\beta)}$ for pairs of positive roots are given by the table 
    \[
        \begin{array}{rr|rrr}
            &  & 1 & 0 & 1 \\
            &  & 0 & 1 & 1 \\
            \hline
            &  & \mathit{\mb I} \mathit{\mb J} \mathit{\mb K} & \mathit{\mb I} \mathit{\mb J} \mathit{\mb K} & \mathit{\mb I} \mathit{\mb J} \\
            \hline
            1 & 0 & 0 & -\mathit{\mb K} & 0 \\
            0 & 1 & \mathit{\mb K} & 0 & 0 \\
            1 & 1 & 0 & 0 & 0
        \end{array}
    \]
    The first row and first column list the positive roots. The second row  shows the values at opposite roots. 

    In order to obtain the remaining structure constants, we fill in Table \ref{tab:values of epsilon and omega^2 on orbit} for this particular case, with $\alpha$ and $\beta$ corresponding to the simple roots $(1,0)$ and $(0,1)$ respectively, and obtain
    \begin{center}
        \begin{tabular}{lll}
            $(\gamma,\delta)$&$\epsilon(\gamma,\delta)$&$\omega^2(\gamma,\delta)$\\[2mm]
            \hline\\[-3mm]
            $(\alpha,\beta)$&$1$&$(0,0,1)$\\[2mm]
            $(\alpha+\beta,-\alpha)$&$-1$&$(1,1,1)-(0,0,1)=(1,1,0)$\\[2mm]
            $(\beta,-\alpha-\beta)$&$1$&$(1,1,0)-(1,1,0)=(0,0,0)$\\[2mm]
            $(-\alpha,-\beta)$&$-1$&$(1,1,1)-(0,0,0)=(1,1,1)$\\[2mm]
            $(-\alpha-\beta,\alpha)$&$1$&$(1,1,1)-(1,1,1)=(0,0,0)$\\[2mm]
            $(-\beta,\alpha+\beta)$&$-1$&$(1,1,0)-(0,0,0)=(1,1,0)$
        \end{tabular}
    \end{center}
    The column on the right displays the same information as the graph in Example \ref{ex:graph}.
\end{Example}

\begin{table}[ht]
    \caption{Structure constants for an automorphic Lie algebra of type $G_2$}
    \label{tab:structure constants G_2}
    \vspace{4mm}
    \begin{adjustbox}{width=\linewidth}
    \rowcolors{3}{lightgray}{}
    $
        \begin{array}{rr|rrrrrrrrrrrr}
             &  & 1 & 0 & 1 & 3 & 2 & 3 & -1 & 0 & -1 & -3 & -2 & -3 \\
             &  & 0 & 1 & 1 & 1 & 1 & 2 & 0 & -1 & -1 & -1 & -1 & -2 \\
            \hline
             1 & 0 & 0 & -\mb K & -2 \, \mb J & 0 & -3 \, \mb K & 0 & \mb I \mb J \mb K & 0 & 3 & \mb J & 2 \, \mb K & 0 \\
            0 & 1 & \mb K & 0 & 0 & -\mb I & 0 & 0 & 0 & \mb I \mb J \mb K & -1 & 0 & 0 & \mb K \\
            1 & 1 & 2 \, \mb J & 0 & 0 & 0 & 3 \, \mb I & 0 & 3 \, \mb I \mb J & -\mb I \mb J & \mb I \mb J & 0 & -2 & -\mb J \\
            3 & 1 & 0 & \mb I & 0 & 0 & 0 & 0 & \mb I \mb J & 0 & 0 & \mb I \mb J & -\mb I & -1 \\
            2 & 1 & 3 \, \mb K & 0 & -3 \, \mb I & 0 & 0 & 0 & 2 \, \mb I \mb K & 0 & -2 \, \mb I & -1 & \mb I \mb K & \mb K \\
            3 & 2 & 0 & 0 & 0 & 0 & 0 & 0 & 0 & \mb J \mb K & -\mb J & -\mb J & \mb K & \mb J \mb K \\
            -1 & 0 & -\mb I \mb J \mb K & 0 & -3 \, \mb I \mb J & -\mb I \mb J & -2 \, \mb I \mb K & 0 & 0 & \mb I \mb J \mb K & 2 \, \mb I \mb J & 0 & 3 \, \mb I \mb K & 0 \\
            0 & -1 & 0 & -\mb I \mb J \mb K & \mb I \mb J & 0 & 0 & -\mb J \mb K & -\mb I \mb J \mb K & 0 & 0 & \mb I \mb J & 0 & 0 \\
            -1 & -1 & -3 & 1 & -\mb I \mb J & 0 & 2 \, \mb I & \mb J & -2 \, \mb I \mb J & 0 & 0 & 0 & -3 \, \mb I & 0 \\
            -3 & -1 & -\mb J & 0 & 0 & -\mb I \mb J & 1 & \mb J & 0 & -\mb I \mb J & 0 & 0 & 0 & 0 \\
            -2 & -1 & -2 \, \mb K & 0 & 2 & \mb I & -\mb I \mb K & -\mb K & -3 \, \mb I \mb K & 0 & 3 \, \mb I & 0 & 0 & 0 \\
            -3 & -2 & 0 & -\mb K & \mb J & 1 & -\mb K & -\mb J \mb K & 0 & 0 & 0 & 0 & 0 & 0
        \end{array}
    $
    \end{adjustbox}
\end{table}

\ifExceptional

\afterpage{
    % \clearpage% Flush earlier floats (otherwise order might not be correct)
    \begin{landscape}
        \centering 
        \captionof{table}{Structure constants for an automorphic Lie algebra of type $F_4$}
        \vspace{7mm}
        \label{tab:structure constants F_4}
        \begin{adjustbox}{width=\linewidth}
        \rowcolors{6}{lightgray}{}
            $
                \begin{array}{rrrr|rrrrrrrrrrrrrrrrrrrrrrrr}
                &  &  &  & 1 & 0 & 0 & 0 & 1 & 0 & 0 & 0 & 1 & 0 & 1 & 0 & 1 & 1 & 1 & 0 & 1 & 1 & 1 & 1 & 1 & 1 & 2 & 1 \\
                &  &  &  & 0 & 1 & 0 & 0 & 1 & 1 & 1 & 0 & 1 & 1 & 1 & 1 & 2 & 1 & 1 & 1 & 2 & 1 & 2 & 2 & 3 & 2 & 3 & 2 \\
                &  &  &  & 0 & 0 & 1 & 0 & 0 & 2 & 1 & 1 & 2 & 2 & 1 & 1 & 2 & 2 & 1 & 2 & 2 & 2 & 4 & 2 & 4 & 3 & 4 & 3 \\
                &  &  &  & 0 & 0 & 0 & 1 & 0 & 0 & 0 & 1 & 0 & 2 & 0 & 1 & 0 & 2 & 1 & 1 & 2 & 1 & 2 & 1 & 2 & 1 & 2 & 2 \\
                \hline
                &  &  &  & \mb I \mb J \mb K & \mb I \mb J \mb K & \mb I \mb J \mb K & \mb I \mb J \mb K & \mb I \mb J & \mb I \mb K & \mb I \mb J & \mb I \mb J & \mb I \mb J & \mb J \mb K & \mb I \mb K & \mb I \mb K & \mb J \mb K & \mb I & \mb I \mb J & \mb I \mb J & \mb I \mb J \mb K & \mb J \mb K & \mb I \mb K & \mb I & \mb J & \mb I \mb J \mb K & \mb I \mb J \mb K & \mb I \mb J \\
                \hline
                1 & 0 & 0 & 0 & 0 & -\mb K & 0 & 0 & 0 & \mb K & \mb J & 0 & 0 & \mb J \mb K & 0 & \mb K & 0 & 0 & 0 & \mb I & 0 & 0 & 0 & 0 & -1 & 0 & 0 & 0 \\
                0 & 1 & 0 & 0 & \mb K & 0 & -\mb K & 0 & 0 & 0 & 0 & \mb J & -\mb I & 0 & 0 & 0 & 0 & -1 & 0 & 0 & 0 & -\mb J \mb K & -\mb I \mb K & 0 & 0 & 0 & 0 & 0 \\
                0 & 0 & 1 & 0 & 0 & \mb K & 0 & -\mb K & -\mb J & 0 & -2 \, \mb J & 0 & 0 & 0 & -2 \, \mb K & -\mb K & 0 & 0 & -\mb I & 0 & -\mb K & 0 & 0 & -1 & 0 & 0 & 0 & -2 \, \mb J \\
\end{array}
            $
        
    \end{adjustbox}
    \end{landscape}
    % \clearpage% Flush page
}

\afterpage{
    \clearpage% Flush earlier floats (otherwise order might not be correct)
    \begin{landscape}
        \centering 
        \captionof{table}{Structure constants for an automorphic Lie algebra of type $E_6$}
        \vspace{7mm}
        \label{tab:structure constants E_6}
        \begin{adjustbox}{width=\linewidth}
        \rowcolors{8}{lightgray}{}
            $
            \begin{array}{rrrrrr|rrrrrrrrrrrrrrrrrrrrrrrrrrrrrrrrrrrr} &  &  &  &  &  & 1 & 0 & 0 & 0 & 0 & 0 & 1 & 0 & 0 & 0 & 0 & 1 & 0 & 0 & 0 & 0 & 1 & 1 & 0 & 0 & 0 & 1 & 1 & 0 & 0 & 1 & 1 & 0 & 1 & 1 & 0 & 1 & 1 & 1 & 1 & 1 \\
                &  &  &  &  &  & 0 & 1 & 0 & 0 & 0 & 0 & 0 & 1 & 0 & 0 & 0 & 0 & 1 & 1 & 0 & 0 & 1 & 0 & 1 & 1 & 0 & 1 & 0 & 1 & 1 & 1 & 1 & 1 & 1 & 1 & 1 & 1 & 1 & 1 & 1 & 2 \\
                &  &  &  &  &  & 0 & 0 & 1 & 0 & 0 & 0 & 1 & 0 & 1 & 0 & 0 & 1 & 1 & 0 & 1 & 0 & 1 & 1 & 1 & 0 & 1 & 1 & 1 & 1 & 1 & 1 & 1 & 1 & 2 & 1 & 1 & 2 & 1 & 2 & 2 & 2 \\
                &  &  &  &  &  & 0 & 0 & 0 & 1 & 0 & 0 & 0 & 1 & 1 & 1 & 0 & 1 & 1 & 1 & 1 & 1 & 1 & 1 & 1 & 1 & 1 & 1 & 1 & 2 & 1 & 2 & 1 & 2 & 2 & 2 & 2 & 2 & 2 & 2 & 3 & 3 \\
                &  &  &  &  &  & 0 & 0 & 0 & 0 & 1 & 0 & 0 & 0 & 0 & 1 & 1 & 0 & 0 & 1 & 1 & 1 & 0 & 1 & 1 & 1 & 1 & 1 & 1 & 1 & 1 & 1 & 1 & 1 & 1 & 1 & 2 & 1 & 2 & 2 & 2 & 2 \\
                &  &  &  &  &  & 0 & 0 & 0 & 0 & 0 & 1 & 0 & 0 & 0 & 0 & 1 & 0 & 0 & 0 & 0 & 1 & 0 & 0 & 0 & 1 & 1 & 0 & 1 & 0 & 1 & 0 & 1 & 1 & 0 & 1 & 1 & 1 & 1 & 1 & 1 & 1 \\
                \hline
                    &  &  &  &  &  & \mb{I} \mb{J} \mb{K} & \mb{I} \mb{J} \mb{K} & \mb{I} \mb{J} \mb{K} & \mb{I} \mb{J} \mb{K} & \mb{I} \mb{J} \mb{K} & \mb{I} \mb{J} \mb{K} & \mb{I} \mb{J} & \mb{I} \mb{J} & \mb{I} \mb{J} & \mb{I} \mb{J} & \mb{I} \mb{J} & \mb{I} \mb{K} & \mb{I} \mb{K} & \mb{I} \mb{K} & \mb{I} \mb{K} & \mb{I} \mb{K} & \mb{I} \mb{J} & \mb{I} \mb{J} & \mb{I} \mb{J} & \mb{I} \mb{J} & \mb{I} \mb{J} & \mb{J} \mb{K} & \mb{J} \mb{K} & \mb{J} \mb{K} & \mb{J} \mb{K} & \mb{I} & \mb{I} & \mb{I} & \mb{I} \mb{J} \mb{K} & \mb{I} \mb{J} \mb{K} & \mb{I} \mb{J} \mb{K} & \mb{I} \mb{J} & \mb{I} \mb{J} & \mb{I} \mb{K} & \mb{J} & \mb{I} \mb{J} \mb{K} \\
\end{array}
            $
        \end{adjustbox}
    \end{landscape}
    \clearpage% Flush page
}

\afterpage{
    \clearpage% Flush earlier floats (otherwise order might not be correct)
    \begin{landscape}
        \centering 
        \captionof{table}{Structure constants for an automorphic Lie algebra of type $E_7$}
        \vspace{7mm}
        \label{tab:structure constants E_7}
        \begin{adjustbox}{width=\linewidth}
        \rowcolors{9}{lightgray}{}
            $
            \begin{array}{rrrrrrr|rrrrrrrrrrrrrrrrrrrrrrrrrrrrrrrrrrrrrrrrrrrrrrrrrrrrrrrrrrrrrrr}
            &  &  &  &  &  &  & 1 & 0 & 0 & 0 & 0 & 0 & 0 & 1 & 0 & 0 & 0 & 0 & 0 & 1 & 0 & 0 & 0 & 0 & 0 & 1 & 1 & 0 & 0 & 0 & 0 & 1 & 1 & 0 & 0 & 0 & 0 & 1 & 1 & 1 & 0 & 0 & 1 & 1 & 1 & 0 & 0 & 1 & 1 & 1 & 0 & 1 & 1 & 1 & 0 & 1 & 1 & 1 & 1 & 1 & 1 & 1 & 1 & 1 & 1 & 1 & 1 & 1 & 2 \\
            &  &  &  &  &  &  & 0 & 1 & 0 & 0 & 0 & 0 & 0 & 0 & 1 & 0 & 0 & 0 & 0 & 0 & 1 & 1 & 0 & 0 & 0 & 1 & 0 & 1 & 1 & 0 & 0 & 1 & 0 & 1 & 1 & 1 & 0 & 1 & 1 & 0 & 1 & 1 & 1 & 1 & 1 & 1 & 1 & 1 & 1 & 1 & 1 & 1 & 1 & 1 & 1 & 1 & 1 & 1 & 2 & 1 & 1 & 2 & 1 & 2 & 1 & 2 & 2 & 2 & 2 \\
            &  &  &  &  &  &  & 0 & 0 & 1 & 0 & 0 & 0 & 0 & 1 & 0 & 1 & 0 & 0 & 0 & 1 & 1 & 0 & 1 & 0 & 0 & 1 & 1 & 1 & 0 & 1 & 0 & 1 & 1 & 1 & 1 & 0 & 1 & 1 & 1 & 1 & 1 & 1 & 2 & 1 & 1 & 1 & 1 & 2 & 1 & 1 & 1 & 2 & 2 & 1 & 1 & 2 & 2 & 1 & 2 & 2 & 2 & 2 & 2 & 2 & 2 & 2 & 2 & 3 & 3 \\
            &  &  &  &  &  &  & 0 & 0 & 0 & 1 & 0 & 0 & 0 & 0 & 1 & 1 & 1 & 0 & 0 & 1 & 1 & 1 & 1 & 1 & 0 & 1 & 1 & 1 & 1 & 1 & 1 & 1 & 1 & 2 & 1 & 1 & 1 & 2 & 1 & 1 & 2 & 1 & 2 & 2 & 1 & 2 & 2 & 2 & 2 & 2 & 2 & 2 & 2 & 2 & 2 & 3 & 2 & 2 & 3 & 3 & 2 & 3 & 3 & 3 & 3 & 3 & 4 & 4 & 4 \\
            &  &  &  &  &  &  & 0 & 0 & 0 & 0 & 1 & 0 & 0 & 0 & 0 & 0 & 1 & 1 & 0 & 0 & 0 & 1 & 1 & 1 & 1 & 0 & 1 & 1 & 1 & 1 & 1 & 1 & 1 & 1 & 1 & 1 & 1 & 1 & 1 & 1 & 1 & 1 & 1 & 1 & 1 & 2 & 1 & 1 & 2 & 1 & 2 & 2 & 1 & 2 & 2 & 2 & 2 & 2 & 2 & 2 & 2 & 2 & 2 & 2 & 3 & 3 & 3 & 3 & 3 \\
            &  &  &  &  &  &  & 0 & 0 & 0 & 0 & 0 & 1 & 0 & 0 & 0 & 0 & 0 & 1 & 1 & 0 & 0 & 0 & 0 & 1 & 1 & 0 & 0 & 0 & 1 & 1 & 1 & 0 & 1 & 0 & 1 & 1 & 1 & 0 & 1 & 1 & 1 & 1 & 0 & 1 & 1 & 1 & 1 & 1 & 1 & 1 & 1 & 1 & 1 & 1 & 2 & 1 & 1 & 2 & 1 & 1 & 2 & 1 & 2 & 2 & 2 & 2 & 2 & 2 & 2 \\
\end{array}
        $
        
    \end{adjustbox}
    \end{landscape}
    \clearpage% Flush page
}

\afterpage{
    \clearpage% Flush earlier floats (otherwise order might not be correct)
    \begin{landscape}
        \centering 
        \captionof{table}{Structure constants for an automorphic Lie algebra of type $E_8$}
        \vspace{7mm}
        \label{tab:structure constants E_8}
        \begin{adjustbox}{width=\linewidth}
        \rowcolors{10}{lightgray}{}
            $
        \begin{array}{rrrrrrrr|rrrrrrrrrrrrrrrrrrrrrrrrrrrrrrrrrrrrrrrrrrrrrrrrrrrrrrrrrrrrrrrrrrrrrrrrrrrrrrrrrrrrrrrrrrrrrrrrrrrrrrrrrrrrrrrrrrrrrrrr}
            &  &  &  &  &  &  &  & 1 & 0 & 0 & 0 & 0 & 0 & 0 & 0 & 1 & 0 & 0 & 0 & 0 & 0 & 0 & 1 & 0 & 0 & 0 & 0 & 0 & 0 & 1 & 1 & 0 & 0 & 0 & 0 & 0 & 1 & 1 & 0 & 0 & 0 & 0 & 0 & 1 & 1 & 1 & 0 & 0 & 0 & 0 & 1 & 1 & 1 & 1 & 0 & 0 & 0 & 1 & 1 & 1 & 1 & 0 & 0 & 1 & 1 & 1 & 1 & 0 & 0 & 1 & 1 & 1 & 1 & 1 & 0 & 1 & 1 & 1 & 1 & 1 & 0 & 1 & 1 & 1 & 1 & 1 & 1 & 1 & 1 & 1 & 1 & 1 & 1 & 1 & 1 & 1 & 1 & 1 & 1 & 1 & 1 & 1 & 1 & 2 & 1 & 1 & 1 & 2 & 1 & 1 & 2 & 1 & 1 & 2 & 1 & 2 & 1 & 2 & 1 & 2 & 2 & 2 & 2 & 2 & 2 & 2 & 2 \\
            &  &  &  &  &  &  &  & 0 & 1 & 0 & 0 & 0 & 0 & 0 & 0 & 0 & 1 & 0 & 0 & 0 & 0 & 0 & 0 & 1 & 1 & 0 & 0 & 0 & 0 & 1 & 0 & 1 & 1 & 0 & 0 & 0 & 1 & 0 & 1 & 1 & 1 & 0 & 0 & 1 & 1 & 0 & 1 & 1 & 1 & 0 & 1 & 1 & 1 & 0 & 1 & 1 & 1 & 1 & 1 & 1 & 1 & 1 & 1 & 1 & 1 & 1 & 1 & 1 & 1 & 1 & 1 & 1 & 1 & 1 & 1 & 2 & 1 & 1 & 1 & 1 & 1 & 2 & 1 & 1 & 1 & 1 & 2 & 2 & 1 & 1 & 1 & 2 & 2 & 1 & 1 & 2 & 2 & 2 & 1 & 2 & 2 & 2 & 1 & 2 & 2 & 2 & 2 & 2 & 2 & 2 & 2 & 2 & 2 & 2 & 2 & 2 & 2 & 2 & 3 & 3 & 2 & 3 & 3 & 3 & 3 & 3 & 3 \\
            &  &  &  &  &  &  &  & 0 & 0 & 1 & 0 & 0 & 0 & 0 & 0 & 1 & 0 & 1 & 0 & 0 & 0 & 0 & 1 & 1 & 0 & 1 & 0 & 0 & 0 & 1 & 1 & 1 & 0 & 1 & 0 & 0 & 1 & 1 & 1 & 1 & 0 & 1 & 0 & 1 & 1 & 1 & 1 & 1 & 0 & 1 & 2 & 1 & 1 & 1 & 1 & 1 & 1 & 2 & 1 & 1 & 1 & 1 & 1 & 2 & 2 & 1 & 1 & 1 & 1 & 2 & 2 & 2 & 1 & 1 & 1 & 2 & 2 & 2 & 2 & 1 & 1 & 2 & 2 & 2 & 2 & 1 & 2 & 2 & 2 & 2 & 2 & 2 & 2 & 2 & 2 & 2 & 2 & 2 & 2 & 3 & 2 & 2 & 2 & 3 & 3 & 2 & 2 & 3 & 3 & 2 & 3 & 3 & 2 & 3 & 3 & 3 & 3 & 3 & 3 & 3 & 4 & 4 & 4 & 4 & 4 & 4 & 4 \\
\end{array}
        $
    \end{adjustbox}
    \end{landscape}
    \clearpage% Flush page
}

\fi

\rowcolors{1}{}{}

\section{Discussion and further research directions}
\label{sec:further research directions}

In hyperbolic geometry we found an intertwiner analogous to (\ref{eq:modautlie}) (cf.~equation (23) in \cite{10.1093/imrn/rnab376}). This 
suggests a duality between
$\frac{X}{Y}$, $-\frac{Y\partial_X P}{\deg(P)P}$ and $\frac{1}{Y}$ on the one hand, and $\tau$, $2\pi i\frac{E_2(\tau)}{12}$ and $\rd \tau^{\frac{1}{2}}$ on the other, respectively. Here $\tau$ is an element of the upper half complex plane and $E_2$ the Eisenstein series. This duality is mysterious to the author, but also fascinating due to the tremendous importance of $E_2$ in the theory of modular forms.

The literature on automorphic Lie algebras has shown that many of these algebras can be described by symmetric $2$-cocycles on a root system with two distinct values. The current paper adds infinitely more such examples. This makes the space $Z^2_+(\roots,\{0,1\})$ of such cocycles an interesting subject of study, were we can answer classification and integration questions, investigate the associated graphs, Lie algebra structures and representations.

Currently, we still lack the ability to prove that two automorphic Lie algebras are not isomorphic over $\mb C$, unless they have nonisomorphic base Lie algebras or distinct abelianisations (cf.~Corollary\ref{cor:abelianisation}), and this prevents us to achieve a complete classification of isomorphism classes. To the best of our knowledge, this is still an open problem for twisted current algebras $\mf g[t]^{\cg{n}}$. 
One approach advised by Erhard Neher is to consider the centroids of the Lie algebras \cite{benkart2006the} and split the isomorphism group into Lie algebra automorphisms over the centroid, and the automorphisms of the centroid over $\mb C$. On can prove that the centroid of an automorphic Lie algebra is the ring of automorphic functions, which makes it plausible that this challenging question can be solved.

Theorem \ref{thm:aliam} provides the structure constants for automorphic Lie algebras of arbitrary simple Lie type. We would like to present this structure in different ways to gain further understanding. One way is through the graphs on root systems depicting the $2$-cocycles as was done in Section \ref{sec:graphs}, but beyond rank $2$ this becomes too complicated to be practical (the PhD thesis \cite{knibbeler2014invariants} of the author contains some such graphs for $A_3$, which are arguable already too complicated to be usable). In \cite{knibbeler2017higher} a model in one matrix was introduced to tackle exactly this problem (see also Section \ref{sec:integrals of 2-forms in normal form}). However, this model only works for Lie type $A_N$ and we do not see a useful generalisation to the other Lie types.
What we would really like is a method to encode the Lie structure of Theorem \ref{thm:aliam} in a Dynkin diagram with a small amount of additional data, but the Lie structure has resisted any attempt to do this so far.

In \cite{knibbeler2019hereditary} we show that any hereditary automorphic Lie algebra has reduction group consisting of inner automorphisms. Now we have Theorem \ref{thm:aliam} to see that any automorphic Lie algebra with factorisable action is hereditary. The notions of inner and factorisable actions are much easier to work with than that of hereditary automorphic Lie algebras. How the latter sit in between these two classes is not yet clear, but it would be interesting to know. There exist inner polyhedral group actions on simple Lie algebras which do not factor through $\SLNC[2]$ (an example is given in Proposition \ref{prop:one inner regular action not factorisable}) but it is not clear how common this is.

\appendix

\section{Case by case analysis}
\label{sec:case by case analysis}
This section provides lists that allow us to construct automorphic Lie algebras of our choice, and also to check Lemma \ref{lem:characters}.
See the lecture notes \cite{dolgachev2009mckay} by Dolgachev, and the work of Lombardo and Sanders \cite{lombardo2010on} for similar lists.

The results in this paper are independent of basis choices. In this section we pick bases in order to provide matrices, but these bases do not have any particular significance.

We provide generators for the groups $\brg$ in $SU(2)$. We study its action on the Riemann sphere $\overline{\mb{C}}=\mb{C}\cup\{\infty\}$ by Moebius transformations, 
rather than the equivalent action on the sphere in $\mb{R}^3$, or the projective line, to keep the text short. 
This action is given by
\[
    \begin{pmatrix}a&b\\c&d\end{pmatrix}\cdot \lambda = \frac{a\lambda+b}{c\lambda+d}.
\]
The map $\mb{C}^2\setminus\{(0,0)\}\to\overline{\mb{C}}$ defined by 
\[
    (X,Y)\mapsto\lambda=
    \left\{
        \begin{array}{ll}
            X/Y&\text{ if }Y\ne 0\\
            \infty&\text{ if }Y=0
        \end{array}
    \right.
\]
commutes with the group actions.

Recall that 
$\poles_i$ is an orbit of $\rg$ with nontrivial stabiliser (in this section the orbit is in the Riemann sphere), $P_i$ is a homogeneous polynomial of $X$ and $Y$ which vanishes precisely at the lines corresponding to the elements of $S_i$, and $\chi_i$ is the character of $\brg$ describing the action of $\brg$ on $P_i$.  

We provide generators $a$ and $b$ of the noncyclic binary polyhedral groups fitting in the uniform description 
\[\brg=\langle a,b,c\,|\,a^{\nu_1}=b^{\nu_2}=c^{\nu_3}=abc\rangle.\]
The roots of unity will be denoted 
\[\zeta_n=e^{\frac{2\pi i}{n}}.\]
The character $\chi_i$ and other elements of $\hombgc$ will be given by their value(s) on the generator(s) $a$ (and $b$ respectively).

\subsection{The cyclic groups $\cg{n}$}
\[
    a=\begin{pmatrix}
        \zeta_{2n}&0\\
        0&\bar{\zeta}_{2n}
    \end{pmatrix}
\]

\begin{center}
    \begin{table}[ht!] 
    \caption{Ground forms of cyclic groups}
        \begin{center}
            \begin{tabular}{lll}
                    $S_i$ & $P_i$ & $\chi_i$ 
                    \\
                    \hline
                    \\[-2mm]
                    $\{0\}$ 
                    & $X$ 
                    & $\zeta_{2n}$
                    \\[3mm]
                    $\{\infty\}$ 
                    & $Y$ 
                    & $\bar{\zeta}_{2n}$
                    \\[3mm]
                    &
                    &$-1$
            \end{tabular}
        \end{center}
    \end{table}
\end{center}
\begin{align*}
    &\ch:\left(\zn{n}\right)^2\to\Hom(\cg{2n},\mb{C}^\ast)/\langle\chi\rangle,\quad 
    (\bar{i},\bar{j})\mapsto\{a\mapsto \zeta_{2n}^{i-j}, a\mapsto -\zeta_{2n}^{i-j}\}
\end{align*}

\subsection{The dihedral groups $\dg{n}$}
\[
    a=\begin{pmatrix}
        \zeta_{2n}&0\\
        0&\bar{\zeta}_{2n}
    \end{pmatrix},\quad 
    b= \begin{pmatrix}
        0&i\bar{\zeta}_{2n}\\
        i\zeta_{2n}&0
    \end{pmatrix}
\]

\begin{center}
    \begin{table}[ht!] 
    \caption{Ground forms of the dihedral groups}
        \begin{center}
            \begin{tabular}{lll}
                    $S_i$ & $P_i$ & $\chi_i$
                    \\
                    \hline
                    \\[-2mm]
                    $\{0,\infty\}$
                    & $XY$
                    & $(1,-1)$
                    \\[3mm]
                    $\{\zeta_{2n}, \zeta_{2n}^3, \zeta_{2n}^5,\ldots,\zeta_{2n}^{2n-1}\}$ 
                    & $X^n+Y^n$
                    & $(-1,-\zeta_4^n)$
                    \\[3mm]
                    $\{1, \zeta_{2n}^2, \zeta_{2n}^4,\ldots,\zeta_{2n}^{2n-2}\}$
                    & $X^n-Y^n$
                    & $(-1,\zeta_4^n)$
            \end{tabular}
        \end{center}
    \end{table}
\end{center}
To write down the map \[\ch:\zn{n}\times\zn{2}\times \zn{2} \to\Hom(\bd{n},\mb{C}^\ast)/\langle\chi\rangle\] we split into two cases. If $n=2m-1$ we have $\chi=(1,-1)$ and
\[
    (\bar{i},\bar{j},\bar{k})\mapsto\{((-1)^{j+k},\zeta_4^{j-k}), ((-1)^{j+k},-\zeta_4^{j-k}).\}
\]
If on the other hand $n=2m$ we find $\chi=(1,1)$ and
\[
    (\bar{i},\bar{j},\bar{k})\mapsto((-1)^{j+k},(-1)^{i+(m+1)j+mk}).
\]

\subsection{The tetrahedral group $\tg$}

\[
    a=\left(\begin{array}{rr}
        \frac{1}{2} \zeta_{12}^{3} + \frac{1}{2} & \frac{1}{2} \zeta_{12}^{3} + \frac{1}{2} \\
        \frac{1}{2} \zeta_{12}^{3} - \frac{1}{2} & -\frac{1}{2} \zeta_{12}^{3} + \frac{1}{2}
        \end{array}\right),\quad
    b=\left(\begin{array}{rr}
        \frac{1}{2} \zeta_{12}^{3} + \frac{1}{2} & \frac{1}{2} \zeta_{12}^{3} - \frac{1}{2} \\
        \frac{1}{2} \zeta_{12}^{3} + \frac{1}{2} & -\frac{1}{2} \zeta_{12}^{3} + \frac{1}{2}
        \end{array}\right)
\]

\begin{center}
    \begin{table}[ht!] 
        \caption{Ground forms of the tetrahedral group}    
        \begin{center}
            \begin{tabular}{lll}
                    $S_i$ & $P_i$ & $\chi_i$
                    \\
                    \hline
                    \\[-2mm]
                    $\{\pm\frac{1-i}{2}(1+\sqrt{3}), \pm\frac{1+i}{2}(1-\sqrt{3})\}$
                    & $2 i \, \sqrt{3} X^{2} Y^{2} + X^{4} + Y^{4}$
                    & $\left( \zeta_{3}, \zeta_3^2 \right)$
                    \\[3mm]
                    $\{\pm\frac{1-i}{2}(1-\sqrt{3}), \pm\frac{1+i}{2}(1+\sqrt{3})\}$
                    & $-2 i \, \sqrt{3} X^{2} Y^{2} + X^{4} + Y^{4}$ 
                    & $\left(\zeta_3^2, \zeta_3 \right)$
                    \\[3mm]
                    $\{0,\infty, \pm 1, \pm i\}$
                    & ${\left(X^{4} - Y^{4}\right)} X Y$
                    & $(1, 1)$
            \end{tabular}
        \end{center}
    \end{table}
\end{center}

\begin{align*}
    &\ch:\zn{3}\times\zn{3}\times \zn{2} \to\Hom(\bt,\mb{C}^\ast),\quad 
    (\bar{i},\bar{j},\bar{k})\mapsto(\zeta_3^{i-j},\zeta_3^{j-i})
\end{align*}

\subsection{The octahedral group $\og$}

\[
    a=\left(\begin{array}{rr}
        \zeta_{24}^{3} & 0 \\
        0 & -\zeta_{24}^{5} + \zeta_{24}
        \end{array}\right),\quad
    b=\left(\begin{array}{rr}
        \frac{1}{2} \zeta_{24}^{6} + \frac{1}{2} & \frac{1}{2} \zeta_{24}^{6} - \frac{1}{2} \\
        \frac{1}{2} \zeta_{24}^{6} + \frac{1}{2} & -\frac{1}{2} \zeta_{24}^{6} + \frac{1}{2}
        \end{array}\right)
\]
\begin{center}
    \begin{table}[ht!] 
        \caption{Ground forms of the octahedral group}
        \begin{center}
            \begin{tabular}{lll}
                    $S_i$ & $P_i$ & $\chi_i$ 
                    \\
                    \hline
                    \\[-2mm]
                    $\{0,\infty, \pm 1, \pm i\}$
                    & ${\left(X^{4} - Y^{4}\right)} X Y$ 
                    & $\left(-1, 1\right)$
                    \\[3mm]
                    $\{\pm\frac{1\pm i}{2}(1\pm \sqrt{3})\}$
                    & $X^{8} + 14 \, X^{4} Y^{4} + Y^{8}$
                    & $\left(1, 1\right)$
                    \\[3mm]
                    $\{\zeta_{8}, \ldots\}$
                    & $-{\left(36 \, X^{4} Y^{4} - {\left(X^{4} + Y^{4}\right)}^{2}\right)} {\left(X^{4} + Y^{4}\right)}$
                    & $\left(-1, 1\right)$
            \end{tabular}
        \end{center}
    \end{table}
\end{center}

\begin{align*}
    &\ch:\zn{4}\times\zn{3}\times \zn{2} \to\Hom(\bo,\mb{C}^\ast),\quad 
    (\bar{i},\bar{j},\bar{k})\mapsto((-1)^{i+k},1)
\end{align*}

\subsection{The icosahedral group $\yg$}

\begin{align*}
    &a=\left(\begin{array}{rr}
        \zeta_{20}^{2} & 0 \\
        0 & -\zeta_{20}^{6} + \zeta_{20}^{4} - \zeta_{20}^{2} + 1
        \end{array}\right),\\
    &b=\left(\begin{array}{rr}
        \frac{2}{5} \zeta_{20}^{6} + \frac{1}{5} \zeta_{20}^{4} + \frac{1}{5} \zeta_{20}^{2} + \frac{2}{5} & -\frac{1}{5} \zeta_{20}^{6} + \frac{2}{5} \zeta_{20}^{4} + \frac{2}{5} \zeta_{20}^{2} - \frac{1}{5} \\
        \frac{4}{5} \zeta_{20}^{6} - \frac{3}{5} \zeta_{20}^{4} + \frac{2}{5} \zeta_{20}^{2} - \frac{1}{5} & -\frac{2}{5} \zeta_{20}^{6} - \frac{1}{5} \zeta_{20}^{4} - \frac{1}{5} \zeta_{20}^{2} + \frac{3}{5}
        \end{array}\right)
    \end{align*}

\begin{center}
    \begin{table}[ht] 
    \caption{Ground forms of the icosahedral group}
        \begin{center}
            \begin{tabular}{lll}
                    $S_i$
                    \\
                    \hline
                    \\[-2mm]
                    $\{0, \ldots\}$
                    \\[3mm]
                    $\left\{{\left(5 \, \sqrt{3} \sqrt{38 \, \sqrt{5} + 85} + 25 \, \sqrt{5} + 57\right)}^{\frac{1}{5}},\ldots\right\}$
                    \\[3mm]
                    $\{\zeta_{20},\ldots\}$
                    \\[3mm]
                    $P_i$  
                    \\
                    \hline
                    \\[-2mm]
                    $X^{11} Y - 11 \, X^{6} Y^{6} - X Y^{11}$
                    \\[3mm]
                    $-X^{20} + 228 \, X^{15} Y^{5} - 494 \, X^{10} Y^{10} - 228 \, X^{5} Y^{15} - Y^{20}$ 
                    \\[3mm]
                    $X^{30} + 522 \, X^{25} Y^{5} - 10005 \, X^{20} Y^{10} - 10005 \, X^{10} Y^{20} - 522 \, X^{5} Y^{25} + Y^{30}$
            \end{tabular}
        \end{center}
    \end{table}
\end{center}

The group $\Hom(\by,\mb{C}^\ast)$ only contains the trivial map $\gamma\mapsto 1$, hence the characters $\chi_i$, $\chi$ and $\ch$ are all trivial.

\begin{center}
    \begin{table}[ht!] 
    \caption{Values $\ex(k)$ of the periodic map $\ex:2\mb{Z}\to \prod\{0,\ldots,\nu_i-1\}$ defined in Section \ref{sec:characters and degrees}}
    \label{tab:exponents}
        \begin{center}
            % \begin{adjustbox}{width=\linewidth}
                \begin{tabular}{l|cccccc}
                    $k$ & $\cg{n=2m-1}$ & $\cg{n=2m}$ & $\dg{n}$ & $\tg$ & $\og$ & $\yg$
                    \\[2mm]
                    & $(n,n)$ & $(n,n)$ & $(n,2,2)$ & $(3,3,2)$ & $(4,3,2)$ & $(5,3,2)$
                    \\[2mm]
                    \hline
                    30 & $ $ & $ $ & $ $ & $ $ & $ $ & $(0,0,1)$ 
                    \\
                    28 & $ $ & $ $ & $ $ & $ $ & $ $ & $(4,2,0)$ 
                    \\
                    26 & $ $ & $ $ & $ $ & $ $ & $ $ & $(3,1,1)$ 
                    \\
                    \hline
                    24 & $ $ & $ $ & $ $ & $ $ & $ $ & $(2,0,0)$ 
                    \\
                    22 & $ $ & $ $ & $ $ & $ $ & $ $ & $(1,2,1)$ 
                    \\
                    20 & $ $ & $ $ & $ $ & $ $ & $ $ & $(0,1,0)$ 
                    \\
                    \hline
                    18 & $ $ & $ $ & $ $ & $ $ & $ $ & $(4,0,1)$ 
                    \\
                    16 & $ $ & $ $ & $ $ & $ $ & $ $ & $(3,2,0)$ 
                    \\
                    14 & $ $ & $ $ & $ $ & $ $ & $ $ & $(2,1,1)$ 
                    \\
                    \hline
                    12 & $ $ & $ $ & $ $ & $ $ & $(2,0,0)$ & $(1,0,0)$ 
                    \\
                    10 & $ $ & $ $ & $ $ & $ $ & $(1,2,1)$ & $(0,2,1)$ 
                    \\
                    8 & $ $ & $ $ & $ $ & $ $ & $(0,1,0)$ & $(4,1,0)$ 
                    \\
                    \hline
                    \\[-6.5mm]
                    6 & $\vdots$ & $\vdots$ & $\vdots$ & $(0,0,1)$ & $(3,0,1)$ & $(3,0,1)$ 
                    \\
                    4 & $(2,2)$ & $(2,2)$ & $(2,0,0)$ & $(2,2,0)$ & $(2,2,0)$ & $(2,2,0)$ 
                    \\
                    2 & $(1,1)$ & $(1,1)$ & $(1,1,1)$ & $(1,1,1)$ & $(1,1,1)$ & $(1,1,1)$ 
                    \\
                    0 & $(0,0)$ & $(0,0)$ & $(0,0,0)$ & $(0,0,0)$ & $(0,0,0)$ & $(0,0,0)$ 
                    \\
                    -2 & $(n-1,n-1)$ & $(m-1,m-1)$ & $(n-1,1,1)$ & $(2,2,1)$ & $(3,2,1)$ & $(4,2,1)$ 
                    \\
                    -4 & $(n-2,n-2)$ & $(m-2,m-2)$ & $(n-2,0,0)$ & $(1,1,0)$ & $(2,1,0)$ & $(3,1,0)$ 
                    \\[-2mm]
                    -6 & $\vdots$ & $\vdots$ & $\vdots$ & $(0,0,1)$ & $(1,0,1)$ & $(2,0,1)$ 
                    \\
                    \hline
                    -8 & $ $ & $ $ & $ $ & $ $ & $(0,2,0)$ & $(1,2,0)$ 
                    \\
                    -10 & $ $ & $ $ & $ $ & $ $ & $(3,1,1)$ & $(0,1,1)$ 
                    \\
                    -12 & $ $ & $ $ & $ $ & $ $ & $(2,0,0)$ & $(4,0,0)$ 
                    \\
                    \hline
                    -14 & $ $ & $ $ & $ $ & $ $ & $ $ & $(3,2,1)$ 
                    \\
                    -16 & $ $ & $ $ & $ $ & $ $ & $ $ & $(2,1,0)$ 
                    \\
                    -18 & $ $ & $ $ & $ $ & $ $ & $ $ & $(1,0,1)$ 
                    \\
                    \hline
                    -20 & $ $ & $ $ & $ $ & $ $ & $ $ & $(0,2,0)$ 
                    \\
                    -22 & $ $ & $ $ & $ $ & $ $ & $ $ & $(4,1,1)$ 
                    \\
                    -24 & $ $ & $ $ & $ $ & $ $ & $ $ & $(3,0,0)$ 
                    \\
                    \hline
                    -26 & $ $ & $ $ & $ $ & $ $ & $ $ & $(2,2,1)$ 
                    \\
                    -28 & $ $ & $ $ & $ $ & $ $ & $ $ & $(1,1,0)$ 
                    \\
                    -30 & $ $ & $ $ & $ $ & $ $ & $ $ & $(0,0,1)$ 
                    \\
                    \hline
                \end{tabular}
            % \end{adjustbox}
        \end{center}
    \end{table}
\end{center}

\section{Kostant}
\label{sec:kostant}
Kostant expresses the generating functions of isotypical components in $\mb{C}[X,Y]$
in terms of Lie algebras associated to the binary polyhedral groups through the McKay correspondence \cite{kostant1984on}. 
\begin{center}
\begin{table}[ht!] 
\caption{Kostant's table}
\begin{center}
\begin{tabular}{lllllllll}
$\rg$&$X_N$&$a$&$b$&$h$\\
\hline
$\cg{n}$&$A_{2n-1}$&$2$&$2n$&$2n$\\
$\dg{n}$&$D_{n+2}$&$4$&$2n$&$2n+2$\\
$\tg$&$E_{6}$&$6$&$8$&$12$\\
$\og$&$E_{7}$&$8$&$12$&$18$\\
$\yg$&$E_{8}$&$12$&$20$&$30$
\end{tabular}
\end{center}
\end{table}
\end{center}
In particular, the generating function of the algebra of invariants is expressed in terms of the Coxeter number $h$ and the coefficient $d$ in the highest root corresponding to the middle root in the Dynkin diagram of $A_{2n-1}$ or the branch point of the Dynkin diagram of type $D$ or $E$.
\[Q_\triv(t)=\frac{1+t^h}{(1-t^a)(1-t^b)}\]
\[a=2d,\quad b=h+2-a,\quad |\rg|=\frac{ab}{4}\]
The groups with Schur multiplier of order $2$ are precisely the groups where $P_i^{\nu_i}$ is invariant (cf. Lemma \ref{lem:characters}). For these groups, the generating function $Q_\triv(t)$ can be rewritten with denominator $(1-t^{|\rg|})^2$ and numerator
\[N_\triv(t)=\frac{(1+t^h)\left(\sum_{j=0}^{b/2-1}t^{ja}\right)\left(\sum_{j=0}^{a/2-1}t^{jb}\right)}{(1+t^{|\rg|})^2}\]
in terms of the parameters of the Lie algebra (cf.~Table \ref{tab:numerators of generating functions}).
\begin{center}
    \begin{table}[ht!] 
    \caption{Numerators of generating functions}
    \label{tab:numerators of generating functions}
        \begin{center}
            \begin{tabular}{ll}
                $\rg$&$N_\triv(t)$\\\hline\\[-2mm]
                $\dg{n=2m}$&$(1+t^{2n+2})(1+t^4+t^8+\ldots+t^{2n-4})$\\[2mm]
                $\tg$&$t^{22} + t^{16} + t^{14} + t^{8} + t^{6} + 1$\\[2mm]
                $\og$&$t^{46} + t^{38} + t^{34} + t^{30} + t^{28} + t^{26} + t^{20} + t^{18} + t^{16} + t^{12} + t^{8} + 1$\\[2mm]
                $\yg$&$t^{118} + t^{106} + t^{98} + t^{94} + t^{88} + t^{86} + t^{82} + t^{78} + t^{76} + t^{74}$ 
                \\&$+ t^{70} + t^{68} + t^{66} + t^{64} + t^{62} + t^{56} + t^{54} + t^{52} + t^{50} + t^{48}$
                \\&$+t^{44} + t^{42} + t^{40} + t^{36} + t^{32} + t^{30} + t^{24} + t^{20} + t^{12} + 1$
            \end{tabular}
        \end{center}
    \end{table}
\end{center}
We compare this with Section \ref{sec:characters and degrees}. 
Recall that we were forced to specify to the groups which satisfy any of the equivalent facts 
\[d=2,\quad \chi=1, \quad\chi_i^{\nu_i}=1,\quad \Hom(\brg,\mb{C}^\ast)\cong\Hom(\rg,\mb{C}^\ast).\]
The set $E$ of exponents of $N_\triv(t)$ mapped canonically to $\bar{E}\subset\mb{Z}/|\rg|\mb{Z}$  is contained in the image of 
\[N_\ch/(N_\deg\cap N_\ch)\xrightarrow{\odeg} d\mb{Z}/|\rg|\mb{Z}.\]
Indeed, $\bar{E}$ is contained in $\odeg(N_\ch)$ since we only consider the generating function for the trivial character (invariants). We can divide out $N_\deg\cap N_\ch$ since we consider the numerator relative to denominator  $(1-t^{|\rg|})^2$. 
For our specific groups $N_\deg\cap N_\ch$ is trivial, the above map is an isomorphism
and $\odeg(N_\ch)$ is contained in $\bar{E}$ since $\chi=1$.
We get a commutative diagram of bijections
\[\begin{tikzcd}
& E \arrow[d] \\
N_\ch \arrow[r, "\odeg"']\arrow[ru, "\deg"]
& 2\mb{Z}/|\rg|\mb{Z}
\end{tikzcd}
\]
where $\deg(\bar{n}_i)=\sum d_i\res(\bar{n}_i)$. Lemma \ref{lem:generated by 1} states that the group $N_\ch$ is generated by $(\bar{1},\bar{1},\bar{1})$. This gives a construction of $E=\deg(N_\ch)$ using only the group parameters $\nu_i$.
\begin{Example}[Tetrahedral group] If we take group parameters $(\nu_1, \nu_2, \nu_3)=(3,3,2)$ we obtain $|\rg|=12$ and $(d_1, d_2, d_3)=(4,4,6)$. The kernel $N_\ch$ is $$N_\ch=\langle(\bar{1},\bar{1},\bar{1})\rangle=\{(\bar{0},\bar{0},\bar{0}),(\bar{1},\bar{1},\bar{1}),(\bar{2},\bar{2},\bar{0}),(\bar{0},\bar{0},\bar{1}),(\bar{1},\bar{1},\bar{0}),(\bar{2},\bar{2},\bar{1})\}.$$ The degree map sends this to $\{0,14,16,6,8,22\}$. We recognise the exponents of the numerator $N_\triv(t)=t^{22} + t^{16} + t^{14} + t^{8} + t^{6} + 1$ which was obtained from the Lie algebra parameters $(d,h)=(3,12)$.
\end{Example}

\section*{Acknowledgements}
We are grateful to Sara Lombardo for her continuous support and belief in our projects. We thank Casper Oelen for the excellent suggestion to focus on factorisable actions, and Alexander Veselov for asking the questions that started the research for this paper. We are indebted to Jan A.~Sanders for many stimulating discussions and suggestions, to which we owe most of Section \ref{sec:examples}. Finally, we thank Erhard Neher for an enlightening correspondence on centroids and isomorphism classes of automorphic Lie algebras.

\printbibliography

\end{document}